\documentclass[12pt]{amsart}

\usepackage[letterpaper,margin=1.1in]{geometry}
\usepackage{amssymb} % For \lesssim
\usepackage{graphicx}
\usepackage{xcolor}

\newtheorem{maintheorem}{Theorem}
\newtheorem{maincor}[maintheorem]{Corollary}

\newtheorem{theorem}{Theorem}[section]
\newtheorem{lemma}[theorem]{Lemma}
\newtheorem{proposition}[theorem]{Proposition}
\newtheorem{corollary}[theorem]{Corollary}

\theoremstyle{definition}
\newtheorem{definition}[theorem]{Definition}

\theoremstyle{remark}
\newtheorem{remark}[theorem]{Remark}

\newtheorem*{remark-no-number}{Remark}

\newcommand{\RR}{\mathbb{R}}
\newcommand{\ZZ}{\mathbb{Z}}
\newcommand{\NN}{\mathbb{N}}
\newcommand{\cT}{\mathcal{T}}
\newcommand{\Haus}{\mathcal{H}}
\newcommand{\two}{{I\!I}}
\newcommand{\three}{{I\!I\!I}}
\newcommand{\tJ}{\widetilde{J}}
\DeclareRobustCommand{\positivedensity}{{$\underline{D}^1(\mu,x)>0$ }}

\newcommand{\dist}{\mathop\mathrm{dist}\nolimits}
\newcommand{\diam}{\mathop\mathrm{diam}\nolimits}
\newcommand {\side}{\mathop\mathrm{side}\nolimits}
\newcommand{\lD}[1]{\mathop{\underline{D}^{#1}}\nolimits}
\newcommand{\uD}[1]{\mathop{\overline{D}^{\,#1}}\nolimits}
\newcommand{\spt}{\mathop\mathrm{spt}\nolimits}
\newcommand{\excess}{\mathop\mathrm{ex}\nolimits}
\newcommand{\HD}{\mathop\mathrm{HD}\nolimits}
\newcommand{\res}{\hbox{ {\vrule height .22cm}{\leaders\hrule\hskip.2cm} }} % Measure Restriction Symbol

\newcommand{\double}{\mathop\mathsf{Double}\nolimits}
\newcommand{\bridges}{\mathsf{Bridges}}
\newcommand{\edges}{\mathsf{Edges}}
\newcommand{\cores}{{\mathsf{Cores}_{\two}}}
\newcommand{\phan}{\mathsf{Phantom}}
\newcommand{\leaves}{\mathsf{Leaves}}
\newcommand{\Top}{\mathop\mathsf{Top}}
\newcommand{\bad}{\mathsf{Bad}}
\newcommand{\good}{\mathsf{Good}}
\newcommand{\rect}{{rect}}
\newcommand{\pu}{{pu}}

\numberwithin{figure}{section}
\numberwithin{equation}{section}

\begin{document}

\title[Multiscale analysis of 1-rectifiable measures II]{Multiscale analysis of 1-rectifiable measures II: characterizations}
\thanks{M.~ Badger was partially supported by NSF DMS 1500382.
R.~ Schul was partially supported by NSF  DMS 1361473.}
\date{January 10, 2017}
\subjclass[2010]{Primary 28A75. Secondary 26A16, 42B99, 54F50}
\keywords{1-rectifiable measures, purely 1-unrectifiable measures, rectifiable curves, Jones beta numbers, Jones square functions, Analyst's traveling salesman theorem, doubling measures, Hausdorff densities, Hausdorff measures.}
\author{Matthew Badger \and Raanan Schul}
\address{Department of Mathematics\\ University of Connecticut\\ Storrs, CT 06269-3009}
\email{matthew.badger@uconn.edu}
\address{Department of Mathematics\\ Stony Brook University\\ Stony Brook, NY 11794-3651}
\email{schul@math.sunysb.edu}

\begin{abstract} A measure is 1-rectifiable if there is a countable union of finite length curves whose complement has zero measure. We characterize 1-rectifiable Radon measures $\mu$ in $n$-dimensional Euclidean space for all $n\geq 2$ in terms of positivity of the lower density and finiteness of a geometric square function, which loosely speaking, records in an $L^2$ gauge the extent to which $\mu$ admits approximate tangent lines, or has rapidly growing density ratios, along its support. In contrast with the classical theorems of Besicovitch, Morse and Randolph, and Moore, we do not assume an \emph{a priori} relationship between $\mu$ and 1-dimensional Hausdorff measure $\Haus^1$. We also characterize purely 1-unrectifiable Radon measures, i.e.~ locally finite measures that give measure zero to every finite length curve.
Characterizations of this form were originally conjectured to exist by  P.~Jones.
Along the way, we develop an $L^2$ variant of P.~Jones'  traveling salesman construction, which is of independent interest.
\end{abstract}

\maketitle

\setcounter{tocdepth}{1}
\tableofcontents

\section{Introduction}\label{s:intro}

A fundamental concept in geometric measure theory is a general notion of rectifiability of a set or a measure, which generalizes the classical notion of rectifiability of a curve. For sets, this notion of rectifiability is due to A.S. Besicovitch \cite{Bes28}. For measures, this notion of rectifiability is due to A.P. Morse and J. Randolph \cite{MR} and H. Federer \cite{Fed47,Federer} (see Definition \ref{def:rect} below). Rectifiability has been extensively studied for sets and also for measures $\mu$ that satisfy an additional regularity assumption, which is often expressed in terms of finiteness $\mu$-almost everywhere of the upper Hausdorff density $\uD{m}(\mu,\cdot)$ of the measure. This assumption is equivalent to an \emph{a priori} relationship between the null sets of the measure $\mu$ and  null sets of the $m$-dimensional Hausdorff measure $\Haus^m$, more specifically that $\mu$ vanishes on every set of $\Haus^m$ measure zero. One reason that this regularity assumption is often imposed it that it allows one to replace the class of Lipschitz images of bounded subsets of $\RR^m$ appearing in Federer's definition of rectifiability of a measure with bi-Lipschitz images or Lipschitz graphs or $C^1$ graphs without changing the class of rectifiable measures. For arbitrary (or doubling) Radon measures, however, it is known by an example of Garnett, Killip, and Schul \cite{GKS} that the class of measures that are rectifiable with respect to Lipschitz images is strictly larger than the class of measures that are rectifiable with respect to bi-Lipschitz images. While (countable) rectifiability of a set or a measure is an inherently qualitative property, a quantitative counterpart of the theory of rectifiability was developed in the early 1990s by P. Jones \cite{Jones-TSP} and by G. David and S. Semmes \cite{DS91,DS93}. One goal of theses investigations was to study the connection between rectifiability and singular integral operators. In David and Semmes' theory of uniformly rectifiable sets and measures, it is essential that the measures involved are Ahlfors regular, a strong form of regularity of a measure.

The work in this paper addresses studying Federer's definition of rectifiability without imposing the standing regularity hypotheses of past investigations.  We repurpose tools from the Jones-David-Semmes theory of quantitative rectifiability to characterize Federer 1-rectifiable measures using snapshots of a measure (beta numbers) at multiple scales. Moreover, we identify the 1-rectifiable and purely 1-unrectifiable parts of an arbitrary Radon measure $\mu$ in $\RR^n$ in terms of the pointwise behavior of the lower Hausdorff density $\lD1(\mu,x)$ and a weighted geometric square function $J^*_p(\mu,x)$, which records in an $L^p$ gauge the extent to which $\mu$ admits approximate tangent lines or has rapidly growing density ratios along its support. For the precise statement of these main results, see \S2. The central reason we restrict ourselves to $m$-rectifiability with $m=1$ is the special role that connectedness has for one-dimensional sets: every closed, connected set in $\RR^n$ of finite $\Haus^1$ measure is a Lipschitz image of a closed interval. The current lack of a Lipschitz parameterization theorem for surfaces is the key obstruction to understanding Federer $m$-rectifiability when $m\geq 2$. The innovation of this paper is to provide the first full treatment of rectifiability of arbitrary Radon measures, including measures which have infinite Hausdorff density or which are mutually singular with respect to Hausdorff measure.

\subsection{Decompositions of Radon measures} A \emph{Radon measure} $\mu$ on $\RR^n$ is a Borel regular outer measure that is finite on compact subsets of $\RR^n$.

\begin{definition}[Rectifiable and purely unrectifiable measures] \label{def:rect}
Let $n\geq 1$ and $m\geq 0$ be integers.  We say that a Radon measure $\mu$ on $\RR^n$ is \emph{$m$-rectifiable} (or equivalently, in Federer's terminology, \emph{$\RR^n$ is countably $(\mu,m)$ rectifiable}) if there exist countably many Lipschitz maps $f_i:[0,1]^m\rightarrow\RR^n$ such that $$\mu\left(\RR^n\setminus \bigcup_i f_i([0,1]^m)\right)=0,$$ where we define $[0,1]^m:=\{0\}$ when $m=0$. At the other extreme, we say that a Radon measure $\mu$ on $\RR^n$ is \emph{purely $m$-unrectifiable} if $\mu(f([0,1]^m))=0$ for every Lipschitz map $f:[0,1]^m\rightarrow\RR^n$. See \cite[pp.~251--252]{Federer}.
\end{definition}

When $m\geq n$, every measure on $\RR^n$ is trivially $m$-rectifiable.

When $m=0$, every Radon measure $\mu$ on $\RR^n$ can be written uniquely as $\mu=\mu^0_{\rect}+\mu^0_{\pu}$, where $\mu^0_{\rect}$ is $0$-rectifiable and $\mu^0_{\pu}$ is purely $0$-unrectifiable. The decomposition is given by \begin{align}\mu^0_{\rect}&=\mu\res\{x\in\RR^n: \lim_{r\downarrow 0}\mu(B(x,r))>0\},\\
\mu^0_{\pu}&=\mu\res\{x\in\RR^n: \lim_{r\downarrow 0}\mu(B(x,r))=0\}.\end{align} Here and below $\mu\res E$ denotes the \emph{restriction of $\mu$ to $E$} defined by $\mu\res E(F)=\mu(E\cap F)$ for all $F\subseteq\RR^n$.  The $0$-rectifiable part $\mu^0_\rect$ of $\mu$ is a weighted sum of Dirac masses at the \emph{atoms} $\{x\in\RR^n:\lim_{r\downarrow 0}\mu(B(x,r))>0\}$ of $\mu$. The purely $0$-unrectifiable part $\mu^0_{\pu}$ of $\mu$ is the \emph{atomless} part of $\mu$.

\begin{proposition}\label{p:decomp} Let $1\leq m\leq n-1$. Every Radon measure $\mu$ on $\RR^n$ can be written uniquely as \begin{equation}\label{e:decomp}\mu=\mu^m_{\rect}+\mu^m_{\pu},\end{equation} where $\mu^m_{\rect}$ is $m$-rectifiable and $\mu^m_{\pu}$ is purely $m$-unrectifiable.\end{proposition}

\begin{proof} We present a simple variation of \cite[Theorem 15.6]{Mattila}, which is tailored to the setting of finite measures of the form $\mu=\Haus^m\res E$. Let $\mu$ be an arbitrary Radon measure on $\RR^n$. For each $j\geq 1$, let $M_j = \sup_f \mu(B(0,j)\cap f([0,1]^m))$, where the supremum runs over all Lipschitz maps $f:[0,1]^m\rightarrow\RR^n$. Because $\mu$ is Radon, $M_j<\infty$ for all $j\geq 1$. Choose any sequence of Lipschitz maps $f_j:[0,1]^m\rightarrow \RR^n$ such that $$\mu(B(0,j)\cap f_j([0,1]^m))>M_j-1/j.$$ Then define  $\mu^m_{\rect}:=\mu\res V$ and $\mu^{m}_{\pu}:= \mu\res \RR^n\setminus V$, where $V:=\bigcup_{j=1}^\infty f_j([0,1]^m)$. It is clear that $\mu=\mu^m_\rect+\mu^m_\pu$ and $\mu^m_{\rect}$ is $m$-rectifiable. To see that $\mu^m_{\pu}$ is purely $m$-unrectifiable, assume for contradiction that $\mu^m_{\pu}(g([0,1]^m))>0$ for some Lipschitz map $g:[0,1]^m\rightarrow\RR^n$. Pick $j\geq 1$ large enough such that $g([0,1]^m)\subseteq B(0,j)$ and $\mu^m_{\pu}(g([0,1]^m))>1/j$. Next, let $h:[0,1]^m\rightarrow\RR^n$ be any Lipschitz map such that $h([0,1]^m)\supseteq f_j([0,1]^m)\cup g([0,1]^m)$. (For example, one could simply take $h$ to be a Lipschitz extension of the map defined by $h(x)=f_j(3x_1,x_2,\dots,x_m)$ for $x_1\in[0,1/3]$ and by $h(x)=g(3x_1-2,x_2,\dots,x_m)$ for $x_1\in[2/3,1]$.) It follows that \begin{equation*}\begin{split} M_j &\geq \mu(B(0,j)\cap h([0,1]^m))\\ &\geq \mu^m_{\rect}(B(0,j)\cap f_j([0,1]^m))+\mu^m_{\pu}(B(0,j)\cap g([0,1]^m))>(M_j-1/j)+1/j=M_j, \end{split}\end{equation*} an apparent contradiction. Therefore, $\mu^m_\pu$ is purely $m$-unrectifiable. A similar argument shows that the decomposition $\mu=\mu^m_\rect+\mu^m_\pu$ is unique.
%\RS{
%Indeed,
%if $\mu^{m}_{\pu}(f([0,1]^m)>0$, then one gets a contradiction by considering (for sufficiently large $j$)  a Lipschitz map $\tilde f_j:[0,1]^m\to \RR^n$ which is an extension of the map
%$\tilde f_j(x_1,..,x_m)=f_j(3x_1,x_2,...,x_m)$ for $x_1\in [0,\frac13]$ and $\tilde f_j(x)=f_j(3x_1-2,x_2,...,x_m)$ for $x_1\in [\frac23, 1]$ given by linear interpolation.
% }
\end{proof}

The abstract proof of Proposition \ref{p:decomp} given above does not provide a concrete method for identifying $\mu^m_{\rect}$ and $\mu^m_{\pu}$ for a given Radon measure $\mu$. A fundamental problem in geometric measure theory is to provide (geometric, measure-theoretic) characterizations of $\mu^m_{\rect}$ and $\mu^m_{\pu}$. When $n=2$ and $m=1$, this problem was first formulated and investigated by Besicovitch \cite{Bes28,Bes38} for positive and finite measures $\mu$ of the form $\mu=\Haus^1\res E$, $E\subseteq\RR^2$, and later by Morse and Randolph \cite{MR} (resp.~ Moore \cite{Moore}) for Radon measures $\mu$ on $\RR^2$ (resp.~ $\RR^n$, $n\geq 3$) such that $\mu\ll\Haus^1$. Here and below $\Haus^m$ denotes the \emph{$m$-dimensional Hausdorff measure on $\RR^n$} (see e.g.~\cite{Federer} or \cite{Mattila}), normalized to agree with the Lebesgue measure in $\RR^m$. The condition $\mu\ll \Haus^1$, called \emph{absolute continuity}, means that $\mu(E)=0$ for every Borel set $E$ such that $\Haus^1(E)=0$. In this paper, we provide characterizations of the 1-rectifiable part $\mu^1_{\rect}$ and the purely 1-unrectifiable part $\mu^1_{\pu}$ of an arbitrary Radon measure $\mu$ on $\RR^n$ (see \S\ref{s:new}).  We emphasize that in contrast with previous works, our main result does not require an \emph{a priori} relationship between $\mu$ and the 1-dimensional Hausdorff measure $\Haus^1$. A remarkable feature of the proof of our characterization is that we adapt techniques originating from the theory of quantitative rectifiability to study the qualitative rectifiability of measures.
In fact, our identification of the rectifiable part of the measure $\mu$ is constructive in nature.

\subsection{Hausdorff densities and rectifiability} Let $1\leq m\leq n-1$. The \emph{lower} and \emph{upper $m$-dimensional Hausdorff densities} of a Radon measure $\mu$ at $x\in\RR^n$ are defined by \begin{equation} \lD{m}(\mu,x) = \liminf_{r\downarrow 0} \frac{\mu(B(x,r))}{\omega_m r^m}\in[0,\infty]\end{equation} and \begin{equation} \uD{m}(\mu,x) = \limsup_{r\downarrow 0} \frac{\mu(B(x,r))}{\omega_m r^m}\in[0,\infty],\end{equation} respectively, where $\omega_m$ is the $m$-dimensional Hausdorff measure of the unit ball in $\RR^m$. Whenever $\lD{m}(\mu,x)=\uD{m}(\mu,x)$, we may denote the common value by $D^m(\mu,x)$.

A Radon measure $\mu$ satisfies $\mu\ll \Haus^m$ if and only if $\uD{m}(\mu,x)<\infty$ at $\mu$-a.e.~$x\in\RR^n$. In addition, for Radon measures
which happen to be of the form $\mu=\Haus^m\res E$, the upper density satisfies $2^{-m}\leq \uD{m}(\mu,x)\leq 1$ at $\mu$-a.e.~$x\in\RR^n$. (For example, see Exercise 4 and Theorem 6.2 in \cite[Chapter 6]{Mattila}, respectively.) If $\mu$ is a Radon measure on $\RR^n$ and $\mu$ is $m$-rectifiable, then $\lD{m}(\mu,x)>0$ at $\mu$-a.e.~ $x\in\RR^n$ (see \cite[Lemma 2.7]{BS}).
Rectifiability of absolutely continuous measures is tied up with the existence of the density $D^m(\mu,x)$.

\begin{theorem}[Mattila \cite{Mattila75}]\label{t:M} Let $1\leq m\leq n-1$. Assume that $\mu=\Haus^m\res E$ is Radon for some $E\subseteq\RR^n$. Then $\mu$ is $m$-rectifiable if and only if the Hausdorff $m$-density of $\mu$ exists and $D^m(\mu,x)=1$ at $\mu$-a.e.~ $x\in\RR^n$.\end{theorem}

\begin{theorem}[Preiss \cite{Preiss}]\label{t:P} Let $1\leq m\leq n-1$ and let $\mu$ be a Radon measure on $\RR^n$. Then $\mu$ is $m$-rectifiable and $\mu\ll\Haus^m$ if and only if the Hausdorff $m$-density of $\mu$ exists and $0<D^m(\mu,x)<\infty$ at $\mu$-a.e.~ $x\in\RR^n$.\end{theorem}

The proofs of Theorems \ref{t:M} and \ref{t:P} in arbitrary dimensions $1\leq m\leq n-1$ have a distinguished place in the history of geometric measure theory.
When $n=2$ and $m=1$, Theorem \ref{t:M}  is originally due to Besicovitch \cite{Bes28} (also see \cite{Bes38}) and Theorem \ref{t:P} is originally due to Morse and Randolph \cite{MR}. An extension of the later result to $n\geq 3$ was given by Moore \cite{Moore}.
In the aforementioned works, Besicovitch also characterized rectifiability of $\Haus^1\res E$, $E\subseteq\RR^2$, by existence $\Haus^1$-a.e.~of approximate tangent lines to $E$; and also in terms of the $\Haus^1$ measure of orthogonal projections (of subsets) of $E$ onto lines. Characterizations of rectifiability of $\Haus^m\res E$, $E\subseteq\RR^n$, in terms of approximate tangents and projections were extended to all dimensions $1\leq m\leq n-1$ by Federer \cite{Fed47}. The next case of Theorem \ref{t:M} was settled by Marstrand \cite{Marstrand61}, who proved the density characterization of rectifiability when $n=3$ and $m=2$. A few years later, in \cite{Marstrand}, Marstrand proved that if there exists a Radon measure $\mu$ on $\RR^n$ and a real number $s>0$ such that $\lim_{r\downarrow 0} r^{-s}\mu(B(x,r))$ exists and is positive and finite on a set of positive $\mu$ measure, then $s$ is an integer. Mattila's proof  of the general case of Theorem \ref{t:M}, which is based on Marstrand's approach, was published in \cite{Mattila75} nearly 50 years after the pioneering paper by Besicovitch.

To prove the general case of Theorem \ref{t:P}, Preiss \cite{Preiss} had to give a careful analysis of the geometry of \emph{$m$-uniform measures} $\mu$ on $\RR^n$, i.e.~ Radon measures with the property that $\mu(B(x,r))=\omega_m\,r^m$ for all $r>0$, for all $x\in\RR^n$ such that $\mu(B(x,r))>0$ for all $r>0$. Although it is obvious that the restriction $\Haus^m\res L$ of $m$-dimensional Hausdorff measure to an $m$-plane $L$ in $\RR^n$ is an example of an $m$-uniform measure, it is less clear if these are the only examples. Surprisingly, starting with $m\geq 3$, there exist ``non-flat" uniform measures such as $\Haus^3\res\{(x,t)\in\RR^3\times \RR:|x|^2=t^2\}$; see \cite{Preiss} (also \cite{KP}). Preiss introduced several original ideas to study the geometry of non-flat uniform measures, including the notion of a \emph{tangent measure} to a Radon measure. For an in-depth introduction, we refer the reader to the exposition of Preiss' theorem by De Lellis \cite{DeLellis}. The classification of $m$-uniform measures in $\RR^n$ is as of yet incomplete, but some progress has recently been made by Tolsa \cite{Tolsa-uniform} and Nimer \cite{nimer,nimer3}.

The deep connections between the existence of densities and rectifiability of sets and absolutely continuous measures in Euclidean space, as described above, have been explored in metric spaces beyond $\RR^n$ by several authors; see e.g.~\cite{PT}, \cite{Kirchheim}, \cite{Lorent03}, \cite{Lorent04}, \cite{CT15}. For perspectives on rectifiability in metric spaces related to existence of tangents or projection properties, see e.g.~\cite{Ambrosio-Kirchheim}, \cite{Bate}, \cite{Hajlasz}, \cite{Bate-Li}.

The \emph{support} $\spt \mu$ of a Borel measure $\mu$ on a metric space $X$ is the set of all $x\in X$ such that $\mu(B(x,r))>0$ for all $r>0$. In \cite{AM15}, Azzam and Mourgoglou proved that positive lower density is a sufficient condition for
a
%an arbitrary
locally finite Borel measure to be 1-rectifiable under additional global assumptions on the measure and its support. A locally finite Borel measure $\mu$ on $X$ is \emph{doubling} if there is a constant $C>1$ such that $\mu(B(x,2r))\leq C\mu(B(x,r))$ for all $x\in\spt\mu$ and for all $r>0$.

\begin{theorem}[Azzam and Mourgoglou {\cite{AM15}}] \label{t:AM} Assume that $\mu$ is a doubling measure whose support is a connected metric space, and let $E\subseteq\spt\mu$ be compact. Then $\mu\res E$ is 1-rectifiable if and only if $\lD1(\mu,x)>0$ for $\mu$-a.e.~$x\in E$.
\end{theorem}

It is important to emphasize that in Theorem \ref{t:AM}, no assumption is made on the upper density of the measure. Thus, Theorem \ref{t:AM} characterizes a class of 1-rectifiable measures, which includes measures that are not absolutely continuous with respect to $\Haus^1$. Examples of 1-rectifiable doubling measures $\mu$ on $\RR^n$ with $\spt\mu=\RR^n$ and $\mu\perp\Haus^1$, which satisfy the hypothesis of Theorem \ref{t:AM}, were constructed by Garnett, Killip, and Schul \cite{GKS}. Interestingly, such examples give measure zero to every bi-Lipschitz image of $\RR^m$ in $\RR^n$ (in particular, they give measure zero to every Lipschitz graph), but nevertheless give full measure to a countable family of Lipschitz images of $\RR$. Following \cite{BS}, it is also known that such measures have  density $D^1(\mu,x)=\infty$ $\mu$-a.e., with $\mu(B(x,r))/r\rightarrow\infty$ at a rapid rate as $r \rightarrow 0$.

\subsection{$L^p$ Jones beta numbers and rectifiability}

Let $1\leq p<\infty$ and let $\mu$ be a Radon measure on $\RR^n$. For every bounded Borel set $E\subseteq\RR^n$ of positive diameter (typically, either a ball $B(x,r)$ or a cube $Q$) and straight line $\ell$ in $\RR^n$, we define $\beta_{p}(\mu,E,\ell)\in[0,1]$ by
\begin{equation}\label{e:beta-p-def}
\beta_{p}(\mu,E,\ell)^p = \int_E \left(\frac{\dist(x,\ell)}{\diam E}\right)^p\frac{d\mu(x)}{\mu(E)},\end{equation}
where we interpret $\beta_{p}(\mu,E,\ell)=0$ if $\mu(E)=0$. We then define $\beta_{p}(\mu,E)\in[0,1]$ by \begin{equation}\beta_{p}(\mu,E)= \inf_\ell \beta_{p}(\mu,E,\ell),\end{equation} where the infimum runs over all straight lines in $\RR^n$. Note that $\beta_p(\mu,E,\ell)$ and $\beta_p(\mu,E)$ are increasing in the exponent $p$ for all $\mu$, $E$, and $\ell$. Higher dimensional beta numbers $\beta_p^{(m)}(\mu,E)$ may be defined by letting $\ell$ range over all $m$-dimensional affine planes in $\RR^n$ instead of over all lines in $\RR^n$.

In \cite{Jones-TSP}, Peter Jones characterized subsets of finite length curves in the plane in terms of
an $\ell^2$ sum of a $\sup$-norm variant of \eqref{e:beta-p-def}; see  Definition \ref{d:beta-infty} and  Theorem \ref{t:tst} below.
One goal was to bring these quantities into the study of singular integrals
\cite{Jones-Sq-fn-Cauchy-int-etc}. Guy David and Stephen Semmes did this, and more, for curves and surfaces in their investigation of \emph{uniformly rectifiable} sets and measures (e.g.~see \cite{DS91} and \cite{DS93}).  Much work has been done in connecting beta numbers and rectifiability (highlights include \cite{Pajot97}, \cite{Leger}, \cite{Lerman}, \cite{H08}, \cite{LW-I}, \cite{DT}, \cite{BS}, \cite{ADT1}, \cite{BS2}, \cite{NV}), but we single out two recent papers, \cite{Tolsa-n} and \cite{AT}, and state one theorem, which is a combination of their main results.

\begin{theorem}[{Tolsa \cite{Tolsa-n}, {Azzam and Tolsa \cite{AT}}}]
Let $1\leq m\leq n-1$ and let $\mu$ be a Radon measure on $\RR^n$. Then $\mu$ is $m$-rectifiable and $\mu\ll\Haus^m$ if and only if $0<\uD{m}(\mu,x)<\infty$ and
\begin{equation}\label{e:tolsa-beta}\int_0^1 \beta^{(m)}_2(\mu,B(x,r))^2\,\frac{\mu(B(x,r))}{r^m}\,\frac{dr}{r}<\infty\quad\text{at $\mu$-a.e.~ $x\in\RR^n$}.\end{equation}
\end{theorem}

The factor $\mu(B(x,r))/r^m$ appearing in \eqref{e:tolsa-beta} translates between different conventions in the definition of beta numbers. In \cite{Tolsa-n} and \cite{AT}, the authors use the convention that integration in the definition of $\beta^{(m)}_2(\mu,B(x,r))$ is against the measure $r^{-m}\mu$, whereas our convention is that integration is against the measure $\mu(B(x,r))^{-1}\,\mu$. Note that
\begin{equation}\label{e:beta-m}
\beta^{(m)}_2(\mu,B(x,r))^2\,\frac{\mu(B(x,r))}{r^m}=
\inf_{\ell} \int_{B(x,r)} \left(\frac{\dist(x,\ell)}{\diam (B(x,r))}\right)^2\frac{d\mu(x)}{r^m}.
\end{equation}

A good way to access a more comprehensive survey about uniform rectifiability, its connection to singular integrals, and its connection to analytic capacity, is to read David and Semmes \cite{DS93}, Pajot \cite{Pajot}, and Tolsa \cite{Tolsa-book}. We also make some remarks in \cite[\S4]{BS}.

\subsection{Conventions} We may write $a\lesssim b$ (or $b\gtrsim a$) to denote that $a\leq Cb$ for some absolute constant $0<C<\infty$ and write $a\sim b$ if $a\lesssim b$ and $b\lesssim a$. Likewise we may write $a\lesssim_{t} b$ (or $b\gtrsim_{t} a$) to denote that $a\leq C b$ for some constant $0<C<\infty$ that may depend on a list of parameters $t$ and write $a\sim_t b$ if $a\lesssim_t b$ and $b\lesssim_t a$.

Below we use several grids of dyadic cubes. Unless stated otherwise, we take all dyadic cubes in $\RR^n$ to be half open, say of the form $$Q=\left[\frac{j_1}{2^k},\frac{j_1+1}{2^k}\right)\times \dots \times \left[\frac{j_n}{2^k},\frac{j_n+1}{2^k}\right),\quad k,j_1,\dots,j_n\in\ZZ.$$  The \emph{side length} of $Q$, which we denote by $\side Q$, is $2^{-k}$; the \emph{diameter} of $Q$, which we denote by $\diam Q$, is $2^{-k}\sqrt{n}$. Let $\Delta(\RR^n)$ denote the collection of all dyadic cubes in $\RR^n$ and let $\Delta_1(\RR^n)$ denote the collection of all dyadic cubes in $\RR^n$ of side length at most 1.
For any cube $Q$ and $\lambda>0$, we let $\lambda Q$ denote the unique cube in $\RR^n$ that is obtained by dilating $Q$ by a factor of $\lambda$ with respect to the center of $Q$. Note that $\side \lambda Q= \lambda \side Q$ and $\diam \lambda Q = \lambda \diam Q$ for all cubes $Q$ and for all $\lambda>0$.

\section{Main results and organization of the paper} \label{s:new}

%We now state our main results. In Theorem \ref{t:big}, we characterize the 1-rectifiable and purely 1-unrectifiable parts of an arbitrary Radon measures in terms of the pointwise behavior of the lower 1-density $\lD1(\mu,x)$ of the measure and a geometric square function $J^*_p(\mu,x)$ defined below. In Theorem \ref{t:tstm}, we characterize finite measures supported on a rectifiable curve. In Theorem \ref{t:pd-big}, we characterize the 1-rectifiable part and purely 1-unrectifiable parts of pointwise doubling measures in terms of a geometric square function $\tJ_p(\mu,x)$ \emph{without} any assumption on the lower density of the measure.

In our main result, Theorem \ref{t:big}, we characterize the 1-dimensional rectifiable and purely unrectifiable parts of arbitrary Radon measures in terms of the pointwise behavior of the lower density and a geometric square function to be defined below. Also, see Theorem \ref{t:pd-big}, where we characterize the rectifiable and purely unrectifiable parts of a pointwise doubling measure in terms of the pointwise behavior of a simpler geometric square function alone.

For any dyadic cube $Q$ in $\RR^n$, define the set $\Delta^*(Q)$ of \emph{nearby cubes} to be the set of all dyadic cubes $R$ such that  \begin{equation}\side Q\leq \side R \leq 2\side Q\end{equation} and \begin{equation}3R\subseteq 1600\sqrt{n}\,Q.\end{equation} The constant $1600\sqrt{n}$ in the definition of $\Delta^*(Q)$ is chosen to be large enough to invoke Proposition \ref{p:curve} in the proof of Lemma \ref{l:draw}, but has not been optimized.

Let $1\leq p<\infty$ and let $\mu$ be a Radon  measure on $\RR^n$. For all $Q\in\Delta(\RR^n)$, we define $\beta_p^*(\mu,Q)\in[0,1]$ by \begin{equation}\beta_p^{*}(\mu,Q)^2 = \inf_\ell \max\left\{\beta_p(\mu,3R,\ell)^2\min\left(\frac{\mu(3R)}{\diam 3R},1\right):R\in \Delta^*(Q)\right\},\end{equation} where as usual the infimum runs over all straight lines in $\RR^n$. Note that $\beta_p^*(\mu,Q)=0$ whenever $\mu(1600\sqrt{n}\,Q)=0$, and $\beta_p^*(\mu,Q)$ is increasing in $p$ for all $\mu$ and $Q$. When $p=2$, \begin{equation}\begin{split}\label{e:beta2}\beta_2(\mu,3R,\ell)^2&\min\left(\frac{\mu(3R)}{\diam 3R},1\right) \\ &= \int_{3R} \left(\frac{\dist(x,\ell)}{\diam 3R}\right)^2 \min\left(\frac{1}{\diam 3R},\frac{1}{\mu(3R)}\right)\,d\mu(x).\end{split}\end{equation} Compare the normalizations in \eqref{e:beta-p-def}, \eqref{e:beta-m}, and \eqref{e:beta2}.

Define a \emph{density-normalized Jones function} $J^*_p(\mu,x)$ associated to the numbers $\beta^*_p(\mu,Q)$ as follows. For every $n\geq 2$, $1\leq p<\infty$, and Radon measure $\mu$ on $\RR^n$, define
\begin{equation}\label{e:J*def}
J^{*}_{p}(\mu,x):=\sum_{Q\in\Delta_1(\RR^n)} \beta^{*}_{p}(\mu,Q)^2\frac{\diam Q}{\mu(Q)}\chi_Q(x)\in[0,\infty]\quad\text{for all }x\in\RR^n,
\end{equation} (In the definition, we use the convention $0/0=0$ and $1/0=\infty$.)
This is a variant of the density-normalized Jones function $\tJ_2(\mu,x)$ used in \cite{BS}, which was associated to the beta numbers $\beta_2(\mu,3Q)$. Peter Jones conjectured circa 2000 that these types of weighted Jones functions could be used to characterize rectifiabilty of a measure (personal communication). In this paper's main result, Theorem \ref{t:big}, we verify Jones' conjecture by using $J^*_p(\mu,x)$ to identify the 1-rectifiable and purely 1-unrectifiable parts of an arbitrary Radon measure.

\begin{maintheorem}[characterization of the 1-rectifiable / purely 1-unrectifiable decomposition] \label{t:big} Let $n\geq 2$ and let $1\leq p\leq 2$. If $\mu$ is a Radon measure on $\RR^n$, then the decomposition $\mu=\mu^1_{\rect}+\mu^1_{\pu}$ in \eqref{e:decomp} is given by \begin{align} \mu^1_{\rect} &= \mu\res \left\{x\in\RR^n: \lD{1}(\mu,x)>0\text{ and }J_{p}^{*}(\mu,x)<\infty\right\},\\
\mu^1_{\pu} &= \mu \res\left\{x\in\RR^n: \lD{1}(\mu,x)=0\text{ or } J_p^{*}(\mu,x)=\infty\right\}.\end{align}\end{maintheorem}

\begin{maincor}[characterization of 1-rectifiable measures]\label{t:rect}Let $n\geq 2$ and let $1\leq p\leq 2$. If $\mu$ is a Radon measure on $\RR^n$, then $\mu$ is 1-rectifiable if and only if $\lD{1}(\mu,x)>0$ and $J^*_p(\mu,x)<\infty$ at $\mu$-a.e.~$x\in\RR^n$.\end{maincor}

\begin{maincor}[characterization of purely 1-unrectifiable measures] \label{t:pu} Let $n\geq 2$ and let $1\leq p\leq 2$. If $\mu$ is a Radon measure on $\RR^n$, then $\mu$ is purely 1-unrectifiable if and only if  $\lD{1}(\mu,x)=0$ or $J^{*}_{p}(\mu,x)=\infty$ at $\mu$-a.e.~$x\in\RR^n$.\end{maincor}

The proof of Theorem \ref{t:big} and its corollaries takes up \S\S \ref{s:nec}--\ref{s:abcd} below. A description of each of these sections appears at the end of this section. The restriction to exponents $p\geq 1$ in Theorem \ref{t:big} appears in the proof of Lemma \ref{l:lerman}; the restriction to $p\leq 2$ appears in the proof of Proposition \ref{p:1L}. It is an \emph{open problem} to determine if the conclusion of the theorem holds in the range $p>2$. The restriction to half open cubes (in the characteristic function $\chi_Q$) in the definition of $J^*_p(\mu,x)$ is imposed so that in the proof of Theorem \ref{t:suf}, we may use Lemma \ref{l:localize} (also see Remark \ref{rk:star}).

The methods that we develop to prove Theorem \ref{t:big} also yield a characterization of rectifiability of a measure with respect to a single rectifiable curve. Let $1\leq p<\infty$ and let $\mu$ be a Radon measure on $\RR^n$. For all $Q\in\Delta(\RR^n)$, we define $\beta_p^{**}(\mu,Q)\in[0,1]$ by \begin{equation}\label{e:starstar}\beta_p^{**}(\mu,Q) = \inf_\ell \max\left\{\beta_p(\mu,3R,\ell):R\in \Delta^*(Q)\right\},\end{equation} where the infimum runs over all straight lines in $\RR^n$.

\begin{maintheorem}[Traveling salesman theorem for measures]\label{t:tstm} Let $n\geq 2$ and let $1\leq p<\infty$. Let $\mu$ be a finite Borel measure on $\RR^n$ with bounded support. If $\Gamma\subseteq\RR^n$ is a rectifiable curve such that $\mu(\RR^n\setminus\Gamma)=0$, then \begin{equation}\label{e:S*} S^{**}_p(\mu):=\sum_{Q\in\Delta(\RR^n)} \beta_p^{**}(\mu,Q)^2\diam Q\lesssim_n \Haus^1(\Gamma).\end{equation} Conversely, if $S^{**}_p(\mu)<\infty$, then there is a rectifiable curve $\Gamma$ such that $\mu(\RR^n\setminus \Gamma)=0$ and \begin{equation}\Haus^1(\Gamma) \lesssim_{n} \diam \spt\mu + S^{**}_p(\mu).\end{equation}
\end{maintheorem}

Theorem \ref{t:tstm} may be viewed as an extension of the Analyst's traveling salesman theorem (see \S\ref{s:tst}), which characterizes subsets of rectifiable curves. A characterization of measures that are supported on a rectifiable curve was already known for Ahlfors regular measures (see \cite[Theorem 5.1]{Lerman}), but in this generality is new even for absolutely continuous measures of the form $\mu=\Haus^1\res E$. For the proof of Theorem \ref{t:tstm}, see \S\ref{s:abcd}.

For measures satisfying an additional weak regularity property, we also obtain simpler characterizations of the 1-rectifiable and purely 1-unrectifiable parts. Let $\mu$ be a Radon measure on $\RR^n$ and let $1\leq p<\infty$. The \emph{density-normalized Jones function} $\tJ_p(\mu,x)$ is defined by
\begin{equation}\label{e:tJ} \tJ_p(\mu,x):= \sum_{Q\in\Delta_1(\RR^n)} \beta_p(\mu,3Q)^2\frac{\diam Q}{\mu(Q)}\chi_Q(x)\in[0,\infty]\quad\text{for all }x\in\RR^n.\end{equation}
A Radon measure $\mu$ on $\RR^n$ is  called \emph{pointwise doubling} if \begin{equation} \limsup_{r\downarrow 0} \frac{\mu(B(x,2r))}{\mu(B(x,r))}<\infty\quad\text{for $\mu$-a.e.~$x\in\RR^n$.}\end{equation} The class of pointwise doubling measures includes the class of Radon measures $\mu$ on $\RR^n$ with $0<\lD1(\mu,x)\leq \uD1(\mu,x)<\infty$ for $\mu$-a.e.~$x\in\RR^n$, but is strictly larger.

\begin{maintheorem}[characterization of the 1-rectifiable / purely 1-unrectifiable decomposition for pointwise doubling measures] \label{t:pd-big} Let $n\geq 2$ and let $1\leq p\leq 2$. If $\mu$ is a pointwise doubling measure on $\RR^n$, then the decomposition $\mu=\mu^1_{\rect}+\mu^1_{\pu}$ in \eqref{e:decomp} is given by \begin{align} \mu^1_{\rect} &= \mu\res \left\{x\in\RR^n: \tJ_{p}(\mu,x)<\infty\right\},\\
\mu^1_{\pu} &= \mu \res\left\{x\in\RR^n: \tJ_p(\mu,x)=\infty\right\}.\end{align}\end{maintheorem}

See \S\ref{s:doubling} for the proof of Theorem \ref{t:pd-big}. It is an \emph{open problem} to decide if the conclusion of Theorem \ref{t:pd-big} holds for arbitrary Radon measures.

\begin{remark}[Added in May 2016] \label{r:mo} Shortly after a second draft of this paper appeared on the arXiv in April 2016, the problem following Theorem \ref{t:pd-big} was answered in the negative by Martikainen and Orponen \cite{MO-example}. For all $\varepsilon>0$, Martikainen and Orponen construct an example of a probability measure $\mu$ supported in the unit square in the plane for which \begin{enumerate}
 \item $\tJ_2(\mu,x)\leq \varepsilon$ for all $x\in\spt\mu$, and \item $\lD1(\mu,x)=0$ at $\mu$-a.e.~$x\in\RR^2$.\end{enumerate} (We caution the interested reader that \cite{MO-example} uses different notation for $\tJ_2(\mu,x)$.) Thus, the measure $\mu$ is purely 1-unrectifiable by Corollary \ref{t:pu} (or Lemma \ref{l:spu}) despite having $\tJ_2(\mu,\cdot)$ uniformly bounded. This shows that 1-rectifiable or purely 1-unrectifiable Radon measures cannot be characterized in terms of pointwise control of the Jones function $\tJ_2(\mu,\cdot)$ alone.  Moreover, let us note that since $\mu$ is a finite measure with bounded support, $$\sum_{Q\in\Delta(\RR^n)}\beta_2(\mu,3Q)^2\diam Q = \int \tJ_2(\mu,x)\,d\mu(x)+\sum_{Q\in\Delta(\RR^n)\setminus\Delta_1(\RR^n)} \beta_2(\mu,3Q)^2\diam Q<\infty$$ by (1). This shows that in Theorem \ref{t:tstm}, the numbers $\beta_2^{**}(\mu,Q)$, which take into account how $\mu$ looks in cubes $R$ nearby $Q$, cannot be replaced with the simpler numbers $\beta_2(\mu,3Q)$. For further discussion in this direction, see Remarks \ref{r:not-necessary} and \ref{r:sufficient}.\end{remark}

\subsection{Organization}

In Section \ref{s:tst}, we recall a metric characterization of rectifiable curves in $\RR^n$ as well as the Analyst's traveling salesman theorem, which characterizes subsets of rectifiable curves in $\RR^n$ in terms of a quadratic sum of  Jones' beta numbers. Both are indispensable tools in the theory of 1-rectifiable sets and measures. At the end of the section, we state Proposition \ref{p:curve}, which is a flexible extension of Jones' original traveling salesman construction that we use to draw rectifiable curves capturing positive measure in \S\S \ref{s:suf} and \ref{s:doubling}.

The proofs of Theorems \ref{t:big} and \ref{t:tstm} are developed over \S\S \ref{s:nec}--\ref{s:abcd}. In Section \ref{s:nec}, we focus on proving necessary conditions for a Radon measure to be 1-rectifiable, or equivalently, sufficient conditions for a Radon measure to be purely 1-unrectifiable. In particular, we prove that if $\mu$ is a Radon measure and $\Gamma$ is a rectifiable curve in $\RR^n$, then $J^*_2(\mu,x)<\infty$ at $\mu$-a.e.~$x\in\Gamma$ (see Theorem \ref{t:nec}). This result is some generalization and extension of the main result of the predecessor \cite{BS} of the current paper. In Section \ref{s:suf}, we establish sufficient conditions, which guarantee that a Radon measure is 1-rectifiable. In fact, we introduce beta numbers $\beta^{*,c}_p(\mu,Q)$, which are adapted to cubes $R\in\Delta^*(Q)$ such that $\mu(3R)\geq c\diam 3R$, and prove that for every Radon measure $\mu$ in $\RR^n$, $$\mu\res\{x\in\RR^n:\lD1(\mu,x)>(3/2)\sqrt{n}\cdot c\text{ and }J^{*,c}_p(\mu,x)<\infty\}$$ is 1-rectifiable for all $c>0$, where $J^{*,c}_p(\mu,x)$ is a density-normalized Jones function that is associated with the beta numbers $\beta^{*,c}_p(\mu,Q)$ (see Theorem \ref{t:suf}). The proof of our main result, Theorem \ref{t:big}, as well as the proofs of Corollary \ref{t:rect}, Corollary \ref{t:pu}, and Theorem \ref{t:tstm} are recorded in Section \ref{s:abcd}, using the results of Sections \ref{s:nec} and \ref{s:suf}.

In Section \ref{s:doubling}, we show how to modify proofs in Section \ref{s:suf} in order to prove that for every Radon measure in $\RR^n$, $$\mu\res\left\{x\in\RR^n: \limsup_{r\downarrow 0}\frac{\mu(B(x,2r))}{\mu(B(x,r))}<\infty\text{ and }\tJ_p(\mu,x)<\infty\right\}$$ is 1-rectifiable (see Theorem \ref{t:suf2}). Theorem \ref{t:pd-big} is then proved by combining this result with the main result of \cite{BS}.

In the last two sections, \S\S\ref{s:proof26} and \ref{ss:length}, we give a self-contained proof of Proposition \ref{p:curve}, which is modeled on Jones' traveling salesman construction. The proof of the proposition gives an algorithm for drawing a rectifiable curve $\Gamma$ through the leaves $V=\lim_{k\rightarrow\infty} V_k$ of a ``tree-like" sequence of $2^{-k}$-separated sets $V_k$. For example, the sets $V_k$ could be $2^{-k}$-nets of points in a bounded set $E\subseteq\RR^n$ (as in the proof of the Analyst's traveling salesman theorem) or the sets $V_k$ could be $\mu$ centers of mass (of the triples) of dyadic cubes of side length $2^{-k}$ (as in the proof of Lemma \ref{l:draw}). An important technical difference between Jones' original construction and Proposition \ref{p:curve} is that the latter does not require $V_{k+1}\supseteq V_k$. The added flexibility provided by Proposition \ref{p:curve} is crucial for the proofs of the sufficient conditions for 1-rectifiable measures, which we present in \S\S \ref{s:suf} and \ref{s:doubling}.

\section{The Analyst's traveling salesman theorem, again}\label{s:tst}

A \emph{rectifiable curve} $\Gamma$ in $\RR^n$ is the image $f([0,1])$ of a Lipschitz map $f:[0,1]\rightarrow\RR^n$. As Lipschitz maps are continuous and do not increase Hausdorff measure by more than a constant multiple, every rectifiable curve $\Gamma$ is a closed, connected set such that $\Haus^1(\Gamma)<\infty$. It is a remarkable fact---and an essential fact for the theory of 1-rectifiable sets and measures---that the converse of this observation is also true. For a proof of this fact that is valid in Hilbert space, see \cite[Lemma 3.7]{Schul-Hilbert}.

\begin{lemma}\label{l:fund} If $\Gamma\subseteq\RR^n$ (or $\Gamma$  inside the Hilbert space $\ell_2$) is a closed, connected set such that $\Haus^1(\Gamma)<\infty$, then there exists a Lipschitz map $f:[0,1]\rightarrow\RR^n$ such that $\Gamma=f([0,1])$. Moreover, $f$ can be found such that $|f(s)-f(t)| \leq 32 \Haus^1(\Gamma)|s-t|$ for all $0\leq s,t\leq 1$.\end{lemma}

\begin{corollary} \label{c:fund} If $\Gamma_1,\Gamma_2,\dots\subseteq\RR^n$ is a sequence of uniformly bounded, closed, connected sets, then there exists a compact, connected set $\Gamma\subseteq\RR^n$ and a subsequence $(\Gamma_{k_j})_{j=1}^\infty$ of $(\Gamma_k)_{k=1}^\infty$ such that $\Gamma_{k_j}\rightarrow\Gamma$ in the Hausdorff metric as $j\rightarrow\infty$ and $$\Haus^1(\Gamma)\leq 32\liminf_{k\rightarrow\infty} \Haus^1(\Gamma_{k})$$\end{corollary}

It is known that the constant 32 in Corollary \ref{c:fund} may be replaced with 1 (for example, see \cite[Theorem 3.18]{Falconer}), but knowledge of the optimal constant will not be important for the development below. The constant 32 in Lemma \ref{l:fund} is not optimal and likely may be replaced with 2. However, once again, knowledge of the optimal constant is not crucial for the applications to follow.

Next, we recall the \emph{Analyst's traveling salesman theorem}, which characterizes \emph{subsets} of rectifiable curves in $\RR^n$. The theorem was first conceived and proved by P.~Jones \cite{Jones-TSP} for sets in the plane and then extended by Okikiolu \cite{Ok-TST} for sets in $\RR^n$, for all $n\geq 3$. For a formulation of the theorem  in infinite-dimensional Hilbert space, see Schul \cite{Schul-Hilbert}. Partial information is also known in the Heisenberg group; see Li and Schul \cite{LS1,LS2} (as well as the previous work by Ferrari, Franchi, and Pajot \cite{FFP} and Juillet \cite{Juillet}). For traveling salesman type theorems in graph inverse limit spaces, see G.C.~David and Schul \cite{Laakso-TSP}.

\begin{definition}\label{d:beta-infty} Let $E\subseteq\RR^n$ be any set. For every bounded set $Q\subseteq\RR^n$ such that $E\cap Q\neq\emptyset$, define the quantity $\beta_{E}(Q)\in[0,1]$ by \begin{equation*}
  \beta_{E}(Q):=\inf_{\ell}\sup_{x\in E\cap Q}\frac{\dist(x,\ell)}{\diam Q},\end{equation*} where $\ell$ ranges over all lines in $\RR^n$. By convention, we set $\beta_E(Q)=0$ if $E\cap Q=\emptyset$.
\end{definition}

\begin{theorem}[Analyst's traveling salesman theorem, \cite{Jones-TSP,Ok-TST}] \label{t:tst} A bounded set $E\subseteq\RR^n$ is a subset of a rectifiable curve in $\RR^n$ if and only if
\begin{equation*} \beta^2(E):=\sum_{Q\in\Delta(\RR^n)}\beta_E(3Q)^{2}\diam Q<\infty.\end{equation*} Moreover, there exists a constant $C=C(n)\in(1,\infty)$ (independent of $E$) such that
\begin{itemize}
\item $\diam E+\beta^2(E)\leq C \Haus^1(\Gamma)$ for every connected set $\Gamma$ containing $E$, and
\item there exists a connected set $\Gamma\supseteq E$ such that $\Haus^1(\Gamma)\leq C(\diam E +\beta^2(E)).$
\end{itemize} \end{theorem}

The cube dilation factor 3 appearing in Theorem \ref{t:tst} is somewhat arbitrary and may be replaced with any value strictly greater than 1. In particular, in \S\ref{s:nec}, we need the ``necessary" half of the Analyst's traveling salesman theorem with a dilation factor  strictly greater than $3$. For a derivation of Corollary \ref{c:tst} from Theorem \ref{t:tst}, see \cite[\S2]{BS}.

\begin{corollary} \label{c:tst} For all $n\geq 2$ and $3<a<\infty$, there is a constant $C'=C'(n,a)\in(1,\infty)$ such that if $E\subseteq \RR^n$ is bounded and $\Gamma$ is a connected set containing $E$, then $$\sum_{Q\in\Delta(\RR^n)} \beta_E(aQ)^2\diam Q \leq C' \Haus^1(\Gamma).$$\end{corollary}

The following proposition is modeled on and is some extension of a lemma from \cite{JLS} (currently in preparation by P.~Jones, G.~Lerman, and the second author of this paper) and has roots in P. Jones' proof of the Analyst's traveling salesman theorem from \cite{Jones-TSP}.
The variant in \cite{JLS} is a criterion for constructing Lipschitz graphs, whereas Proposition \ref{p:curve} is a criterion for constructing rectifiable curves. For a related criterion for constructing bi-Lipschitz surfaces, see \cite[Theorem 2.5]{DT}. One technical difference between Jones' original construction and Proposition \ref{p:curve} is that in the latter we do not assume $V_{k+1}\supseteq V_{k}$. This added flexibility is crucial for our applications in sections \ref{s:suf} and \ref{s:doubling} below.

\begin{proposition} \label{p:curve} Let $n\geq 2$, let $C^\star>1$, let $x_0\in\RR^n$, and let $r_0>0$. Let $(V_k)_{k=0}^\infty$ be a sequence of nonempty finite subsets of $B(x_0,C^\star r_0)$ such that
\begin{enumerate}
\item[($V_I$)] distinct points $v,v'\in V_k$ are uniformly separated: $|v-v'|\geq 2^{-k}r_0$;
\item[($V_\two$)] for all $v_k\in V_k$, there exists $v_{k+1}\in V_{k+1}$ such that $|v_{k+1}-v_k|< C^\star 2^{-k}r_0$; and,
\item[($V_\three$)] for all $v_{k}\in V_{k}$ ($k\geq 1$), there exists $v_{k-1}\in V_{k-1}$ such that $|v_{k-1}-v_k|< C^\star 2^{-k}r_0$.
\end{enumerate} Suppose that for all $k\geq 1$ and for all $v\in V_k$ we are given a straight line $\ell_{k,v}$ in $\RR^n$ and a number $\alpha_{k,v}\geq 0$ such that
 \begin{equation}\label{e:alpha}\sup_{x\in (V_{k-1}\cup V_k)\cap B(v, 65 C^\star 2^{-k}r_0)} \dist(x,\ell_{k,v}) \leq \alpha_{k,v} 2^{-k}r_0\end{equation} and  \begin{equation}\label{e:asum}\sum_{k=1}^\infty\sum_{v\in V_k} \alpha_{k,v}^2 2^{-k} r_0<\infty.\end{equation}
Then the sets $V_k$ converge in the Hausdorff metric to a compact set $V\subseteq \overline{B(x_0,C^\star r_0)}$ and
there exists a compact, connected set $\Gamma\subseteq\overline{B(x_0,C^\star r_0)}$ such that $\Gamma\supseteq V$ and
\begin{align}\label{e:G-length}
\Haus^1(\Gamma)\lesssim_{n,C^\star} r_0 + \sum_{k=1}^\infty\sum_{v\in V_k} \alpha_{k,v}^2 2^{-k}r_0.
\end{align}
\end{proposition}

\begin{remark} The ``sufficient" half of the Analyst's traveling salesman theorem is an application of Proposition \ref{p:curve}. To see this, suppose that $E\subseteq\RR^n$ is a bounded set with diameter $r_0>0$. For each $k\geq 0$, let $V_k$ be a maximal subset of $E$ such that $|v-v'|\geq 2^{-k}r_0$ for all distinct $v,v'\in V_k$. Then conditions $(V_I)$, $(V_\two)$, and $(V_\three)$ of Proposition \ref{p:curve} hold with $C^\star=2$. For each $k\geq 1$ and $v\in V_k$, let $Q$ be a minimal dyadic cube such that $v\in Q$ and $3Q\supseteq B(v,65C^\star 2^{-k}r_0)$ and let $\ell_{k,v}$ be a line for which $\sup_{x\in E\cap 3Q} \dist(x,\ell_{k,v})\leq 2\beta_E(3Q)\diam 3Q\sim_n \beta_E(3Q) 2^{-k}r_0$. Then $$\sup_{x\in (V_{k-1}\cup V_k)\cap B(v,65 C^\star 2^{-k}r_0)} \dist(x,\ell_{k,v})\leq C(n)\beta_E(3Q) 2^{-k}r_0=:\alpha_{k,v} 2^{-k}r_0.$$ Each cube $Q\in \Delta(\RR^n)$ can be associated to the pair $(k,v)$ in this way for at most $C(n)$ values of $k\geq 1$ and $v\in V_k$ by $(V_I)$. Thus, by Proposition \ref{p:curve}, there exists a compact, connected set $\Gamma\subseteq\RR^n$ containing $V:=\lim_{k\rightarrow\infty} V_k=\overline{E}$ with $$\Haus^1(\Gamma)\lesssim_n r_0+ \sum_{k=1}^\infty\sum_{v\in V_k} \alpha_{k,v}^2 2^{-k}r_0 \lesssim_n \diam E+\sum_{Q\in\Delta(\RR^n)} \beta_E(3Q)^2\diam Q.$$ \end{remark}

The proof of Proposition \ref{p:curve} is deferred to \S\S\ref{s:proof26} and \ref{ss:length}, which are independent of \S\S \ref{s:nec}--\ref{s:doubling}.

\section{Necessity: $\mu$ is 1-rectifiable implies \positivedensity and $J^{*}_{2}(\mu,x)<\infty$ $\mu$-a.e.} \label{s:nec}

Recall from the introduction that $m$-rectifiable measures have positive lower Hausdorff $m$-density almost everywhere.

\begin{lemma}[{\cite[Lemma 2.7]{BS}}] \label{l:positive2} Let $\mu$ be a Radon measure on $\RR^n$ and let $1\leq m\leq n-1$. If $\mu$ is $m$-rectifiable, then $\lD{m}(\mu,x)>0$ at $\mu$-a.e.~$x\in\RR^n$. \end{lemma}

Lemma \ref{l:positive2} is a consequence of the connection between lower Hausdorff $m$-density of a measure and $m$-dimensional packing measure $\mathcal{P}^m$. See \cite{BS} for details. By inspection, the proof in \cite{BS} shows that $\mu(\{x\in f([0,1]^m):\lD{m}(\mu,x)=0\})=0$ for every Lipschitz map $f:[0,1]^m\rightarrow\RR^n$. Thus, we have the following stronger formulation of Lemma \ref{l:positive2}.

\begin{lemma}\label{l:spu} Let $\mu$ be a Radon measure on $\RR^n$ and let $1\leq m\leq n-1$. Then the measure $\mu\res\{x\in\RR^n: \lD{m}(\mu,x)=0\}$ is purely $m$-unrectifiable.\end{lemma}

Our goal in the remainder of the section is to prove the following theorem.

\begin{theorem} \label{t:nec} Let $n\geq 2$. If $\mu$ is a Radon measure on $\RR^n$ and $\Gamma$ is a rectifiable curve, then the function $J^*_2(\mu,\cdot)\in L^1(\mu\res\Gamma)$ and $J^*_2(\mu,x)<\infty$ at $\mu$-a.e.~$x\in\Gamma$.\end{theorem}

At the core of Theorem \ref{t:nec} is the following quantitative statement, which is some extension and generalization of \cite[Proposition 3.1]{BS}. In particular, let us stress that the lower Ahlfors regularity condition on $E\subseteq \Gamma$ has been removed.

\begin{proposition}\label{p:1L} Let $n\geq 2$. If $\nu$ is a finite Borel measure on $\RR^n$ and $\Gamma$ is a rectifiable curve, then
\begin{equation} \label{e:1L} \sum_{\stackrel{Q\in\Delta(\RR^n)}{\nu(\Gamma\cap 1600\sqrt{n}\,Q)>0}} \beta^*_2(\nu,Q)^2\diam Q\lesssim_{n} \Haus^1(\Gamma)+\nu(\RR^n\setminus \Gamma).\end{equation}
\end{proposition}

\begin{proof} The proof that we present is an adaptation of the proof of Proposition 3.1 in \cite{BS}, the forerunner to this paper by the same name. For clarity, we develop the part of the proof that needs to be altered.

Fix constants $\varepsilon>0$ and $a>3$ to be specified later, ultimately depending on only the ambient dimension $n$. Define two families $\Delta_\Gamma$ and $\Delta_2$ of dyadic cubes in $\RR^n$, as follows.
\begin{align*}\Delta_\Gamma&=\{Q\in \Delta(\RR^n): \nu(\Gamma\cap 1600\sqrt{n}\,Q)>0\text{ and } \varepsilon\beta^*_2(\nu,Q)\leq \beta_{\Gamma}(aQ)\},\text{ and}\\
\Delta_2&=\{Q\in\Delta(\RR^n):\nu(\Gamma\cap 1600\sqrt{n}\,Q)>0\text{ and } \beta_{\Gamma}(aQ) < \varepsilon\beta_2^*(\nu,Q)\}.\end{align*} Note that $\Delta_\Gamma$ and $\Delta_2$ consist of the cubes appearing in \eqref{e:1L} for which either $\beta_{\Gamma}(aQ)$ or $\varepsilon\beta_2^*(\nu,Q)$ is the dominant quantity, respectively. It follows that \begin{equation}\sum_{\stackrel{Q\in\Delta(\RR^n)}{\nu(\Gamma\cap 1600\sqrt{n}\,Q)>0}} \beta^*_2(\nu,Q)^2\diam Q \leq
  \underbrace{\varepsilon^{-2}\sum_{Q\in\Delta_\Gamma}\beta_{\Gamma}(aQ)^2\diam Q}_I +
  \underbrace{\sum_{Q\in\Delta_2}\beta_{2}^{*}(\nu,Q)^2\diam Q}_{\two}. \label{I-II}
\end{equation} We shall estimate the terms $I$ and $\two$ separately. The former will be controlled by $\Haus^1(\Gamma)$ and the latter will be controlled by $\nu(\RR^n\setminus\Gamma)$.

To estimate $I$, we note that by Jones' (when $n=2$) and by Okikiolu's (when $n\geq 3$) traveling salesman theorems (in the form of Corollary \ref{c:tst}),
\begin{equation}\label{Iest} I \leq \varepsilon^{-2} \sum_{Q\in\Delta(\RR^n)}\beta_\Gamma(aQ)^2\diam Q \leq C'\varepsilon^{-2}\Haus^1(\Gamma),\end{equation} where $C'$ is a finite constant determined by $n$ and $a$.

In order to estimate $\two$, first decompose $\RR^n\setminus \Gamma$ into a family $\cT$ of Whitney cubes with the following specifications. \begin{itemize}
\item The union over all sets in $\cT$ is $\RR^n\setminus \Gamma$.
\item Each set in $\cT$ is a half open cube in $\RR^n$ of the form $[a_1,b_1)\times\dots\times[a_n,b_n)$.
\item If $T_1,T_2\in\cT$, then either $T_1=T_2$ or $T_1\cap T_2=\emptyset$.
\item If $T\in\cT$, then $\dist(T,\Gamma)\leq \diam T\leq 4\dist(T,\Gamma)$.
\end{itemize} (To obtain this decomposition, one can modify the standard Whitney decomposition in Stein's book \cite{Stein} by replacing each closed cube with the corresponding half open cube.) Here $\dist(T,\Gamma)=\inf_{x\in T}\inf_{y\in \Gamma}|x-y|$. For each $k\in\ZZ$, we define \begin{equation*} \cT_k=\{T\in\cT:2^{-k-1}<\dist(T,\Gamma)\leq 2^{-k}\}.\end{equation*} Also, for every cube $Q$, we set $\cT(Q)=\{T\in\cT:\nu(Q\cap T)>0\}$ and $\cT_k(Q)=\cT_k\cap\cT(Q)$. First we will estimate $\beta_2^*(\nu,Q)^2\diam Q$ for each $Q\in\Delta_2$ and then we will estimate $\two$.

Fix $Q\in\Delta_2$, say with $\side Q=2^{-k_0}$, and pick any line $\ell$ in $\RR^n$ such that \begin{equation}\label{e:apple0} \sup_{z\in \Gamma\cap aQ} \dist(z,\ell) \leq 2\beta_{\Gamma}(aQ)\diam aQ  <2\varepsilon \beta^*_2(\nu,Q) \diam aQ = 2a\varepsilon \beta_2^*(\nu,Q)\diam Q.\end{equation} We will estimate $\beta_2^*(\nu,Q)^2$ from above using $\ell$: \begin{equation*} \beta_2^*(\nu,Q)^2 \leq \max_{R\in\Delta^*(Q)} \beta_2(\nu,3R,\ell)^2 \min\left\{\frac{\nu(3R)}{\diam 3R},1\right\}.\end{equation*} Fix a cube $R\in\Delta^*(Q)$ nearby $Q$ and recall that $\diam Q\leq \diam R\leq 2\diam Q$. For ease of notation, set $m_R=\min\{\nu(3R)/\diam 3R,1\}$. To estimate $\beta_2(\nu,3R,\ell)^2 m_R$ from above, divide $3R$ into two sets $N_R$ (``near") and $F_R$ (``far"), where
\begin{equation*}N_R=\{x\in 3R: \dist(x,\ell)\leq 2a\varepsilon\beta_2^*(\nu,Q)\diam Q \}\end{equation*} and \begin{equation*}F_R=\{x\in 3R :\dist(x,\ell)>2a\varepsilon\beta_2^*(\nu,Q)\diam Q\}.\end{equation*} It immediately follows that \begin{align*}\beta_2(\nu,3R,\ell)^2 m_R &\leq
\int_{N_R} \left(\frac{\dist(x,\ell)}{\diam 3R}\right)^2 \, \frac{d\nu(x)}{\nu(3R)} + \int_{F_R} \left(\frac{\dist(x,\ell)}{\diam 3R}\right)^2 m_R\, \frac{d\nu(x)}{\nu(3R)}\\
&\leq
\frac{4}{9}a^2\varepsilon^2 \beta_2^*(\nu,Q)^2 + \int_{F_R} \left(\frac{\dist(x,\ell)}{\diam 3R}\right)^2 m_R\, \frac{d\nu(x)}{\nu(3R)}.\end{align*} By the triangle inequality and \eqref{e:apple0}, we have $$\dist(x,\ell) \leq \dist(x,\Gamma\cap aQ) + 2a\varepsilon\beta^*_2(\nu,Q)\diam Q.$$ Using this and the inequality $(p+q)^2 \leq 2p^2+2q^2$, it follows that $$\beta_2(\nu,3R,\ell)^2 m_R \leq \frac{4}{3}a^2\varepsilon^2 \beta_2^*(\nu,Q)^2 + 2\int_{F_R} \left(\frac{\dist(x,\Gamma\cap aQ)}{\diam 3R} \right)^2 m_R\, \frac{d\nu(x)}{\nu(3R)}.$$
Now, for each $x\in F_R\subseteq 3R$, we have $\dist(x,\Gamma) \leq \diam 1600\sqrt{n}\,Q$, because $3R\subseteq 1600\sqrt{n}\,Q$ and $\nu(\Gamma\cap 1600\sqrt{n}\,Q)>0$.
Hence, taking the constant $$a:=1600\sqrt{n}+3200n$$ (so $a=1600\sqrt{n}+2\diam 1600\sqrt{n}\,Q$) ensures that $\dist(x,\Gamma\cap aQ)=\dist(x,\Gamma)$ for all $x\in F_R$. Thus,  $$\beta_2(\nu,3R,\ell)^2 m_R \leq \frac{4}{3}a^2\varepsilon^2\beta_2^*(\nu,Q)^2 + 2\int_{F_R} \left(\frac{\dist(x,\Gamma)}{\diam 3R}\right)^2m_R\,\frac{d\nu(x)}{\nu(3R)}.$$ Now, letting $R$ range over  $\Delta^*(Q)$ and declaring that $\varepsilon$ be chosen so that $(4/3)a^2\varepsilon^2=1/3$, we conclude that $$\beta_2^*(\nu,Q)^2 \leq \max_{R\in\Delta^*(Q)} 3\int_{F_R} \left(\frac{\dist(x,\Gamma)}{\diam 3R}\right)^2 m_R\, \frac{d\nu(x)}{\nu(3R)}.$$ Next, because $m_R/\nu(3R) \leq 1/\diam 3R \leq 1 / \diam 3Q= 1/3\diam Q$, we obtain \begin{equation}\label{e:prewhit}\beta_2^*(\nu,Q)^2 \diam Q \leq \max_{R\in\Delta^*(Q)} \int_{F_R} \left(\frac{\dist(x,\Gamma)}{\diam 3R}\right)^2\,d\nu(x).\end{equation}
Note that if $x\in F_R$ for some $R\in\Delta^*(Q)$, then $x\not\in\Gamma$ by \eqref{e:apple0} and $3R\subseteq aQ$.
Thus, we may employ the Whitney decomposition $\cT$ of $\RR^n\setminus\Gamma$ to estimate the right hand side of \eqref{e:prewhit}: $$\beta_2^*(\nu,Q)^2 \diam Q \leq \max_{R\in\Delta^*(Q)} \sum_{T\in \cT(3R)} \sup_{x\in T} \left(\frac{\dist(x,\Gamma)}{\diam Q} \right)^2 \nu(T\cap 3R).$$
Because $3R\subseteq aQ$ for all $R\in\Delta^*(Q)$ by our choice of $a$ above, it follows that $$\beta_2^*(\nu,Q)^2\diam Q \leq \sum_{T\in\cT(aQ)} \sup_{x\in T}\left(\frac{\dist(x,\Gamma)}{\diam Q}\right)^2 \nu(T\cap aQ).$$ Recall that $\side Q = 2^{-k_0}$. If $T\in \cT_k(aQ)$, then by bounding the distance between a point in $T\cap aQ$ and a point in $\Gamma\cap 1600\sqrt{n}\,Q$, we observe that $$2^{-k-1} \leq \dist(T,\Gamma) \leq \diam aQ = a\sqrt{n}\side Q = a\sqrt{n}2^{-k_0}.$$ It follows that $k\geq k_1:= k_0-1-\lfloor \log_2 a\sqrt{n}\rfloor$ whenever $T\in\cT_k(aQ)$. Also, if $T\in\cT_k$ and $x\in T$, then $\dist(x,\Gamma) \leq \dist(T,\Gamma)+\diam T \leq 5\dist(T,\Gamma) \leq 5\cdot 2^{-k}$. Therefore, \begin{equation}\label{e:38analogue}\beta_2^*(\nu,Q)^2\diam Q \leq 25 \sum_{k=k_1}^\infty \sum_{T\in\cT_k(aQ)} \left(\frac{2^{-k}}{\diam Q}\right)^2 \nu(T\cap aQ).\end{equation} This estimate is valid for every cube $Q\in\Delta_2$. We emphasize that equation \eqref{e:38analogue} is the analogue of \cite[(3.8)]{BS} in the proof of \cite[Proposition 3.1]{BS}.

Equipped with \eqref{e:38analogue}, one may now repeat the argument appearing after \cite[(3.8)]{BS} in the proof of \cite[Proposition 3.1]{BS}, \emph{mutatis mutandis}, to obtain \begin{equation} \label{IIest} \two \lesssim_{n} \nu(\RR^n\setminus \Gamma).\end{equation}
Combining (\ref{I-II}), (\ref{Iest}) and (\ref{IIest}), we obtain \eqref{e:1L}, as desired.\end{proof}

\begin{remark}The proof of Proposition \ref{p:1L} is robust in the sense that it does not overly rely on the specific geometry or combinatorics of sets in $\Delta^*(Q)$. For example, a version of the proposition holds if the triples $3R$ of cubes $R\in\Delta^*(Q)$ appearing the definition of $\beta^*_2(\mu,Q)$ are replaced with a family of balls that are nearby $Q$ and whose diameters are comparable to the diameter of $Q$, provided that all relevant constants are chosen uniformly across $Q\in\Delta(\RR^n)$. We leave details to the interested reader.\end{remark}

\begin{proof}[Proof of Theorem \ref{t:nec}] Let $\mu$ be a Radon measures on $\RR^n$ and let $\Gamma$ be a rectifiable curve. Then
\begin{align}
\int_{\Gamma} J^{*}_{2}(\mu,x)\,d\mu(x)
 \notag &= \sum_{Q\in\Delta_1(\RR^n)} \beta_{2}^{*}(\mu,Q)^2\frac{\diam Q}{\mu(Q)}\int_{\Gamma}\chi_Q(x)\,d\mu(x)\\
 &=  \sum_{\stackrel{Q\in\Delta_1(\RR^n)}{\mu(\Gamma\cap Q)>0}} \beta_{2}^{*}(\mu,Q)^2 \diam Q\, \frac{\mu(\Gamma\cap Q)}{\mu(Q)} \leq \sum_{\stackrel{Q\in\Delta_1(\RR^n)}{\mu(\Gamma\cap Q)>0}} \beta_{2}^{*}(\mu,Q)^2 \diam Q. \label{e:nec1}
\end{align} Let $K$ be the closure of the union of cubes $$\bigcup_{\stackrel{Q\in\Delta_1(\RR^n)}{\mu(\Gamma\cap Q)>0}}\bigcup_{R\in\Delta^*(Q)} 3R,$$ which is compact since $\Gamma$ is bounded. Then the restriction $\nu=\mu\res K$ is finite and \begin{equation}\label{e:nec2}\sum_{\stackrel{Q\in\Delta_1(\RR^n)}{\mu(\Gamma\cap Q)>0}} \beta_{2}^{*}(\mu,Q)^2 \diam Q = \sum_{\stackrel{Q\in\Delta_1(\RR^n)}{\nu(\Gamma\cap Q)>0}} \beta_{2}^{*}(\nu,Q)^2 \diam Q\lesssim_n \Haus^1(\Gamma)+\nu(\RR^n\setminus\Gamma)<\infty\end{equation} by Proposition \ref{p:1L}. Chaining together inequalities in \eqref{e:nec1} and \eqref{e:nec2}, we conclude that $J^*_2(\mu,\cdot)\in L^1(\mu\res \Gamma)$. Therefore, $J_2^*(\mu,x)<\infty$ at $\mu$-a.e.~ $x\in\Gamma$, as well.
\end{proof}

\begin{remark}\label{r:not-necessary} Recall from \eqref{e:starstar} that if $\mu$ is a Radon measure on $\RR^n$ and $Q\in\Delta(\RR^n)$, then $$\beta_2^{**}(\mu,Q)= \inf_\ell \max\left\{\beta_2(\mu,3R,\ell):R\in \Delta^*(Q)\right\},$$ where the infimum runs over all straight lines in $\RR^n$. Define the density-normalized Jones function $J^{**}_2(\mu,x)$ associated to the numbers $\beta_2^{**}(\mu,Q)$ by $$J_2^{**}(\mu,x):=\sum_{Q\in\Delta_1(\RR^n)} \beta_2^{**}(\mu,Q)^2\frac{\diam Q}{\mu(Q)}\chi_Q(x).$$ In Theorem \ref{t:nec}, we showed that if $\mu$ is a Radon measure and $\Gamma$ is a rectifiable curve, then $J^*_2(\mu,x)<\infty$ at $\mu$-a.e.~$x\in\Gamma$. However, it is currently an \emph{open problem} to decide whether $\mu$ is a Radon measure and $\Gamma$ is a rectifiable curve imply that $J_2^{**}(\mu,x)<\infty$ at $\mu$-a.e.~$x\in\Gamma$. For the motivation for this problem, see Remark \ref{r:sufficient}.\end{remark}

\section{Sufficiency: \positivedensity and $J^{*}_{p}(\mu,x)<\infty$ $\mu$-a.e.~implies $\mu$ is 1-rectifiable}\label{s:suf}

Our goal in this section is to show that $\lD1(\mu,\cdot)>0$ almost everywhere and $J^*_p(\mu,\cdot)<\infty$ almost everywhere are together a sufficient condition for a Radon measure $\mu$ on $\RR^n$ to be 1-rectifiable. As an intermediate step, we first introduce and work with  beta numbers and weighted Jones functions that are adapted to cubes with uniformly large (coarse) density.

Let $\mu$ be a Radon measure on $\RR^n$, let $1\leq p<\infty$, and let $c>0$. For all dyadic cubes $Q\in\Delta(\RR^n)$, we define $\beta_p^{*,c}(\mu,Q)\in[0,1]$ by \begin{equation}\begin{split}\beta_p^{*,c}(\mu,Q)^2 = \inf_\ell \max\Big\{\beta_p(\mu,3R,\ell)^2 \min\{c,1\}:\ &R\in \Delta^*(Q)\text{ and }\\
&\mu(3R)\geq c\diam 3R\Big\},\end{split}\end{equation} where in the infimum $\ell$ ranges over all straight lines in $\RR^n$. If there do not exist cubes $R\in\Delta^*(Q)$ such that $\mu(3R)\geq c\diam 3R$, then by convention we let $\beta_p^{*,c}(\mu,x)=0$. Evidently, for every measure $\mu$ and dyadic cube $Q$, we have the comparison $$\beta_p^{*,c}(\mu,Q) \leq \beta_p^*(\mu,Q)\quad\text{for all }c>0.$$ The \emph{$c$-adapted $L^p$ density-normalized Jones function} $J_p^{*,c}(\mu,x)$ is defined by \begin{equation}J_p^{*,c}(\mu,x) = \sum_{Q\in\Delta_1(\RR^n)} \beta_p^{*,c}(\mu,Q)^2 \frac{\diam Q}{\mu(Q)}\chi_Q(x)\in[0,\infty]\quad\text{for all $x\in\RR^n$}.\end{equation} In particular, if $J^*_p(\mu,x)<\infty$ $\mu$-a.e., then $J^{*,c}_p(\mu,x)<\infty$ for all $c>0$, at $\mu$-a.e. $x\in\RR^n$.

\begin{theorem} \label{t:suf} Let $\mu$ be a Radon measure on $\RR^n$, let $1\leq p<\infty$, and let $c>0$. Then $$\mu\res\{x\in\RR^n: \lD1(\mu,x)>(3/2)\sqrt{n}\cdot c\text{ and }J^{*,c}_p(\mu,x)<\infty\}$$ is 1-rectifiable. \end{theorem}

In order to construct rectifiable curves that capture measure in Theorem \ref{t:suf}, we will use Proposition \ref{p:curve} in conjunction with the following observation due to Lerman \cite{Lerman}. As stated, Lemma \ref{l:lerman} is a small variation of \cite[Lemma 6.4]{Lerman}.

\begin{lemma}\label{l:lerman} Let $n\geq 2$ and let $1\leq p<\infty$. Let $\mu$ be a Radon measure on $\RR^n$, let $E$ be a Borel set of positive diameter such that $0<\mu(E)<\infty$, and let $$z_E := \int_E z\,\frac{d\mu(z)}{\mu(E)}\in\RR^n$$ denote the center of mass of $E$ with respect to $\mu$. For every straight line $\ell$ in $\RR^n$, $$\dist(z_E,\ell) \leq \beta_p(\mu,E,\ell)\diam E.$$ \end{lemma}

\begin{proof} For every affine subspace $\ell$ in $\RR^n$, the function $\dist(\cdot,\ell)^p$ is convex provided that $1\leq p<\infty$. Thus, \begin{align*}
\dist(z_E,\ell)^p &=\dist\left(\int_E z\,\frac{d\mu(z)}{\mu(E)},\ell\right)^p\\
&\leq \int_E\dist(z,\ell)^p\,\frac{d\mu(z)}{\mu(E)} = \beta_p(\mu,E,\ell)^p (\diam E)^p\end{align*} by Jensen's inequality.\end{proof}

Let us define a \emph{tree of dyadic cubes} to be a set $\cT$ of dyadic cubes with unique maximal element $Q_0$ (ordered by inclusion) such that if $R\in\cT$, then $Q\in\cT$ for all dyadic cubes $R\subseteq Q\subseteq Q_0$. Denote $Q_0$ by $\Top(\cT)$. An \emph{infinite branch} of $\cT$ is defined to be a chain $Q_{0}\supseteq Q_{1}\supseteq Q_{2}\supseteq \dots$ of cubes in $\cT$ such that $\side Q_l = 2^{-l} \side Q_0$ for all $l\geq 0$. We define the set $\leaves(\cT)$ of \emph{leaves of $\cT$} to be $$\leaves(\cT) := \bigcup \left\{\bigcap_{i=0}^\infty \overline{Q_i}:Q_0\supseteq Q_1\supseteq Q_2\supseteq \cdots\text{ is an infinite branch of $\cT$}\right\}.$$

The following lemma is the heart of Theorem \ref{t:suf}.

\begin{lemma}[drawing rectifiable curves through the leaves of lower Ahlfors regular trees] \label{l:draw} Let $n\geq 2$, let $1\leq p<\infty$, and let $c>0$. Let $\mu$ be a Radon measure on $\RR^n$. If $\cT$ is a tree of dyadic cubes such that \begin{equation}\label{e:cond1} \frac{\mu(3Q)}{\diam 3Q}\geq c \quad\text{for all }Q\in\cT,\text{ and}\end{equation} \begin{equation}\label{e:cond2} S_p^{*,c}(\mu,\cT):=\sum_{Q\in\cT} \beta_p^{*,c}(\mu,Q)^2\diam Q <\infty,\end{equation} then there exists a rectifiable curve $\Gamma$ in $\RR^n$ such that $\Gamma\supseteq\leaves(\cT)$ and \begin{equation} \Haus^1(\Gamma)\lesssim_{n} \diam \Top(\cT)+\max\{c^{-1},1\}S_p^{*,c}(\mu,\cT).\end{equation}\end{lemma}

\begin{proof} Applying a dilation and a translation, we may assume without loss of generality that $\Top(\cT)=[0,1)^n$. By deleting irrelevant cubes from $\cT$, we may also assume without loss of generality that every cube $Q\in\cT$ belongs to an infinite branch of $\cT$. To proceed, we will aim to use Proposition \ref{p:curve}. Set parameters $$C^\star=4\quad\text{and}\quad r_0=3\diam \Top(\cT)=\diam [-1,2)^n=3\sqrt{n}.$$ Below we will freely use the fact that $\diam 3Q= r_0\side Q$ for all $Q\in\cT$.

For each $Q\in\cT$, let $z_{3Q}$ denote the $\mu$ center of mass of $3Q$, i.e. $$z_{3Q}=\int_{3Q} z\,\frac{d\mu(z)}{\mu(3Q)}.$$ For each $k\geq 0$, let $Z_k= \{z_{3Q}:Q\in\cT\text{ and }\side Q= 2^{-k}\}$ and choose $V_k$ to be any maximal $2^{-k}r_0$-separated subset of $Z_k$. Pick any $x_0\in 3Q_0$. Then $$V_k\subseteq Z_k\subseteq 3\Top(\cT)\subseteq B(x_0,r_0)\subseteq B(x_0,C^\star r_0).$$ Also note that $V_k$ satisfies condition $(V_I)$ of Proposition \ref{p:curve} by the definition of $V_k$.

To check  condition $(V_\two)$, let $k\geq 0$ and let $v\in V_k$, say $v=z_{3Q}$ for some $Q\in\cT$ with $\side Q= 2^{-k}$. Recall that by assumption every cube in $\cT$ belongs to an infinite branch of $\cT$. Hence there exists $R\in\cT$ such that $R\subseteq Q$ and $\side R=\frac12\side Q$. By maximality, there exists $v'=z_{3P}\in V_{k+1}$ for some $P\in \cT$ such that $\side P=\side R$ and $|z_{3R}-v'|< 2^{-(k+1)}r_0$. It follows that $$|v'-v| < |v'-z_{3R}|+|z_{3R}-v| \leq 2^{-(k+1)}r_0+\diam 3Q = \frac32\cdot 2^{-k}r_0\leq C^\star 2^{-k}r_0.$$ Thus, condition $(V_\two)$ is satisfied.

To check condition $(V_\three)$, let $k\geq 1$ and let $v\in V_k$, say $v=z_{3Q}$ for some $Q\in\cT$ with $\side Q= 2^{-k}$. Let $R$ denote the parent of $Q$, which necessarily belongs to $\cT$. By maximality, there exists $v'=z_{3P}\in V_{k-1}$ for some $P\in\cT$ such that $\side P=\side R$ and $|z_{3R}-v'|< 2^{-(k-1)}r_0$. It follows that $$|v-v'| < |v'-z_{3R}|+|z_{3R}-v| \leq 2^{-(k-1)}r_0+\diam 3R\leq 4\cdot 2^{-k}r_0\leq C^\star 2^{-k}r_0.$$ Thus, condition ($V_\three$) is satisfied.

Now, for each $k\geq 0$ and $v\in V_k$, let $Q_{k,v}\in\cT$ denote a dyadic cube of side length $2^{-k}$ such that $v=z_{3Q_{k,v}}$. Next, we will choose lines $\ell_{k,v}$ in $\RR^n$ and numbers $\alpha_{k,v}\geq 0$ to use with Proposition \ref{p:curve}. Let $k\geq 1$ and $v\in V_k$. By definition of $\beta^{*,c}_p(\mu,Q_{k,v})$, we can choose a line $\ell_{k,v}$ such that \begin{equation}\begin{split}\label{e:good-line}
\max\{\beta_p(\mu,3R,\ell_{k,v}):R\in \Delta^*(Q)\text{ and }\mu(3R)\geq c\diam 3R\}\\ \leq 2\max\{c^{-1/2},1\}\beta_p^{*,c}(\mu,Q_{k,v}).\end{split}\end{equation}
Note that if $x\in V_{j}\cap B(v,65C^\star 2^{-k}r_0)$ for $j=k$ or $j=k-1$, then $\mu(3Q_{j,x})\geq c\diam 3Q_{j,x}$ (since $Q_{j,x}\in\cT$) and $3Q_{j,x}\subseteq 1600\sqrt{n} Q_{k,v}$ (with room to spare). By Lemma \ref{l:lerman} and \eqref{e:good-line}, it follows that for any $x\in V_j\cap B(v,65 C^\star 2^{-k}r_0)$ with $j=k$ or $j=k-1$, \begin{equation*}\begin{split}\dist(x,\ell_{k,v}) &\leq \beta_p(\mu,3Q_{j,x},\ell_{k,v})\diam 3Q_{j,x}\\ &\leq 4\max\{c^{-1/2},1\} \beta_p^{*,c}(\mu,Q_{k,v})\cdot 2^{-k}r_0=:\alpha_{k,v} 2^{-k}r_0.\end{split}\end{equation*} Therefore, the lines $\ell_{k,v}$ and numbers $\alpha_{k,v}$ satisfy \eqref{e:alpha}. Furthermore, \begin{equation*}\begin{split}\sum_{k=1}^\infty \sum_{v\in V_k} \alpha_{k,v}^2 2^{-k}r_0&\leq 48\max\{c^{-1},1\}\sum_{Q\in\cT}\beta_{p}^{*,c}(\mu,Q)^2 \diam Q\\ &\lesssim \max\{c^{-1},1\}S_p^{*,c}(\mu,\cT)<\infty.\end{split}\end{equation*} Thus, \eqref{e:asum} holds, as well.

By Proposition \ref{p:curve}, there exists a connected, compact set $\Gamma\subseteq \RR^n$ such that $$
 \Haus^1(\Gamma) \lesssim_{n,C^\star} r_0+ \sum_{k=1}^\infty\sum_{v\in V_k} \alpha_{k,v}^2 2^{-k}r_0 \lesssim_n \diam \Top(\cT) + \max\{c^{-1},1\} S_p^{*,c}(\mu,\cT)$$ and $\Gamma \supseteq V=\lim_{k\rightarrow\infty} V_k$. By Lemma \ref{l:fund}, $\Gamma$ is a rectifiable curve. It remains to check that $\Gamma\supseteq\leaves(\cT)$. Let $y\in\leaves(\cT)$, say $y=\lim_{k\rightarrow\infty} y_k$ for some sequence of points $y_k\in \overline{Q_k}$, for some infinite branch $Q_0\supseteq Q_1\supseteq Q_2\supseteq\dots$ of $\cT$. Let $z_k=z_{3Q_k}$ denote the center of mass of $3Q_k$ and let $v_k\in V_k$ be any point which minimizes the distance to $z_k$. On one hand, $|y_k-z_k|\leq 2^{-k}r_0=\diam 3Q_k$, since $y_k,z_k\in 3Q_k$. On the other hand, $|z_k-v_k|\leq 2^{-k}r_0$ by maximality of $V_k$ in $Z_k$. Thus, by the triangle inequality, $$|v_k-y| \leq |v_k-z_k|+|z_k-y_k|+|y_k-y|\leq 2r_0\cdot 2^{-k}+|y_k-y|\quad\text{for all }k \geq 0,$$ whence $y=\lim_{k\rightarrow\infty} v_k\in\lim_{k\rightarrow\infty} V_k\subseteq \Gamma$. Therefore, since $y\in\leaves(\cT)$ was arbitrary, $\Gamma\supseteq \leaves(\cT)$.\end{proof}

The specialization to trees of lower Ahlfors regular cubes in the previous lemma can be avoided by making an assumption on the behavior of the measure in all nearby cubes. Recall from the introduction that $\beta^{**}_p(\mu,Q)=\inf_\ell\max_{R\in \Delta^*(Q)}\beta_p(\mu,3R,\ell)$.

\begin{lemma}[drawing rectifiable curves through the leaves of a tree]\label{l:draw-big} Let $n\geq 2$, let $1\leq p<\infty$, and let $\mu$ be a Radon measure on $\RR^n$. If $\cT$ is a tree of dyadic cubes such that \begin{equation}\label{e:cond2-big} S_p^{**}(\mu,\cT):=\sum_{Q\in\cT} \beta_p^{**}(\mu,Q)^2\diam Q <\infty,\end{equation} then there exists a rectifiable curve $\Gamma$ in $\RR^n$ such that $\Gamma\supseteq\leaves(\cT)$ and \begin{equation} \Haus^1(\Gamma)\lesssim_{n} \diam \Top(\cT)+S_p^{**}(\mu,\cT).\end{equation}\end{lemma}

\begin{proof} The proof follows the same pattern as the proof of Lemma \ref{l:draw}. However, instead of choosing $\ell_{k,v}$ according to \eqref{e:good-line}, one uses the definition of $\beta_p^{**}(\mu,Q)$ to choose a line $\ell_{k,v}$ such that
\begin{equation}\label{e:good-line-big}\max\{\beta_p(\mu,3R,\ell_{k,v}):R\in \Delta^*(Q)\}\\ \leq 2\max\beta_p^{**}(\mu,Q_{k,v}).\end{equation} We leave the details to the reader.  \end{proof}

\begin{corollary}\label{c:suf} Let $1\leq p<\infty$ and let $\mu$ be a finite measure on $\RR^n$ with bounded support. If $S^{**}_p(\mu)=\sum_{Q\in\Delta(\RR^n)}\beta^{**}_p(\mu,Q)^2\diam Q<\infty$, then there exists a rectifiable curve $\Gamma$ in $\RR^n$ such that $\mu(\RR^n\setminus\Gamma)=0$ and $\Haus^1(\Gamma)\lesssim_n \diam\spt\mu + S^{**}_p(\mu).$  \end{corollary}

\begin{proof} Let $\mathcal{Q}$ denote the set of maximal cubes $Q\in\Delta(\RR^n)$ such that $\mu(3Q)>0$ and $\side Q \leq \diam \spt\mu$. Note that $\#\mathcal{Q}\lesssim_n 1$. For each $Q_0\in\mathcal{Q}$, define $$\cT_{Q_0}:=\{Q\in\Delta(\RR^n):Q\subseteq Q_0\text{ and } \mu(3Q)>0\}.$$ Then $\cT_{Q_0}$ is a tree of dyadic cubes with $\leaves(\cT_{Q_0})=\overline{Q_0}\cap \spt\mu$ and $S^{**}_p(\mu,\cT_{Q_0})\leq S^{**}_p(\mu)<\infty$. By Lemma \ref{l:draw-big}, there is a rectifiable curve $\Gamma_{Q_0}$ in $\RR^n$ such that $\Gamma_{Q_0}\supseteq \overline{Q_0}\cap \spt\mu$ and \begin{equation*}\Haus^1(\Gamma_{Q_0}) \lesssim_n \diam\Top(\cT_{Q_0})+ S^{**}_p(\mu,\cT_{Q_0}) \lesssim_n \diam \spt\mu +  S^{**}_p(\mu).\end{equation*} Because $\spt\mu \subseteq \bigcup_{Q_0\in \mathcal{Q}} \left(\overline{Q_0}\cap\spt\mu\right)$ and  $\#\mathcal{Q}\lesssim_n 1$, we can find a rectifiable curve $\Gamma$ in $\RR^n$ such that $\Gamma\supseteq\bigcup_{Q_0\in\mathcal{Q}}\Gamma_{Q_0}\supseteq \spt\mu$ and $$\Haus^1(\Gamma) \lesssim_n \diam \spt\mu + \sum_{Q_0\in\mathcal{Q}}\Haus^1(\Gamma_{Q_0})\lesssim_n \diam \spt\mu + S^{**}_p(\mu).$$ Finally, note that $\mu(\RR^n\setminus \Gamma)\leq \mu(\RR^n\setminus\spt\mu)=0$.
\end{proof}

Next, we state and prove a localization lemma for measure-normalized sums over trees of dyadic cubes, which is modeled on \cite[Lemma 3.2]{BS2}. Let $\cT$ be a tree of dyadic cubes and let $b:\cT\rightarrow[0,\infty)$. For each Radon measure $\mu$ on $\RR^n$, define the \emph{$\mu$-normalized sum function} $S_{\cT,b}(\mu,x)$ by $$S_{\cT,b}(\mu,x):=\sum_{Q\in\cT} \frac{b(Q)}{\mu(Q)}\chi_Q(x)\in[0,\infty]\quad\text{for all }x\in\RR^n,$$ where we interpret $0/0=0$ and $1/0=\infty$.

\begin{lemma}[localization lemma] \label{l:localize} Let $\mathcal{T}$ be a tree of dyadic cubes, let $b:\mathcal{T}\rightarrow[0,\infty)$, and let $\mu$ be a Radon measure on $\RR^n$. For all $N<\infty$ and $\varepsilon>0$, there exist a partition of $\cT$ into a set $\good(\cT,N,\varepsilon)$ of \emph{good cubes} and a set $\bad(\cT,N,\varepsilon)$ of \emph{bad cubes} with the following the properties. \begin{enumerate}
\item Either $\good(\cT,N,\varepsilon)=\emptyset$ or $\good(\cT,N,\varepsilon)$ is a tree of dyadic cubes with $$\Top(\good(\cT,N,\varepsilon))=\Top(\cT).$$
\item Every child of a bad cube is a bad cube: if $Q$ and $R$ belong to $\cT$, $R\in \bad(\cT,N,\varepsilon)$, and $Q\subseteq R$, then $Q\in\bad(\cT,N,\varepsilon)$.
\item The sets $A:=\{x\in \Top(\cT): S_{\cT,b}(\mu,x)\leq N\}$ and  $$A':=A \setminus \bigcup_{Q\in\bad(\cT,N,\varepsilon)}Q=A\cap \leaves(\good(\cT,N,\varepsilon))$$ have comparable measure: $\mu(A')\geq (1-\varepsilon\mu(\Top(\cT)))\mu(A).$
\item The sum of the function $b$ over good cubes is finite: $$\sum_{Q\in\good(\cT,N,\varepsilon)} b(Q) < N/\varepsilon.$$
\end{enumerate}
\end{lemma}

\begin{proof} Suppose that $\cT$, $\mu$, $b$, $N$, $\varepsilon$, $A$, and $A'$  are given as above. If $\mu(A)=0$, then we may declare every dyadic cube $Q\in\cT$ to be a bad cube and the conclusion of the lemma holds trivially. Thus, suppose that $\mu(A)>0$. Declare that a dyadic cube $Q\in\cT$ is a \emph{bad cube} if there exists a dyadic cube $R\in\cT$ such that $Q\subseteq R$ and $\mu(A\cap R)\leq \varepsilon \mu(A)\mu(R)$. We call a dyadic cube $Q\in\cT$ a \emph{good cube} if $Q$ is not a bad cube. Properties (1) and (2) are immediate. To check property (3),
observe that
\begin{equation*}\begin{split} \mu(A\setminus A') &\leq \sum_{\text{maximal }Q\in\bad(\cT,N,\varepsilon)} \mu(A\cap Q)\\ &\leq \varepsilon\mu(A)\sum_{\text{maximal }Q\in\bad(\cT,N,\varepsilon)} \mu(Q) \leq \varepsilon\mu(A)\mu(\Top(\cT)),\end{split}\end{equation*} where the last inequality follows because the maximal bad cubes are pairwise disjoint. Let us emphasize that this uses our assumption that $\cT$ is composed of half open cubes. It follows that $$\mu(A')=\mu(A)-\mu(A\setminus A')\geq  (1-\varepsilon\mu(\Top(\cT)))\mu(A).$$ Thus, property (3) holds. Finally, since $S_{\cT,b}(\mu,x)\leq N$ for all $x\in A$, $$N\mu(A)
 \geq \int_A S_{\cT,b}(\mu,x)\,d\mu(x) \geq \sum_{Q\in\cT} b(Q)\, \frac{\mu(A\cap Q)}{\mu(Q)}> \varepsilon\mu(A)\sum_{Q\in\good(\cT,N,\varepsilon)} b(Q),$$ where we interpret $\mu(A\cap Q)/\mu(Q)=0$ if $\mu(Q)=0$. Because $\mu(A)>0$, it follows that $$\sum_{Q\in\good(\cT,N,\varepsilon)} b(Q)<\frac{N}{\varepsilon}.$$ This verifies property (4).
\end{proof}

\begin{remark}\label{rk:star} When $\cT\subseteq \{Q\in\Delta_1(\RR^n):Q\subseteq Q_0\}$ for some dyadic cube $Q_0$ of side length 1 and $b(Q)\equiv \beta_p^{*,c}(\mu,Q)^2\diam Q$, the function $S_{\cT,b}(\mu,x)\leq J_p^{*,c}(\mu,x)$ for all $x\in Q_0$. \end{remark}

We now have all the ingredients required to prove Theorem \ref{t:suf}.

\begin{proof}[Proof of Theorem \ref{t:suf}] Let $\mu$ be a Radon measure on $\RR^n$, let $1\leq p<\infty$, and let $c>0$. Our goal is to prove that the measure $$\mu\res\{x\in\RR^n: \lD1(\mu,x)>(3/2)\sqrt{n}\cdot c\text{ and }J^{*,c}_p(\mu,x)<\infty\}$$ is 1-rectifiable. For every $x\in \RR^n$ such that $\lD1(\mu,x)>(3/2)\sqrt{n}\cdot c$, there exists a radius $r_x>0$ such that $\mu(B(x,r))\geq 3\sqrt{n}\cdot cr$ for all $r\leq r_x$, because $\lD1(\mu,x)=\liminf_{r\rightarrow 0} \mu(B(x,r))/2r$. Hence $$\mu(3Q) \geq \mu(B(x,\side Q)) \geq 3\sqrt{n}\cdot c\side Q = c\diam 3Q$$ whenever $\lD1(\mu,x)>(3/2)\sqrt{n}\cdot c$ and $Q$ is a dyadic cube such that $x\in Q$ and $\side Q\leq r_x$. Thus, if $\lD1(\mu,x)>(3/2)\sqrt{n}\cdot c$, then $x$ belongs to the leaves of the tree \begin{equation*}\begin{split}
\cT_x:=\Big\{Q\in\Delta_1(\RR^n):\ &Q\subseteq Q_x\text{ and }\\
&\mu(3R)\geq c\diam 3R\text{ for all $R\in\Delta(\RR^n)$ such that }Q\subseteq R\subseteq Q_x \Big\},\end{split}\end{equation*} where $Q_x$ is defined to be the maximal dyadic cube containing $x$ with $\side Q_x\leq\min\{r_x,1\}.$ In fact, note that $x\in\Top(\cT_x)\cap \leaves(\cT_x)$ and the tree $\cT_x$ satisfies condition \eqref{e:cond1} of Lemma \ref{l:draw}. Because each tree  $\cT_x$ is determined by $Q_x$ and $\Delta_1(\RR^n)$ is countable, the collection $\{\cT_x:\lD1(\mu,x)>(3/2)\sqrt{n}\cdot c\}$ of trees is enumerable, say $$\{\cT_x:\lD1(\mu,x)>(3/2)\sqrt{n}\cdot c\}=\{\cT_{x_i}:i=1,2,\dots\}$$ for some sequence of points such that $\lD1(\mu,x_i)>(3/2)\sqrt{n}\cdot c$ for all $i\geq 1$.
Therefore, since \begin{align*} \{x\in\RR^n:\lD1(\mu,x)&>(3/2)\sqrt{n}\cdot c\text{ and }J^{*,c}_p(\mu,x)<\infty\} \\
&\subseteq \bigcup_{i=1}^\infty \bigcup_{j=1}^\infty \{x\in\Top(\cT_{x_i}):J^{*,c}_p(\mu,x)\leq j\},\end{align*} it suffices to prove that the measure $\mu\res A_{y,N}$ is 1-rectifiable for all $y\in\RR^n$ such that $\lD1(\mu,y)>(3/2)\sqrt{n}\cdot c$ and for all $N<\infty$, where $A_{y,N}:=\{x\in \Top(\cT_y):J^{*,c}_p(\mu,x)\leq N\}.$

Fix $y\in\RR^n$ such that $\lD1(\mu,y)>(3/2)\sqrt{n}\cdot c$ and fix $N<\infty$. Set $\eta_y:=\mu(\Top(\cT_y))$. Given $0<\varepsilon <\eta_y$, let $\cT_{y,N,\varepsilon}:= \good(\cT_y,N,\varepsilon)\subseteq\cT_y$ denote the tree given by Lemma \ref{l:localize} applied with $\cT=\cT_y$ and $b(Q)\equiv \beta^*_p(\mu,Q)^2\diam Q$ (see Remark \ref{rk:star}). Then $\cT_{y,N,\varepsilon}$ inherits property \eqref{e:cond1} from $\cT_y$ and, by Lemma \ref{l:localize}, $$S_p^{*,c}(\mu,\cT_{y,N,\varepsilon})<\frac{N}{\varepsilon}$$ and \begin{equation*} \mu(A_{y,N}\cap \leaves(\cT_{y,N,\varepsilon}))\geq (1- \varepsilon\eta_y)\mu(A_{y,N}).\end{equation*} Thus, by Lemma \ref{l:draw}, there exists a rectifiable curve $\Gamma_{y,N,\varepsilon}$ in $\RR^n$ such that $\Gamma_{y,N,\varepsilon}\supseteq\leaves(\cT_{y,N,\varepsilon})$. In particular, $\Gamma_{y,N,\varepsilon}$ captures a large portion of the mass of $A_{y,N}$: $$\mu(A_{y,N}\setminus \Gamma_{y,N,\varepsilon})\leq \mu(A_{y,N})-\mu(A_{y,N}\cap \Gamma_{y,N,\varepsilon})\leq \mu(A_{y,N}) - (1-\varepsilon \eta_y)\mu(A_{y,N})= \varepsilon\eta_y\mu(A_{y,n}).$$ (Of course, Lemma \ref{l:draw} also gives a quantitative bound on the length of $\Gamma_{y,N,\varepsilon}$ depending only on $n$, $c$, $N$, $\varepsilon$, and $\diam\Top(\cT_y)$,  but we do not need it here.) To finish, choose $0<\varepsilon_k<\eta_y$ for all $k\geq 1$ so that $\lim_{k\rightarrow\infty} \varepsilon_k=0$. Then $$\mu\left(A_{y,N}\setminus \bigcup_{k=1}^\infty \Gamma_{y,N,\varepsilon_k}\right) \leq \inf_{k\geq 1} \mu(A_{y,N}\setminus \Gamma_{y,N,\varepsilon_k}) \leq \eta_y\mu(A_{y,N})\inf_{k\geq 1} \varepsilon_k = 0.$$ Therefore, $\mu\res A_{y,N}$ is 1-rectifiable. As noted above, this completes the proof.
\end{proof}

\begin{remark}\label{r:sufficient} By substituting Lemma \ref{l:draw} in the proof of Theorem \ref{t:suf} with Lemma \ref{l:draw-big}, one can verify that if $\mu$ is a Radon measure on $\RR^n$ and $1\leq p<\infty$, then the measure $$\mu\res\{x\in\RR^n: J_p^{**}(\mu,x)<\infty\}$$ is 1-rectifiable, where  $$J_p^{**}(\mu,x):=\sum_{Q\in\Delta_1(\RR^n)} \beta_p^{**}(\mu,Q)^2\frac{\diam Q}{\mu(Q)}\chi_Q(x)$$ is the density-normalized Jones function associated with the numbers $\beta_p^{**}(\mu,x)$. However, see Remark \ref{r:not-necessary}.\end{remark}

\section{Proof of Theorems \ref{t:big} and \ref{t:tstm}}\label{s:abcd}

\begin{proof}[Proof of Theorem \ref{t:big}] Let $\mu$ be a Radon measure on $\RR^n$ and let $1\leq p\leq 2$. Partition $\RR^n$ into two sets $$R=\{x\in\RR^n:\lD1(\mu,x)>0 \text{ and }J_p^*(\mu,x)<\infty\}$$ and $$P=\{x\in\RR^n: \lD1(\mu,x)=0\text{ or }J_p^*(\mu,x)=\infty\},$$ which are easily verified to be Borel. Since $\RR^n=R\cup P$ and $R\cap P=\emptyset$, we have $$\mu=(\mu\res R) + (\mu\res P),\quad (\mu\res R)\perp(\mu\res P).$$ By uniqueness of the decomposition $\mu=\mu^1_{\rect}+\mu^1_{\pu}$ in Proposition \ref{p:decomp}, to show that $\mu^1_{\rect}=\mu\res R$ and $\mu^1_{\pu}=\mu\res P$ it suffices to prove that the measure $\mu\res R$ is $1$-rectifiable and the measure $\mu\res P$ is purely 1-unrectifiable.

On one hand, since $J^{*,c}_p(\mu,x) \leq J^*_p(\mu,x)$ for all $x\in\RR^n$ and for all $c>0$, \begin{equation}\begin{split} \label{e:bigcup}
&R=\left\{x\in\RR^n: \lD{1}(\mu,x)>0\text{ and }J^*_p(\mu,x)<\infty\right\}\\
  &\qquad \subseteq \bigcup_{i=1}^\infty \left\{x\in\RR^n:\lD1(\mu,x)> (3/2)\sqrt{n}/i\text{ and } J^{*,1/i}_p(\mu,x)<\infty\right\}.\end{split}\end{equation} Therefore, $\mu\res R$ is 1-rectifiable by \eqref{e:bigcup} and Theorem \ref{t:suf}. On the other hand, because $J^*_p(\mu,x)$ is increasing in $p$ and $p\leq 2$, $$\mu\res P \leq \mu\res \{x\in\RR^n:\lD1(\mu,x)=0\}+\mu\res\{x\in\RR^n: J^*_2(\mu,x)=\infty\}.$$ Therefore, $\mu\res P$ is purely 1-unrectifiable by Lemma \ref{l:spu} and Theorem \ref{t:nec}.
\end{proof}

\begin{proof}[Proof of Corollaries \ref{t:rect} and \ref{t:pu}] A measure $\mu$ is 1-rectifiable if and only if $\mu^1_{\pu}=0$, and a measure $\mu$ is purely 1-unrectifiable if and only if $\mu^1_{\rect}=0$. Therefore, Corollary \ref{t:rect} and Corollary \ref{t:pu} follow immediately from Theorem \ref{t:big}.\end{proof}

\begin{proof}[Proof of Theorem \ref{t:tstm}] Let $n\geq 2$ and let $1\leq p<\infty$. Let $\mu$ be a finite Borel measure with bounded support. To prove the first statement, suppose that there exists a rectifiable curve $\Gamma$ such that $\mu(\RR^n\setminus\Gamma)=0$. Then $\spt\mu\subseteq\Gamma$, since $\Gamma$ is closed. For every $Q\in\Delta(\RR^n)$, let $\ell_{Q}$ be any line such that $$\dist(x,\ell_{Q}) \leq\beta_{\spt\mu}(1600\sqrt{n} Q)\diam 1600\sqrt{n} Q\quad\text{for all }x\in \spt\mu\cap 1600\sqrt{n}Q.$$ Then for every dyadic cube $Q\in\Delta(\RR^n)$ and nearby cube $R\in\Delta^*(Q)$, $$\beta_p(\mu,3R,\ell_Q)^p = \int_{3R}\left(\frac{\dist(x,\ell_Q)}{\diam 3R}\right)^p \frac{d\mu(x)}{\mu(3R)}
\leq \left(\frac{1600\sqrt{n}}{3}\right)^p\beta_{\spt \mu}(1600\sqrt{n} Q)^p.$$ Hence, since $3R\subseteq 1600\sqrt{n} Q$ for all $R\in\Delta^*(Q)$, $$\beta_p^{**}(\mu,Q) \leq \max\{\beta_p(\mu,3R,\ell_Q):R\in\Delta^*(Q)\}\lesssim_n \beta_{\spt\mu}(1600\sqrt{n}Q)\quad\text{for all }Q\in\Delta(\RR^n).$$ Therefore, \begin{equation*}S_p^{**}(\mu)=\sum_{Q\in\Delta(\RR^n)}\beta^{**}_p(\mu,Q)^2\diam Q \lesssim_n\sum_{Q\in\Delta(\RR^n)} \beta_{\spt\mu}(1600\sqrt{n}Q)^2\diam Q \lesssim_{n} \Haus^1(\Gamma)\end{equation*} by Corollary \ref{c:tst}.

The second statement is given by Corollary \ref{c:suf}.
\end{proof}

\section{Variations for pointwise doubling measures}\label{s:doubling}

In the forerunner \cite{BS} to this paper, the authors gave a necessary condition for a Radon measure on $\RR^n$ to be 1-rectifiable using the $L^2$ density-normalized Jones function $\tJ_2(\mu,x)$ (see \eqref{e:tJ}).

\begin{theorem}[{\cite[Theorem A]{BS}}] \label{t:old-nec} If $\mu$ is a 1-rectifiable Radon measure on $\RR^n$, then $\tJ_2(\mu,x)<\infty$ at $\mu$-a.e.~ $x\in\RR^n$. \end{theorem}

Examining the proof of Theorem \ref{t:old-nec} in \cite{BS}, one deduces that $\int_E\tJ_2(\mu,x)\,d\mu(x)<\infty$ for every rectifiable curve $\Gamma\subseteq\RR^n$ and for every Borel set $E\subseteq\RR^n$ of the form $$E=\{x\in\Gamma:\mu(B(x,r))\geq cr\text{ for all }0<r\leq r_0\}\quad\text{for some }c>0\text{ and } r_0>0.$$ Thus, the proof of Theorem \ref{t:old-nec} yields the following stronger formulation of the theorem.

\begin{theorem}\label{t:puds} Let $\mu$ be a Radon measure on $\RR^n$. Then $$\mu\res\{x\in\RR^n: \lD1(\mu,x)>0\text{ and }\tJ_2(\mu,x)=\infty\}$$ is purely 1-unrectifiable.\end{theorem}

We now give a second application of Proposition \ref{p:curve} and the tools of \S\ref{s:suf}, which in combination with Lemma \ref{l:spu} and Theorem \ref{t:puds}, provides characterizations in terms of $\tJ_p(\mu,x)$ of 1-rectifiable and purely 1-unrectifiable pointwise doubling measures.

Let  $Q^\uparrow\in\Delta(\RR^n)$ denote the \emph{parent} of $Q\in\Delta(\RR^n)$. That is, let $Q^\uparrow$ denote the unique dyadic cube such that $Q^\uparrow\supseteq Q$ and $\side Q^\uparrow = 2\side Q$.

\begin{lemma}[drawing rectifiable curves through the leaves of uniformly doubling trees] \label{l:draw2} Let $n\geq 2$, let $1\leq p<\infty$, and let $0<D<\infty$. Let $\mu$ be a Radon measure on $\RR^n$. If $\cT$ is a tree of dyadic cubes such that \begin{equation}\label{e:dc1}\mu(3Q^\uparrow)\leq 2^D\mu(3Q)\quad\text{for all }Q\in\cT\setminus\Top(\cT), \text{ and}\end{equation} \begin{equation}\label{e:Sp} S_p(\mu,\cT):=\sum_{Q\in\cT} \beta_p(\mu,3Q)^2\diam Q <\infty,\end{equation} then there exists a rectifiable curve $\Gamma$ in $\RR^n$ such that $\Gamma\supseteq\leaves(\cT)$ and \begin{equation} \Haus^1(\Gamma)\lesssim_{n} \diam \Top(\cT)+(3200\sqrt{n})^{2D/p} S_p(\mu,\mathcal{T}).\end{equation}\end{lemma}

\begin{proof} As in the proof of Lemma \ref{l:draw}, we may assume without loss of generality that $\Top(\cT)=[0,1)^n$ and every cube $Q\in\cT$ belongs to an infinite branch of $\cT$. Set parameters $C^\star=4$ and $r_0=3\sqrt{n}$. For each $Q\in\cT$, let $z_{3Q}=\mu(3Q)^{-1}\int_{3Q} z\,d\mu(z)$ denote the center of mass of $3Q$. For each $k\geq 0$, let $Z_k= \{z_{3Q}:Q\in\cT\text{ and }\side Q= 2^{-k}\}$ and choose $V_k$ to be any maximal $2^{-k}r_0$-separated subset of $Z_k$. Pick any $x_0\in 3Q_0$. Then \begin{equation}\label{e:containers} V_k\subseteq Z_k\subseteq 3\Top(\cT)\subseteq B(x_0,r_0)\subseteq B(x_0,C^\star r_0).\end{equation} The set $V_k$ satisfies condition $(V_I)$ of Proposition \ref{p:curve} by definition. The set $V_k$ also satisfies conditions $(V_\two)$ and $(V_\three)$ of Proposition \ref{p:curve} (see the proof of Lemma \ref{l:draw}).

For each $k\geq 0$ and $v\in V_k$, choose a dyadic cube $Q_{k,v}\in\cT$ such that $v=z_{3Q_{k,v}}$. Then, for each $k\geq 1$ and $v\in V_k$, choose a minimal dyadic cube  $\widehat{Q}_{k,v}\in\cT$ such that $\widehat{Q}_{k,v}\supseteq Q_{k,v}$ and such that $3\widehat{Q}_{k,v}$ contains $3Q_{j,v'}$ for every $j\in \{k-1,k\}$ and $v'\in V_j\cap B(v,65 C^\star 2^{-k}r_0)$. Since $65 C^\star 2^{-k}r_0 = 780\sqrt{n} 2^{-k}$, we obtain (the overestimate) \begin{equation}\label{e:hat-scale}\frac{\side \widehat{Q}_{k,v}}{\side Q_{j,v'}}=\frac{\diam 3\widehat{Q}_{k,v}}{\diam 3Q_{j,v'}} \leq 2^{\lceil \log_2 1600\sqrt{n}\rceil}<3200\sqrt{n}\end{equation} for all $j\in\{k-1,k\}$ and $v'\in V_j\cap B(v,65C^\star 2^{-k}r_0)$. Thus, by the doubling condition, \begin{equation}\label{e:mu-scale} \frac{\mu(\widehat{Q}_{k,v})}{\mu(Q_{j,v'})} \leq 2^{D\lceil \log_2 1600\sqrt{n}\rceil} < (3200\sqrt{n})^D\end{equation} for all $j\in\{k-1,k\}$ and $v'\in V_j\cap B(v,65C^\star 2^{-k}r_0)$.

We are ready to pick lines $\ell_{k,v}$ and numbers $\alpha_{k,v}\geq 0$ to use in Proposition \ref{p:curve}. Let $k\geq 1$ and let $v\in V_k$. Choose $\ell_{k,v}$ to be any straight line in $\RR^n$ such that $$\beta_p(\mu,3\widehat{Q}_{k,v},\ell_{k,v}) \leq 2 \beta_p(\mu,3\widehat{Q}_{k,v}).$$ Then, by \eqref{e:hat-scale} and \eqref{e:mu-scale},
\begin{equation}\begin{split} \label{e:alpha-hat}\beta_p(\mu,3&Q_{j,v'},\ell_{k,v})\diam 3Q_{j,v'} \\ &\leq 2\,\frac{\diam 3\widehat{Q}_{k,v}}{\diam 3Q_{j,v'}}
\left(\frac{\mu(\widehat{Q}_{k,v})}{\mu(Q_{j,v'})}\right)^{1/p}\beta_p(\mu,3\widehat{Q}_{k,v},\ell_{k,v})\diam 3Q_{k,v}\\
&\leq 6400\sqrt{n}\left(3200\sqrt{n}\right)^{D/p}\beta_p(\mu,3\widehat{Q}_{k,v})\diam 3Q_{k,v}=:\alpha_{k,v} 2^{-k}r_0\end{split}\end{equation} for all $j\in\{k-1,k\}$ and $v'\in V_j\cap B(v,65C^\star 2^{-k}r_0)$. Hence condition \eqref{e:alpha} of Proposition \ref{p:curve} holds by Lemma \ref{l:lerman} and \eqref{e:alpha-hat}. Next, note that each cube $Q\in\cT$ appears as $\widehat{Q}_{k,v}$ for at most $C(n)$ pairs $(k,v)$ by \eqref{e:hat-scale}. Thus, condition \eqref{e:asum} of Proposition \ref{p:curve} holds by \eqref{e:Sp}.

By Proposition \ref{p:curve}, there exists a connected, compact set $\Gamma\subseteq \RR^n$ such that $$
 \Haus^1(\Gamma) \lesssim_{n,C^\star} r_0+ \sum_{k=1}^\infty\sum_{v\in V_k} \alpha_{k,v}^2 2^{-k}r_0 \lesssim_n \diam \Top(\cT) + (3200\sqrt{n})^{2D/p} S_p(\mu,\cT)$$ and $\Gamma \supseteq V=\lim_{k\rightarrow\infty} V_k$. By Lemma \ref{l:fund}, $\Gamma$ is a rectifiable curve. Finally, as in the proof of Lemma \ref{l:draw}, $\Gamma\supseteq\leaves(\cT)$.
\end{proof}

\begin{theorem} \label{t:suf2} Let $\mu$ be a Radon measure on $\RR^n$ and let $1\leq p<\infty$. Then the measure $$\mu\res\left\{x\in\spt\mu: \limsup_{r\downarrow 0} \frac{\mu(B(x,2r))}{\mu(B(x,r))}<\infty\text{ and } \tJ_p(\mu,x)<\infty\right\}$$ is 1-rectifiable. \end{theorem}

\begin{proof}
Let $\mu$ be a Radon measure on $\RR^n$ and let $1\leq p<\infty$. Write $$\double(\mu,x):=\limsup_{r\downarrow 0}\frac{\mu(B(x,2r))}{\mu(B(x,r))}\in[0,\infty]\quad\text{for all }x\in\spt\mu.$$ For every $x\in\spt\mu$ such that $\double(\mu,x)<\infty$, there exists an integer $1\leq D_x<\infty$ and $r_x>0$ such that $\mu(B(x,2r))\leq 2^{D_x}\mu(B(x,r))$ for all $0<r\leq r_x$. Hence $$\mu(3Q^\uparrow) \leq \mu(B(x,\diam 3Q^\uparrow)) \leq 2^{D_x\lceil \log_2 6\sqrt{n}\rceil} \mu(B(x,\side Q)) \leq 2^{D_x\lceil\log_2 6\sqrt{n}\rceil} \mu(3Q)$$ for every dyadic cube $Q\in\Delta(\RR^n)$ containing $x$ such that $2^{\lceil\log_2 6\sqrt{n}\rceil-1}\side Q \leq r_x$. Thus, if $\double(\mu,x)<\infty$, then $x$ belongs to the leaves of the tree \begin{equation*}\begin{split}
\cT_x:=\Big\{&Q\in\Delta_1(\RR^n):Q\subseteq Q_x\text{ and }\\
&\mu(3R^\uparrow) \leq (12\sqrt{n})^{D_x} \mu(3R) \text{ for all $R\in\Delta(\RR^n)$ such that }Q\subseteq R\subseteq Q_x \Big\},\end{split}\end{equation*} where $Q_x$ is defined to be the maximal dyadic cube containing $x$ with $$\side Q_x\leq\min\{r_x/2^{\lceil \log_2 6\sqrt{n}\rceil-1},1\}.$$ In fact, note that $x\in\Top(\cT_x)\cap \leaves(\cT_x)$ and the tree $\cT_x$ satisfies condition \eqref{e:dc1} of Lemma \ref{l:draw2}. Because each tree  $\cT_x$ is determined by $Q_x$ and $D_x$, and $\Delta_1(\RR^n)$ and $\NN$ are countable, the collection $\{\cT_x:\double(\mu,x)<\infty\}$ of trees is enumerable, say $$\{\cT_x:\double(\mu,x)<\infty\}=\{\cT_{x_i}:i=1,2,\dots\}$$ for some sequence of points such that $\double(\mu,x_i)<\infty$ for all $i\geq 1$.
Therefore, since \begin{align*} \{x\in\RR^n:\double(\mu,x)&<\infty\text{ and }\tJ_p(\mu,x)<\infty\} \\
&\subseteq \bigcup_{i=1}^\infty \bigcup_{j=1}^\infty \{x\in\Top(\cT_{x_i}):\tJ_p(\mu,x)\leq j\},\end{align*} it suffices to prove that the measure $\mu\res A_{y,N}$ is 1-rectifiable for all $y\in\RR^n$ such that $\double(\mu,y)<\infty$ and for all $N<\infty$, where $A_{y,N}:=\{x\in \Top(\cT_y):\tJ_p(\mu,x)\leq N\}.$

Fix $y\in\RR^n$ such that $\double(\mu,y)<\infty$ and fix $N<\infty$. Set $\eta_y:=\mu(\Top(\cT_y))$. Given $0<\varepsilon <\eta_y$, let $\cT_{y,N,\varepsilon}:= \good(\cT_y,N,\varepsilon)\subseteq\cT_y$ denote the tree given by Lemma \ref{l:localize} applied with $\cT=\cT_y$ and $b(Q)\equiv \beta_p(\mu,3Q)^2\diam Q$. Then $\cT_{y,N,\varepsilon}$ inherits property \eqref{e:dc1} from $\cT_y$ and, by Lemma \ref{l:localize}, $$S_p(\mu,\cT_{y,N,\varepsilon})<\frac{N}{\varepsilon}$$ and \begin{equation*} \mu(A_{y,N}\cap \leaves(\cT_{y,N,\varepsilon}))\geq (1- \varepsilon\eta_y)\mu(A_{y,N}).\end{equation*} Thus, by Lemma \ref{l:draw2}, there exists a rectifiable curve $\Gamma_{y,N,\varepsilon}$ in $\RR^n$ such that $\Gamma_{y,N,\varepsilon}\supseteq\leaves(\cT_{y,N,\varepsilon})$. In particular, $\Gamma_{y,N,\varepsilon}$ captures a large portion of the mass of $A_{y,N}$: $$\mu(A_{y,N}\setminus \Gamma_{y,N,\varepsilon})\leq \mu(A_{y,N})-\mu(A_{y,N}\cap \Gamma_{y,N,\varepsilon})\leq \mu(A_{y,N}) - (1-\varepsilon \eta_y)\mu(A_{y,N})= \varepsilon\eta_y\mu(A_{y,n}).$$ To finish, choose $0<\varepsilon_k<\eta_y$ for all $k\geq 1$ so that $\lim_{k\rightarrow\infty} \varepsilon_k=0$. Then $$\mu\left(A_{y,N}\setminus \bigcup_{k=1}^\infty \Gamma_{y,N,\varepsilon_k}\right) \leq \inf_{k\geq 1} \mu(A_{y,N}\setminus \Gamma_{y,N,\varepsilon_k}) \leq \eta_y\mu(A_{y,N})\inf_{k\geq 1} \varepsilon_k = 0.$$ Therefore, $\mu\res A_{y,N}$ is 1-rectifiable. As noted above, this completes the proof.
\end{proof}

We now have all the necessary components to prove Theorem \ref{t:pd-big}.

\begin{proof}[Proof of Theorem \ref{t:pd-big}] Let $\mu$ is a pointwise doubling measure on $\RR^n$ and let $1\leq p\leq 2$. Partition $\RR^n$ into two sets, $R=\{x\in\RR^n:\tJ_p(\mu,x)<\infty\}$ and $P=\{x\in\RR^n: \tJ_p(\mu,x)=\infty\}$. It is a standard exercise to show that $R$ and $P$ are Borel sets. Since $\RR^n=R\cup P$ and $R\cap P=\emptyset$, we have $\mu=(\mu\res R) + (\mu\res P)$, $(\mu\res R)\perp(\mu\res P).$ By uniqueness of the decomposition $\mu=\mu^1_{\rect}+\mu^1_{\pu}$ in Proposition \ref{p:decomp}, we can prove $\mu^1_{\rect}=\mu\res R$ and $\mu^1_{\pu}=\mu\res P$ by demonstrating that $\mu\res R$ is $1$-rectifiable and $\mu\res P$ is purely 1-unrectifiable. On one hand, since $\mu$ is pointwise doubling, $$\mu\res R=\mu\res \left\{x\in\spt\mu: \limsup_{r\downarrow 0} \frac{\mu(B(x,2r))}{\mu(B(x,r))}<\infty\text{ and } \tJ_p(\mu,x)<\infty\right\}.$$ Thus, $\mu\res R$ is 1-rectifiable by Theorem \ref{t:suf2}. On the other hand, because $\tJ_p(\mu,x)$ is increasing the exponent $p$, $1\leq p\leq 2$, and $$\mu\res P\leq \mu\res\{x\in\RR^n:\lD1(\mu,x)=0\}+\mu\res\{x\in\RR^n:\lD1(\mu,x)>0\text{ and }\tJ_2(\mu,x)=\infty\},$$ the measure $\mu\res P$ is purely 1-unrectifiable by Lemma \ref{l:spu} and Theorem \ref{t:puds}. Therefore, $\mu^1_{\rect}=\mu\res R$ and $\mu^1_{\pu}=\mu\res P$.
\end{proof}

\section{Drawing rectifiable curves I: description of the curves and connectedness}\label{s:proof26}

The goal of this and the next section is to prove Proposition \ref{p:curve}, which for the reader's convenience we now restate.

\begin{proposition} \label{p:curve2} Let $n\geq 2$, let $C^\star>1$, let $x_0\in\RR^n$, and let $r_0>0$. Let $(V_k)_{k=0}^\infty$ be a sequence of nonempty finite subsets of $B(x_0,C^\star r_0)$ such that
\begin{enumerate}
\item[($V_I$)] distinct points $v,v'\in V_k$ are uniformly separated: $|v-v'|\geq 2^{-k}r_0$;
\item[($V_\two$)] for all $v_k\in V_k$, there exists $v_{k+1}\in V_{k+1}$ such that $|v_{k+1}-v_k|< C^\star 2^{-k}r_0$; and,
\item[($V_\three$)] for all $v_{k}\in V_{k}$ ($k\geq 1$), there exists $v_{k-1}\in V_{k-1}$ such that $|v_{k-1}-v_k|< C^\star 2^{-k}r_0$.
\end{enumerate} Suppose that for all $k\geq 1$ and for all $v\in V_k$ we are given a straight line $\ell_{k,v}$ in $\RR^n$ and a number $\alpha_{k,v}\geq 0$ such that
 \begin{equation}\label{e:alpha2}\sup_{x\in (V_{k-1}\cup V_k)\cap B(v, 65 C^\star 2^{-k}r_0)} \dist(x,\ell_{k,v}) \leq \alpha_{k,v} 2^{-k}r_0\end{equation} and \begin{equation}\label{e:asum2}\sum_{k=1}^\infty\sum_{v\in V_k} \alpha_{k,v}^2 2^{-k} r_0<\infty.\end{equation}
Then the sets $V_k$ converge in the Hausdorff metric to a compact set $V\subseteq \overline{B(x_0,C^\star r_0)}$ and
there exists a compact, connected set $\Gamma\subseteq\overline{B(x_0,C^\star r_0)}$ such that $\Gamma\supseteq V$ and
\begin{align}\label{e:G-length2}
\Haus^1(\Gamma)\lesssim_{n,C^\star} r_0 + \sum_{k=1}^\infty\sum_{v\in V_k} \alpha_{k,v}^2 2^{-k}r_0.
\end{align}
\end{proposition}

By viewing $\bigcup_{k=0}^\infty V_k$ as vertices of an abstract tree $\cT$, where each vertex $v\in V_{k+1}$ is connected by an edge to a nearest vertex in $V_k$, one may view Proposition \ref{p:curve2} as a criterion for being able to draw a rectifiable curve $\Gamma$ (i.e.~ a connected, compact set $\Gamma$ with $\Haus^1(\Gamma)<\infty$) through the leaves $V=\lim_{k\rightarrow\infty} V_k$ of $\cT$. Convergence of the sets $V_k$ in the Hausdorff metric is guaranteed by Lemma \ref{l:V}, whose proof we defer to Appendix \ref{ss:lemmas}.

\begin{lemma}\label{l:V} Let $B\subseteq\RR^n$ be a bounded set and let $V_0,V_1,\dots$ be a sequence of nonempty finite subsets of $B$. If the sequence satisfies ($V_\three$) for some $C^\star>0$ and $r_0>0$, then $V_k$ converges in the Hausdorff metric to a compact set $V\subseteq\overline{B}$.\end{lemma}

The power of 2 in the quantity $\alpha_{k,v}^2$ in Proposition \ref{p:curve2} is a consequence of the following application of the Pythagorean formula. For a proof of Lemma \ref{l:graph}, see Appendix \ref{ss:lemmas}.

\begin{lemma}\label{l:graph} Suppose that $V\subseteq\RR^n$ is a $1$-separated set with $\#V\geq 2$ and there exist lines $\ell_1$ and $\ell_2$ and a number $0\leq \alpha\leq 1/16$ such that $$\dist(v,\ell_i)\leq \alpha\quad\text{for all $v\in V$ and $i=1,2$}.$$ Let $\pi_i$ denote the orthogonal projection onto $\ell_i$. There exist compatible identifications of $\ell_1$ and $\ell_2$ with $\RR$ such that $\pi_1(v')\leq \pi_1(v'')$ if and only if $\pi_2(v')\leq \pi_2(v'')$ for all $v',v''\in V$. If $v_1$ and $v_2$ are consecutive points in $V$ relative to the ordering of $\pi_1(V)$, then \begin{equation}\label{e:graph1} \Haus^1([u_1,u_2]) < (1+ 3\alpha^2)\cdot \Haus^1([\pi_1(u_1),\pi_1(u_2)]) \quad\text{for all }[u_1,u_2]\subseteq[v_1,v_2].\end{equation}
Moreover, \begin{equation}\label{e:graph2} \Haus^1([y_1,y_2]) < (1+12\alpha^2)\cdot\Haus^1([\pi_1(y_1),\pi_1(y_2)])\quad\text{for all }[y_1,y_2]\subseteq\ell_2.\end{equation}
\end{lemma}

In \S\S \ref{ss:overview}--\ref{ss:connected} and \S \ref{ss:length}, we prove Proposition \ref{p:curve2} assuming Lemmas \ref{l:V} and \ref{l:graph}. To begin, in \S \ref{ss:overview}, we make some reductions and give a high level overview of the proof of the proposition. Next, in \S \ref{ss:curves}, we give a self-contained description of rectifiable curves $\Gamma_k$ that contain $V_k$ and converge in the Hausdorff metric to the curve $\Gamma$ in the statement of the proposition. By construction, the sets $\Gamma_k$ are evidently closed. In \S \ref{ss:connected}, we verify that the sets $\Gamma_k$ are connected. In \S\S\ref{ss:phantom}--\ref{ss:sum-b} of the next section, we make detailed estimates on the length of $\Gamma_k$, which yield the estimate \eqref{e:G-length2} on the length of $\Gamma$. Finally, to complete the proof of Proposition \ref{p:curve2}, we supply proofs of Lemmas \ref{l:V} and \ref{l:graph} in Appendix \ref{ss:lemmas}.

\subsection{Overview of the proof of Proposition \ref{p:curve2}}\label{ss:overview}
By scale invariance, it suffices to prove the proposition with $r_0=1$. Let $n\geq 2$ and $C^\star>1$ be given, let $x_0\in\RR^n$, let $r_0=1$, and assume that $V_0,V_1,V_2,\dots$ is a sequence of nonempty finite sets in $B(x_0,C^\star)$ satisfying ($V_I$), ($V_\two$), ($V_\three$). By Lemma \ref{l:V}, there exists a compact set $V\subseteq \overline{B(x_0,C^\star)}$ such that $V_{k}$ converges to $V$ in the Hausdorff metric as $k\rightarrow\infty$. Suppose that for all $k\geq 1$ and $v\in V_k$ we are given a straight line $\ell_{k,v}$ in $\RR^n$ and a number $\alpha_{k,v}\geq 0$ satisfying \eqref{e:alpha2} and \eqref{e:asum2}. If $\#V_k=1$ for infinitely many $k$, then $V$ is a singleton and the conclusion is trivial. Thus, we shall assume that $\#V_k\geq 2$ for all sufficiently large $k$. Let $k_0\geq 0$ be the least index such that $\#V_k\geq 2$ for all $k\geq k_0$.

To complete the proof, we will construct a sequence $\Gamma_{k_0}, \Gamma_{k_0+1},\Gamma_{k_0+2},\dots$ of closed, connected subsets of $\overline{B(x_0,C^\star)}$ such that $\Gamma_k\supseteq V_k$ and \begin{equation}\label{e:goal} \Haus^1(\Gamma_k) \leq C\left(2^{-k_0} + \sum_{j={k_0+1}}^k\sum_{v\in V_j} \alpha_{j,v}^2 2^{-j} \right)\quad\text{for all }k\geq k_0+1,\end{equation} where $C>1$ depends only on $n$ and $C^\star$. By Corollary \ref{c:fund}, there exists a compact, connected set $\Gamma$ and a subsequence $(\Gamma_{k_j})_{j=1}^\infty$ of $(\Gamma_k)_{k=k_0}^\infty$ such that $\Gamma_{k_j}\rightarrow \Gamma$ in the Hausdorff metric as $j\rightarrow\infty$ and $\Gamma$ satisfies \eqref{e:G-length2} with $r_0=1$ and implicit constant $32C$. We note that $V\subseteq \Gamma\subseteq\overline{B(x_0,C^\star)}$, because $V_{k_j}\subseteq \Gamma_{k_j}\subseteq \overline{B(x_0,C^\star)}$, $V_{k_j}\rightarrow V$, and $\Gamma_{k_j}\rightarrow \Gamma$.

In the argument that follows, the points in $\bigcup_{k=k_0}^\infty V_k$ are called \emph{vertices}. A vertex $x\in V_k$ is said to belong to \emph{generation $k$}. Property ($V_I$) states that vertices of the same generation are uniformly separated. Property ($V_\two$) ensures that every vertex is relatively close to some vertex of the next generation. And property ($V_\three$) guarantees that every vertex of generation $k\geq {k_0+1}$ is relatively close to some vertex of the previous generation.  By associating each vertex to a nearest vertex of the previous generation, the set of all vertices may be viewed as a tree with $\#V_{k_0}$ roots.

\subsection{Description of the curves}\label{ss:curves}
Each curve $\Gamma_k$ will be defined to be the union of finitely many closed line segments $[v',v'']$ (``edges") between vertices $v',v''\in V_k$ and closed sets  $B[j,w',w'']$ (``bridges") that connect vertices $w',w''\in V_j$ for some $k_0\leq j\leq k$ and pass through vertices of generation $j'$ nearby $w'$ and $w''$ for every $j'>j$. Bridges will be frozen in the sense that once a bridge appears in some $\Gamma_k$, the bridge remains in $\Gamma_{k'}$ for all $k'\geq k$. If an edge $[v',v'']$ is included in $\Gamma_k$, then $|v-v''|<30 C^\star 2^{-k}$. If a bridge $B[k,v',v'']$ is included in $\Gamma_k$, then $|v'-v''|\geq 30 C^\star 2^{-k}$.

The precise construction depends on a few auxiliary choices. First, choose a small parameter $0<\varepsilon\leq 1/32$ so that the conclusions of Lemma \ref{l:graph} hold for $\alpha=2\varepsilon$. Second, for each generation $k\geq k_0$ and vertex $v\in V_k$, define an \emph{extension $E[k,v]$} to vertices in future generations as follows: \begin{quotation} Given any $v\in V_k$, pick a sequence of vertices $v_{1},v_{2},\dots$ inductively so that $v_{1}$ is a vertex in $V_{k+1}$ that is closest to $v_0=v$, $v_2$ is a vertex in $V_{k+2}$ that is closest to $v_1$, and so on. Then define $$E[k,v]:=\overline{\bigcup_{i=0}^\infty [v_i,v_{i+1}]}.$$\end{quotation} Once extensions have been chosen, for each generation $k\geq k_0$ and for each pair of vertices $v',v''\in \Gamma_k$, we define the \emph{bridge} $B[k,v',v'']$ by $$B[k,v',v''] := E[k,v']\cup [v',v'']\cup E[k,v''].$$ We remark that in the special case $V_{k+1}\supseteq V_k$ for all $k\geq k_0$, the extension $E[k,v]=\{v\}$ and the bridge $B[k,v',v'']=[v',v'']$.

To define the initial curve $\Gamma_{k_0}$, consider each pair of vertices $v',v''\in V_{k_0}$. If $|v'-v''|< 30C^\star 2^{-k_0}$, then we include the edge $[v',v'']$ in $\Gamma_{k_0}$. Otherwise, if $|v'-v''| \geq 30C^\star 2^{-k_0}$, then we include the bridge $B[k_0,v',v'']$ in $\Gamma_{k_0}$. That is, \begin{equation}\label{e:Gamma0-def}\Gamma_{k_0} := \bigcup_{\stackrel{v',v''\in V_{k_0}}{|v'-v''|<30 C^\star 2^{-k_0}}}[v',v'']\cup \bigcup_{\stackrel{v',v''\in V_{k_0}}{|v'-v''|\geq 30 C^\star 2^{-k_0}}}B[k_0,v',v''].\end{equation}
Suppose that $\Gamma_{k_0},\dots,\Gamma_{k-1}$ have been defined for some $k\geq k_0+1$. In order to define the next set $\Gamma_k$, we first describe $\Gamma_{k,v}$, the ``new part" of $\Gamma_k$ nearby $v$, for every $v\in V_k$. Then we declare  $\Gamma_k$ to be the union of new parts and old bridges. That is, \begin{equation}\label{e:Gamma-def}\Gamma_k:=\bigcup_{v\in V_k} \Gamma_{k,v} \cup \bigcup_{j=k_0}^{k-1} \bigcup_{B[j,w',w'']\subseteq \Gamma_{j}} B[j,w',w''].\end{equation}
Let $v$ be an arbitrary vertex in $V_k$. The definition of $\Gamma_{k,v}$ splits into two cases.

\textbf{Case I:} Suppose that $\alpha_{k,v}\geq \varepsilon$. To define $\Gamma_{k,v}$, we mimic the construction of the initial curve $\Gamma_{k_0}$. Consider every pair of vertices $v',v''\in V_k\cap B(v,65 C^\star 2^{-k})$. If $|v'-v''|< 30C^\star 2^{-k}$, then we include the edge $[v',v'']$ in $\Gamma_{k,v}$. Otherwise, if $|v'-v''|\geq 30 C^\star 2^{-k}$, then we include the bridge $B[k,v',v'']$ in $\Gamma_{k,v}$. This ends the description of $\Gamma_{k,v}$ in \textbf{Case I}.

\textbf{Case II:} Suppose that $\alpha_{k,v}<\varepsilon$. Identify the straight line $\ell_{k,v}$ with $\RR$ (in particular, pick directions ``left'' and ``right'') and let $\pi_{k,v}$ denote the orthogonal projection onto $\ell_{k,v}$.
 By Lemma \ref{l:graph} and ($V_I$), both $V_k\cap B(v,65 C^\star 2^{-k})$ and $V_{k-1}\cap B(v, 65C^\star2^{-k})$ are arranged linearly along $\ell_{k,v}$. Put $v_0=v\in V_k$ and let $$v_{-l},\dots, v_{-1},v_0,v_1,\dots,v_m$$ denote the vertices in $V_k\cap B(v, 65 C^\star 2^{-k})$, arranged from left to right according to the relative order of $\pi_{k,v}(v_i)$ in $\ell_{k,v}$ (identified with $\RR$), where $l,m\geq 0$.
We start by  describing the ``right half" $\Gamma_{k,v}^R$ of $\Gamma_{k,v}$, where $$\Gamma_{k,v}^R= \Gamma_{k,v}\cap \pi_{k,v}^{-1}([\pi(v_0),\infty)).$$ There will be three subcases.
Starting from $v_0$ and working to the right, include each closed line segment $[v_i,v_{i+1}]$ as an edge in $\Gamma_{k,v}^R$ until $|v_{i+1}-v_i|\geq 30 C^\star 2^{-k}$, $v_{i+1}\not\in B(v,30 C^\star 2^{-k})$, or $v_{i+1}$ is undefined (because $i=m$). Let $t\geq 0$ denote the number of edges that were included in $\Gamma_{k,v}^R$.

\textbf{Case II-NT:} If $t\geq 1$ (that is, at least one edge was included), then we say that the vertex \emph{$v$ is not terminal to the right} and are done describing $\Gamma_{k,v}^R$.

\textbf{Case II-T1 and Case II-T2:} If $t=0$ (that is, no edges were included), then we say that the vertex \emph{$v$ is terminal to the right} and continue our description of $\Gamma_{k,v}^R$, splitting into subcases depending on how $\Gamma_{k-1}$ looks nearby $v$. Let $w_v$ be a vertex in $V_{k-1}$ that is closest to $v$. Enumerate the vertices in $V_{k-1}\cap B(v,65 C^\star 2^{-k})$ starting from $w_v$ and moving right (with respect to the identification of $\ell_{k,v}$ with $\RR$) by $$w_v=w_{v,0},w_{v,1},\dots,w_{v,s}.$$ Let $w_{v,r}$ denote the rightmost vertex in $V_{k-1}\cap B(v,C^\star 2^{-(k-1)})$. There are two alternatives: \begin{enumerate}
\item[\textbf{T1:}] If $r=s$ or if $|w_{v,r}-w_{v,r+1}|\geq 30 C^\star 2^{-(k-1)}$, then we set $\Gamma_{k,v}^R=\{v\}$.
\item[\textbf{T2:}] If $|w_{v,r}-w_{v,r+1}|<30 C^\star 2^{-(k-1)}$, then $v_{1}$ exists by ($V_\two$) (and $|v-v_1|\geq 30 C^\star 2^{-k}$). In this case, we set $\Gamma_{k,v}^R=B[k,v,v_1]$.
\end{enumerate} The first alternative defines \textbf{Case II-T1}. The second alternative defines \textbf{Case II-T2}. This concludes the description of $\Gamma_{k,v}^R$.

Next, define the ``left half" $\Gamma_{k,v}^L=\Gamma_{k,v}\cap \pi_{k,v}^{-1}((-\infty,\pi_{k,v}(v_0)])$ of $\Gamma_{k,v}$ symmetrically. Also, define the terminology \emph{$v$ is not terminal to the left} and \emph{$v$ is terminal to the left} by analogy with the corresponding terminology to the right. Having separately defined both the ``left half" $\Gamma_{k,v}^L$ and the ``right half" $\Gamma_{k,v}^R$ of $\Gamma_{k,v}$, we now declare $$\Gamma_{k,v}:=\Gamma_{k,v}^L\cup \Gamma_{k,v}^R.$$ This concludes the description of $\Gamma_{k,v}$ in \textbf{Case II}.

\subsection{Connectedness}\label{ss:connected} By construction, for all $k\geq k_0$, every point $x\in\Gamma_k$ is connected to $V_k$ inside $\Gamma_k$, because $x$ belongs to an edge $[v',v'']$ between vertices $v',v''\in V_k$ or $x$ belongs to a bridge $B[j,u',u'']$ between vertices $u',u''\in V_{j}$ for some $k_0\leq j\leq k$. Thus, to prove that $\Gamma_k$ is connected, it suffices to prove that every pair of points in $V_k$ can be connected inside $\Gamma_k$. We argue by double induction.

The set $V_{k_0}$ is connected in $\Gamma_{k_0}$, because $\Gamma_{k_0}$ contains $[v',v'']$ or $B[k_0,v',v'']$ for every pair of vertices $v',v''\in V_{k_0}$. In subsequent generations, if $v',v''\in V_k$ and $|v'-v''|<30C^\star 2^{-k}$, then $v'$ and $v''$ are connected in $\Gamma_k$. This can be seen by inspection of the various cases in the definition of $\Gamma_{k,v'}$.
Suppose for induction that $V_{k-1}$ is connected in $\Gamma_{k-1}$ for some $k\geq k_0+1$. Let $x$ and $y$ be arbitrary vertices in $V_k$ and let $w_x, w_y\in V_{k-1}$ denote vertices that are closest to $x$ and $y$, respectively. Because $V_{k-1}$ is connected in $\Gamma_{k-1}$, $w_x$ and $w_y$ can be joined in $\Gamma_{k-1}$ by a tour of $p+1$ vertices in $V_{k-1}$, say $$w_0=w_x, w_1, \dots, w_{p}=w_y,$$ where each pair $w_i,w_{i+1}$ of consecutive vertices is connected in $\Gamma_{k-1}$ by an edge $[w_i,w_{i+1}]$ or by a bridge $B[j, u',u'']$ for some $k_0\leq j\leq k-1$ and $u',u''\in V_j$ with the property that $w_i\in E[j,u']$ and $w_{i+1}\in E[j,u'']$.

Set $v_0=x$, which satisfies $|v_0-w_0|=|x-w_x|< C^\star 2^{-k}$ by $(V_\three)$. Suppose for induction that $0\leq t\leq p-1$ and there exists a vertex $v_t\in V_k$ such that $|v_t-w_t| < C^\star 2^{-(k-1)}$ and $v_0$ and $v_t$ are connected in $\Gamma_k$. If $t\leq p-2$, choose $v_{t+1}$ to be any vertex in $V_k$ such that $|v_{t+1}-w_{t+1}| < C^\star 2^{-(k-1)}$, which exists by $(V_\two)$. Otherwise, if $t=p-1$, set $v_{t+1}=y$, which also satisfies $|v_{t+1}-w_{t+1}|=|y-w_y|< C^\star 2^{-(k-1)}$ by $(V_\three)$. We will now show that $v_t$ and $v_{t+1}$ are connected in $\Gamma_k$, and thus, $v_0$ and $v_{t+1}$ are connected in $\Gamma_k$. The proof splits into two cases, depending on whether the vertices $w_t$ and $w_{t+1}$ in $V_{k-1}$ are connected by a bridge or an edge.

First, suppose that $w_t,w_{t+1}\in B[j,u',u'']$ for some $k_0\leq j\leq k-1$ and some $u',u''\in V_j$ with $w_t\in E[j,u']$ and $w_{t+1}\in E[j,u'']$. Let $z'$ denote the unique point in $V_k\cap E[j,u']$ and let $z''$ denote the unique point in $V_{k}\cap E[j,u'']$. Then $z',z''\in B[j,u',u'']$. Hence $z'$ and $z''$ are connected in $\Gamma_k$, because $B[j,u',u'']\subseteq\Gamma_k$. Next, by definition of the extensions, $|z'-w_t|< C^\star 2^{-(k-1)}$ and $|z''-w_{t+1}| < C^\star 2^{-(k-1)}$. Thus, $$|v_t-z'| \leq |v_t-w_t|+|w_t-z'| <2 C^\star 2^{-(k-1)}<30 C^\star 2^{-k},$$ and similarly, $|v_{t+1}-z''|< 30 C^\star 2^{-k}$. It follows that $v_t$ is connected to $z'$ in $\Gamma_k$ and $v_{t+1}$ is connected to $z''$ in $\Gamma_k$. Therefore, concatenating paths, $v_{t}$ is connected to $v_{t+1}$ in $\Gamma_k$.

Secondly, suppose that $[w_t,w_{t+1}]$ is an edge in $\Gamma_{k-1}$. Then $|w_t-w_{t+1}| < 30 C^\star 2^{-(k-1)}$. Hence $$|v_t-v_{t+1}| \leq |v_t-w_t|+|w_t-w_{t+1}|+|w_{t+1}-v_{t+1}| < 32 C^\star 2^{-(k-1)}=64 C^\star 2^{-k}.$$ Because $|v_t-v_{t+1}|< 65 C^\star 2^{-k}$, it follows that $v_t$ is connected to $v_{t+1}$ in $V_k$ if $\alpha_{k,v_t}\geq \varepsilon$
by \textbf{Case I} in the definition of $\Gamma_{k,v_t}$.
On the other hand, suppose that $\alpha_{k,v_t}<\varepsilon$. Then $V_k\cap B(v_t,64C^\star 2^{-k})$ may be arranged linearly according to their relative order under orthogonal projection onto $\ell_{k,v_t}$. Label the vertices in $V_k\cap B(v_t, 64C^\star 2^{-k})$ lying between $v_t$ and $v_{t+1}$ inclusively, according to this order, say $$z_0=v_t,z_1,\dots,z_q=v_{t+1}.$$ Because $(1+3\varepsilon^2)64<65$, Lemma \ref{l:graph} ensures that $v_t,v_{t+1}\in B(z_i, 65 C^\star 2^{-k})$ for all $1\leq i\leq q$ (see Figure \ref{fig:1}).
\begin{figure}
\includegraphics[width=\textwidth]{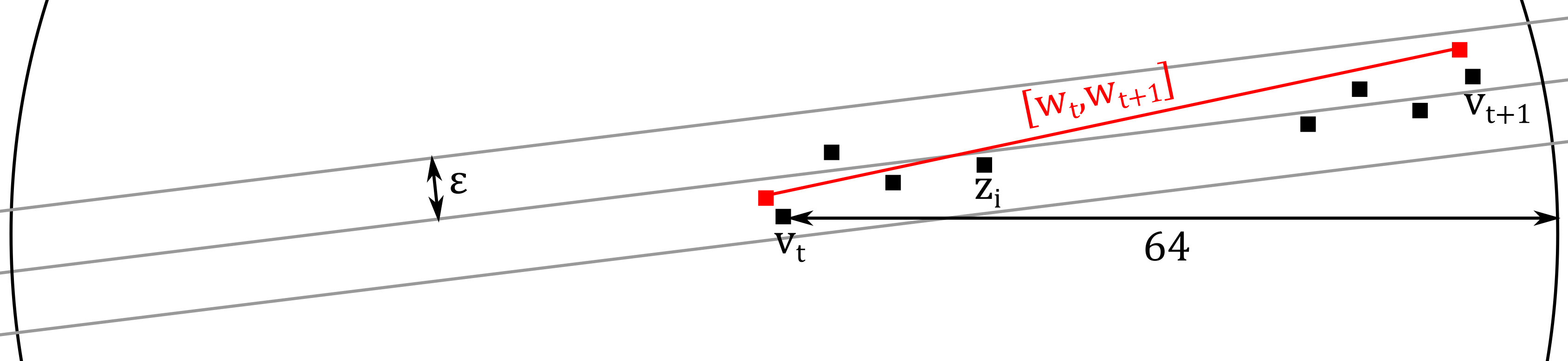}
\caption{Truncated view of $[w_t,w_{t+1}]$ and $z_0=v_t,z_1,\dots,z_q=v_{t+1}$ in $B(v_t,64 C^\star 2^{-k})$ when $\alpha_{k,v_t}<\varepsilon$, rescaled so that $C^\star 2^{-k}=1$.}
\label{fig:1}
\end{figure} Thus, $v_t$ and $v_{t+1}$ are connected in $V_k$ if $\alpha_{k,z_i}\geq \varepsilon$ for some $1\leq i\leq q$ by \textbf{Case I} in the definition of $\Gamma_{k,z_i}$. Finally, suppose that $\alpha_{k,z_i}<\varepsilon$ for all $1\leq i\leq q$. Because $\Gamma_{k-1}$ contains the \emph{edge} $[w_t,w_{t+1}]$, the set $\Gamma_{k,z_i}$ contains $B[k,z_i,z_{i+1}]$ or $[z_i,z_{i+1}]$ for each $0\leq i\leq q-1$, depending on whether $z_i$ is terminal or $z_i$ is not terminal in the direction from $z_i$ to $z_{i+1}$. (In particular, in each instance alternative \textbf{T1} does not occur.) Hence $z_i$ and $z_{i+1}$ are connected in $\Gamma_k$ for all $0\leq i\leq q-1$. Therefore, concatenating paths, we see that $v_t=z_0$ and $v_{t+1}=z_q$ are connected in $\Gamma_k$ in this case, as well.

By induction, $v_0$ and $v_t$ are connected in $V_k$ for all $1\leq t\leq p$. In particular, $x=v_0$ and $y=v_p$ are connected in $V_k$. Since $x$ and $y$ were arbitrary vertices in $V_k$, it follows that $V_k$ is connected in $\Gamma_k$. Therefore, by induction, $\Gamma_k$ is connected for all $k\geq k_0$.

\section{Drawing rectifiable curves II: length estimates}\label{ss:length} We continue to adopt the notation and assumptions of \S\S \ref{ss:overview}--\ref{ss:connected}.

Our goal in this section is to verify that $\Gamma_{k_0+1},\Gamma_{k_0+2},\dots$ satisfy the estimate \eqref{e:goal}. Roughly speaking, we would like to bound the length of $\Gamma_{k_0}$ by $C 2^{-k_0}$ and to bound $\Haus^1(\Gamma_k)$ by $\Haus^1(\Gamma_{k-1})+C \sum_{v\in V_k} \alpha_{k,v}^2 2^{-k}$ for all $k\geq k_0+1$, for some $C$ independent of $k$. In other words, we want to ``pay" for the length of the ``new curve" $\Gamma_k$ with the length of ``old curve" $\Gamma_{k-1}$ and the sum $C\sum_{v\in V_k}\alpha_{k,v}^2 2^{-k}$. This plan works more or less, except that more work is required to pay for an edge $[v',v'']$ in $\Gamma_k$ when the vertex $v'$ or $v''$ is close to a terminal vertex in \textbf{Case II} of the construction, because the old curve may not ``span" the new edge $[v',v'']$. To handle this extra complication, we introduce a mechanism to ``prepay" length called phantom length. The idea for phantom length comes from Jones' original traveling salesman construction (see \cite{Jones-TSP}).

\subsection{Phantom length}\label{ss:phantom} Below it will be convenient to have notation to refer to the vertices appearing in a bridge. For each extension $E[k,v]$, say with $E[k,v]$ defined by $E[k,v]=\overline{\bigcup_{i=0}^\infty [v_i,v_{i+1}]}$, we define the corresponding \emph{extension index set} $I[k,v]$ by $$I[k,v]=\{(k+i,v_i):i\geq 0\}.$$ For each bridge $B[k,v',v'']$, we define the corresponding \emph{bridge index set} $I[k,v',v'']$ by $$I[k,v',v'']=I[k,v']\cup I[k,v''].$$

For all generations $k\geq k_0$ and for all vertices $v\in V_k$, we define the \emph{phantom length} $p_{k,v} := 3 C^\star 2^{-k}$.
If $B[k,v',v'']$ is a bridge between vertices $v',v''\in V_k$, then the totality $p_{k,v',v''}$ of phantom length associated to pairs in $I[k,v',v'']$ is given by $$p_{k,v',v''}:=3C^\star\left(2^{-k}+2^{-(k+1)}+\cdots\right)+3C^\star\left(2^{-k}+2^{-(k+1)}+\cdots\right) =12 C^\star 2^{-k}.$$
During the proof we will keep tally of phantom length at certain pairs $(k,v)$ with $v\in V_k$ as an accounting tool.

We initialize $\phan(k_0)$, the \emph{index set} of phantom length tracked at stage $k_0$, to be the set of all pairs $(j,u)$ such that the vertex $u\in V_{j}$ appears in the definition of $\Gamma_{k_0}$, including all vertices in $V_{k_0}$ and all vertices in extensions in bridges in $\Gamma_{k_0}$. That is, $$\phan(k_0):= \{(k_0,v):v\in V_{k_0}\}\cup \bigcup_{B[k_0,v',v'']\subseteq \Gamma_{k_0}} I[k_0,v',v''].$$ Suppose that $\phan(k_0),\dots,\phan(k-1)$ have been defined for some $k\geq k_0+1$, where the index set $\phan(k-1)$ satisfies the following two properties.

\begin{quotation}\emph{Bridge property:} If a bridge $B[k-1,w',w'']$ is included in $\Gamma_{k-1}$, then $\phan(k-1)$ contains $I[k-1,w',w'']$.\end{quotation}

\begin{quotation}\emph{Terminal vertex property:} Let $w\in V_{k-1}$ and let $\ell$ be a line such that $$\dist(y,\ell)<\varepsilon 2^{-(k-1)}\quad\text{for all }y\in V_{k-1}\cap B(w,30 C^\star 2^{-(k-1)}).$$ Arrange $V_{k-1}\cap B(w,30 C^\star 2^{-(k-1)})$ linearly with respect to the orthogonal projection $\pi_\ell$ onto $\ell$. If there is no vertex $w'\in V_{k-1}\cap B(w,30 C^\star 2^{-(k-1)})$ to the ``left" of $w$ or to the ``right" of $w$, then $(k-1,w)\in \phan(k-1)$. That is, identifying $\ell$ with $\RR$, if  $$V_{k-1}\cap B(w,30 C^\star 2^{-(k-1)})\cap \pi_\ell^{-1}((-\infty,\pi_\ell(w)))=\emptyset\text{ or}$$ $$V_{k-1}\cap B(w,30 C^\star 2^{-(k-1)})\cap \pi_\ell^{-1}((\pi_\ell(w),\infty))=\emptyset,$$ then $(k-1,w)\in \phan(k-1)$. \end{quotation}
(Note that $\phan(k_0)$ satisfies the terminal vertex property trivially, since $\phan(k_0)$ includes $(k_0,v)$ for every $v\in V_{k_0}$.) We form $\phan(k)$ starting from $\phan(k-1)$, as follows. Initialize the set $\phan(k)$ to be equal to $\phan(k-1)$. Next, delete all pairs $(k-1,w)$ and $(k,\tilde v)$ appearing in $\phan(k-1)$ from $\phan(k)$. Lastly, for each vertex $v\in V_k$, include additional pairs in $\phan(k)$ according to the following rules. \begin{itemize}
\item \textbf{Case I:} Suppose that $\alpha_{k,v}\geq \varepsilon$. Include $(k,v')$ in $\phan(k)$ for all vertices $v'\in V_k\cap B(v,65 C^\star 2^{-k})$ and include $I[k,v',v'']$ as a subset of $\phan(k)$ for every bridge $B[k,v',v'']$ in $\Gamma_{k,v}$.
\item \textbf{Case II-NT:} Suppose that $\alpha_{k,v}<\varepsilon$, and $\Gamma_{k,v}^R$ or $\Gamma_{k,v}^L$ is defined by \textbf{Case II-NT}. This case does not generate any phantom length.
\item \textbf{Case II-T1:} Suppose that $\alpha_{k,v}<\varepsilon$, and $\Gamma_{k,v}^R$ or $\Gamma_{k,v}^L$ is defined by \textbf{Case II-T1}. Include $(k,v)\in \phan(k)$.
\item \textbf{Case II-T2:} Suppose that $\alpha_{k,v}<\varepsilon$, and $\Gamma_{k,v}^R$ or $\Gamma_{k,v}^L$ is defined by \textbf{Case II-T2}.
When $\Gamma_{k,v}^R$ is defined by \textbf{Case II-T2}, include $I[k,v,v_1]$ as a subset of $\phan(k)$. When $\Gamma_{k,v}^L$ is defined by \textbf{Case II-T2}, include $I[k,v_{-1},v]$ as a subset of $\phan(k)$. In particular, note that $(k,v)$ is included in $\phan(k)$.
\end{itemize}
The phantom length associated to deleted pairs will be available to pay for the length of edges in $\Gamma_k$ near terminal vertices in $V_k$ and to pay for the phantom length of pairs in $\phan(k)\setminus\phan(k-1)$.

It is clear that $\phan(k)$ satisfies the bridge property. To check that $\phan(k)$ satisfies the terminal vertex property, let $v\in V_{k}$ and let $\ell$ be a line such that $$\dist(v,\ell)<\varepsilon 2^{-k}\quad\text{for all }y\in V_{k}\cap B(v,30 C^\star 2^{-k}).$$ Identify $\ell$ with $\RR$ and arrange $V_{k}\cap B(v,30 C^\star 2^{-k})$ linearly with respect to the orthogonal projection $\pi_\ell$ onto $\ell$. Assume that there is no vertex $v'\in V_{k}\cap B(v,30 C^\star 2^{-k})$ to the ``left" of $v$ or to the ``right" of $v$ with respect the ordering under $\pi_\ell$. If $\alpha_{k,v}\geq \varepsilon$, then $(k,\tilde v)$ was included in $\phan(k)$ for every $\tilde v\in V_k\cap B(v,65 C^\star 2^{-k})$. In particular, $(k,v)$ is in $\phan(k)$. Otherwise, if $\alpha_{k,v}<\varepsilon$, then $V_{k}\cap B(v, 30 C^\star 2^{-k})$ is also linearly ordered with respect to the orthogonal projection onto $\ell_{k,v}$. By Lemma \ref{l:graph}, the two orderings agree modulo the choice of orientation for $\ell$ and $\ell_{k,v}$. In this case, the assumption that there is no vertex $v'\in V_{k}\cap B(v,30 C^\star 2^{-k})$ to the ``left" of $v$ or to the ``right" of $v$ translates to the statement that $\Gamma_{k,v}^L$ or $\Gamma_{k,v}^R$ is defined by \textbf{Case II-T1} or \textbf{Case II-T2}, whence $(k,v)$ was included in $\phan(k)$. Therefore, $\phan(k)$ satisfies the terminal vertex property.

\subsection{Core of a bridge} \label{ss:core} For every bridge $B[k,v',v'']$ between vertices $v',v''\in V_k$, we define the \emph{core} $C[k,v',v'']$ of the bridge to be $$C[k,v',v'']=\frac{9}{10}[v',v''].$$ That is, $C[k,v',v'']$ is the interval of length $9/10$ of the length of $[v',v'']$ that is concentric with $[v',v'']$.
Because $\Haus^1(B[k,v',v''])\geq 30 C^\star 2^{-k}$ for any bridge $B[k,v',v'']$ included in the construction, the corresponding core has significant length: \begin{equation}\label{e:core-length}\Haus^1(C[k,v',v''])\geq 27 C^\star 2^{-k}.\end{equation}

For all $k\geq k_0+1$, let $\cores(k)$ denote the set of cores $C[k,v',v'']$ of bridges $B[k,v',v'']$ between vertices $v',v''\in V_k$ that were included in $\Gamma_k$ in \textbf{Case II-T2} for some $\Gamma_{k,v}^R$ or $\Gamma_{k,v}^L$. We claim that cores in $\bigcup_{j=k_0+1}^\infty \cores(j)$ are pairwise disjoint. To see this, suppose that $C[k,v',v'']\in\cores(k)$ and $C[j,w',w'']\in\cores(j)$ for some $j\geq k\geq k_0+1$. Because the bridges $B[k,v',v'']$ and $B[j,w',w'']$ arise in \textbf{Case II-T2} of the construction, we have \begin{equation}\label{e:roughlength} 30 C^\star 2^{-k} \leq |v'-v''| < 64 C^\star 2^{-k}\quad\text{and}\quad 30C^\star 2^{-j} \leq |w'-w''|< 64 C^\star 2^{-j},\end{equation} where the upper bounds follow from the bound $|w_{v,r}-w_{v,r+1}|<30 C^\star 2^{-(k-1)}$ appearing in \textbf{Case II-T2} of the definition of $\Gamma_{k,v}^R$. Now, by $(V_\three)$, $V_j\cap B(v',64 C^\star 2^{-k})$ is contained in a $C^\star 2^{-k}$ neighborhood of vertices in $V_k\cap B(v',65 C^\star 2^{-k})$ (see Figure \ref{fig:2}).
\begin{figure}
\includegraphics[width=\textwidth]{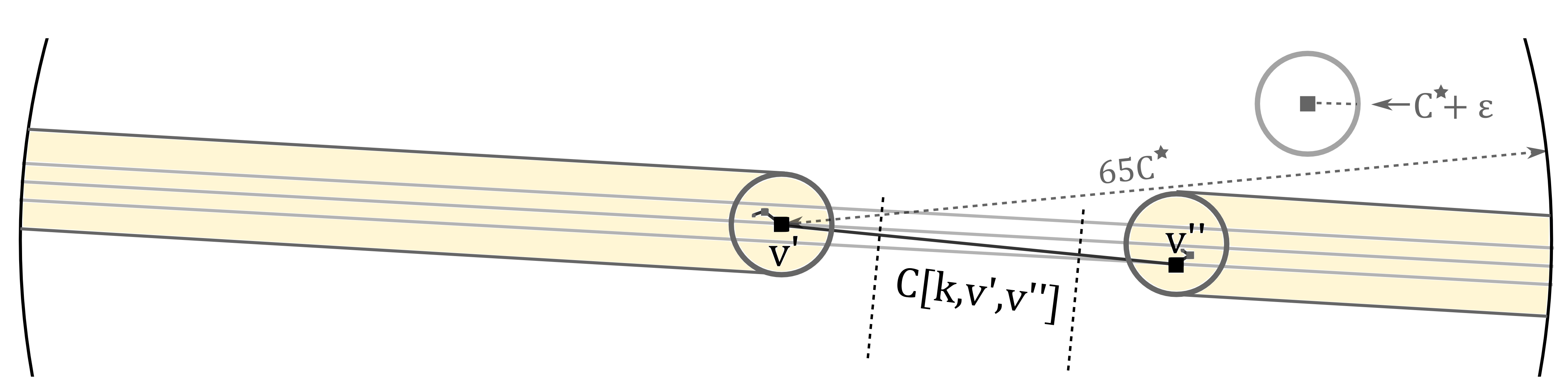}
\caption{Possible location of future vertices nearby a bridge $B[k,v',v'']$ included in $\Gamma_{k,v'}^R$ (\textbf{Case II-T2}) shaded in yellow, rescaled so that $2^{-k}=1$.}
\label{fig:2}
\end{figure} Suppose to get a contradiction that $C[k,v',v'']$ and $C[j,w',w'']$ are distinct cores that intersect. Note that the intersection of the cores implies that the two end points $w'$ and $w''$ of $B[j,w',w'']$ lie in opposite shaded regions in Figure \ref{fig:2}. There are now several cases to consider, but we can reach a contradiction in each one. First, if $j\geq k+2$, then the intersection of $C[k,v',v'']$ and $C[j,w',w'']$ implies (by length considerations, see \eqref{e:roughlength}) that $w'$ or $w''$ lies in the empty space of the figure, where no vertex exists. Next, if $j=k+1$, then the intersection of the cores would imply that $[v',v'']$ is included as an edge in $\Gamma_{k}$, violating the bound $|v'-v''|\geq 30 C^\star 2^{-k}$.  Lastly, if $j=k$, then the intersection of the cores contradicts that fact that $v'$ is terminal in the direction from $v'$ to $v''$ and $w'$ (or $w''$) is terminal in the direction from $w'$ ($w''$) to $w''$ ($w'$). We leave the details to the reader.
\subsection{Proof of \eqref{e:goal}} In this section, we break up the proof of \eqref{e:goal} into two estimates. To state these, we first introduce some useful notation. \begin{itemize}
\item Let $\edges(k)$ denote the set of all edges $[v',v'']$ included in $\Gamma_{k}$.

\item Let $\bridges(k)$ denote the set of all bridges $B[k,v',v'']$ included in $\Gamma_{k}$.

\item Let $\phan(k)$ denote the index set of phantom length, defined above in \S \ref{ss:phantom}.
\item Let $\cores(k)$ denote the set of cores $C[k,v',v'']$ of bridges $B[k,v',v'']$ between vertices $v',v''\in V_k$ that were included in $\Gamma_k$ in \textbf{Case II-T2} for some $\Gamma_{k,v}^R$ or $\Gamma_{k,v}^L$. See \S \ref{ss:core} above.
\end{itemize}
To establish \eqref{e:goal}, it suffices to prove first that

\begin{equation}\begin{split} \label{e:sum-a}
\sum_{[v',v'']\in \edges(k_0)} &\Haus^1([v',v'']) + \sum_{B[k_0,v',v'']\in \bridges(k_0)}\Haus^1(B[k_0,v',v'']) + \sum_{(j,u)\in \phan(k_0)} p_{j,u}\\ &\leq C 2^{-k_0}\end{split}\end{equation} and second that, for all $k\geq k_0+1$,
\begin{equation}\begin{split} \label{e:sum-b}
\sum_{[v',v'']\in \edges(k)} &\Haus^1([v',v''])+ \sum_{B[k,v',v'']\in \bridges(k)}\Haus^1(B[k,v',v''])+ \sum_{(j,u)\in \phan(k)} p_{j,u} \\ &\leq \sum_{[w',w'']\in \edges({k-1})} \Haus^1([w',w'']) + \sum_{(j,u)\in \phan(k-1)} p_{j,u} \\ &\qquad\qquad\qquad\qquad+ C\sum_{v\in V_k} \alpha_{k,v}^2 2^{-k} + \frac{25}{27}\sum_{C[k,v',v'']\in\cores(k)} \Haus^1(C[k,v',v'']),\end{split}\end{equation} where $C$ denotes a constant depending only on $n$ and $C^\star$.
To see this, let $k\geq k_0+1$. Then, by \eqref{e:Gamma-def}, $$\Haus^1(\Gamma_{k}) \leq \sum_{[v',v'']\in \edges(k)} \Haus^1([v',v''])+ \sum_{j=k_0}^k\sum_{B[j,w',w'']\in \bridges(j)}\Haus^1(B[j,w',w'']).$$ Iterating \eqref{e:sum-b} a total of $k-k_0$ times, we obtain
\begin{equation*}\begin{split} \Haus^1(\Gamma_k) \leq \sum_{[v',v'']\in \edges(k_0)}&\Haus^1([v',v'']) + \sum_{B[k_0,v',v'']\in \bridges(k_0)}\Haus^1(B[k_0,v',v'']) \\ &+\sum_{(j,u)\in\phan(k_0)} p_{j,u} + C\sum_{j={k_0+1}}^k \sum_{u\in V_j}\alpha_{j,u}^2 2^{-j} \\ &\qquad\qquad\qquad+ \frac{25}{27}\sum_{j=k_0+1}^{k}\sum_{C[j,w',w'']\in\cores(j)} \Haus^1(C[j,w',w'']), \end{split} \end{equation*} where $C$ depends on at most $n$ and $C^\star$. Since the cores in $\bigcup_{j=k_0+1}^k\cores(j)$ are pairwise disjoint (see \S\ref{ss:core}) and belong to $\Gamma_k$ (see \eqref{e:Gamma-def}), it follows that \begin{equation*}\begin{split}\Haus^1(\Gamma_k) \leq \sum_{[v',v'']\in \edges(k_0)} &\Haus^1([v',v''])+ \sum_{B[k,v',v'']\in \bridges(k_0)}\Haus^1(B[k_0,v',v'']) \\ &+ \sum_{(j,u)\in\phan(k_0)} p_{j,u} + C\sum_{j={k_0+1}}^k \sum_{u\in V_j}\alpha_{j,u}^2 2^{-j}+\frac{25}{27}\Haus^1(\Gamma_k),\end{split}\end{equation*} where $C$ depends on at most $n$ and $C^\star$. Therefore, by \eqref{e:sum-a}, $$\Haus^1(\Gamma_k) \leq C\left(2^{-k_0} + \sum_{j={k_0+1}}^k\sum_{v\in V_j} \alpha_{j,v}^2 2^{-j} \right)\quad\text{for all }k\geq k_0+1,$$ where $C$ depends on at most $n$ and $C^\star$.
This shows that \eqref{e:goal} follows from \eqref{e:sum-a} and \eqref{e:sum-b}.

\subsection{Proof of \eqref{e:sum-a}} \label{ss:sum-a} Let us start with a few preliminary observations. In the curves $\Gamma_{k_0},\Gamma_{k_0+1},\dots$, an edge $[v',v'']$ is included for some $v',v''\in V_k$ only if $$|v'-v''|<30 C^\star 2^{-k},$$ while a bridge $B[k,v',v'']$ is included for some $v',v''\in V_k$ only if $$30 C^\star 2^{-k}\leq |v'-v''|<130 C^\star 2^{-k}.$$ Moreover, the lengths of extensions are controlled by ($V_\two$): For all $k\geq k_0$ and $v\in V_k$, $\Haus^1(E[k,v]) \leq 2C^\star 2^{-k}$. Thus, if $B[k,v',v'']$ is included in some curve, then \begin{equation*}\begin{split}\Haus^1(B[k,v',v''])\leq \Haus^1(E[k,v'])+\Haus^1([v',v''])&+\Haus^1(E[k,v''])\\ \leq 4C^\star 2^{-k}&+\Haus^1([v',v''])< 1.14 \Haus^1([v',v'']).\end{split}\end{equation*}

Recall that $k_0\geq 0$ was defined to be the least index such that $\#V_k\geq 2$ for all $k\geq k_0$. Hence, by ($V_\three$), there exists $y_0$ such that $V_{k_0}\subseteq B(y_0,C^\star 2^{-k_0})$ (where $y_0=x_0$ if $k_0=0$). On one hand, $$\Haus^1([v',v''])\leq \diam B(y_0, C^\star 2^{-k_0})$$ for each edge $[v',v'']$ in $\Gamma_{k_0}$. On the other hand, $$\Haus^1(B[k_0,v',v''])< 1.14 \diam B(y_0, C^\star 2^{-k_0}) $$ for each bridge $B[k_0,v',v'']$ in $\Gamma_{k_0}$. Since $\#V_{k_0}= \#(V_{k_0}\cap B(y_0, C^\star 2^{-k_0}))\lesssim_{n,C^\star} 1$ by ($V_I$), it follows that $$\sum_{[v',v'']\in \edges(k_0)} \Haus^1([v',v'']) + \sum_{B[k_0,v',v'']\in \bridges(k_0)}\Haus^1(B[k_0,v',v''])\lesssim_{n,C^\star} 2^{-k_0}.$$
Also, since $\#V_{k_0}\lesssim_{n,C^\star} 1$, the total amount of phantom length associated to $\phan(k_0)$ can be estimated by $$\sum_{(j,u)\in \phan(k_0)} p_{j,u}\leq \sum_{v\in V_{k_0}} p_{k_0,v}+\sum_{v',v''\in V_{k_0}} p_{k_0,v',v''}\lesssim_{n,C^\star} 2^{-k_0}.$$ Combining the previous two displayed equations yields \eqref{e:sum-a}.

\subsection{Proof of \eqref{e:sum-b}} \label{ss:sum-b}
Edges and bridges forming the curve $\Gamma_k$ and ``new" phantom length associated to pairs in $\phan(k)\setminus\phan(k-1)$ may enter the local picture $\Gamma_{k,v}$ of $\Gamma_k$ near $v$ for several vertices $v\in V_k$, but only need to be accounted for once each to estimate the left hand side of \eqref{e:sum-b}. We prioritize as follows: \begin{enumerate}
\item[1.] \textbf{Case I} edges, \textbf{Case I} bridges, \textbf{Case I} phantom length;
\item[2.] \textbf{Case II-T1} phantom length and (parts of) edges that are nearby \textbf{Case II-T1} terminal vertices (where here and below \emph{nearby} means distance at most $2C^\star 2^{-k}$);
\item[3.] \textbf{Case II-T2} bridges, \textbf{Case II-T2} phantom length, and (parts of) edges that are nearby \textbf{Case II-T2} terminal vertices;
\item[4.] remaining (parts of) edges, which are necessarily far away from \textbf{Case I} vertices and \textbf{Case II-T1} and \textbf{Case II-T2} terminal vertices.
\end{enumerate}

\textbf{First Estimate (Case I):} Suppose that $\alpha_{k,v}\geq \varepsilon$. Abbreviate $B(v,65 C^\star 2^{-k})=:B_{k,v}$. Since $\#(V_k\cap B_{k,v})\lesssim_{n,C^\star} 1$ by ($V_I$), $\varepsilon\leq \alpha_{k,v}$, and $\varepsilon\gtrsim 1$, it follows that $$\sum_{\stackrel{v',v''\in V_k\cap  B_{k,v}}{[v',v'']\in \edges(k)}} \Haus^1([v',v'']) + \sum_{\stackrel{v',v''\in V_k\cap  B_{k,v}}{B[k,v',v'']\in \bridges(k)}}\Haus^1(B[k,v',v''])\lesssim_{n,C^\star} 2^{-k}\lesssim_{n,C^\star} \alpha_{k,v}^2 2^{-k}.$$ Similarly, the total amount of phantom length associated to $\Gamma_{k,v}$ does not exceed $$\sum_{v'\in V_k\cap B_{k,v}} p_{k,v'}+\sum_{\stackrel{v',v''\in V_k\cap B_{k,v}}{|v'-v''|\geq 30 C^\star 2^{-k}}}p_{k,v',v''}\lesssim_{n,C^\star} 2^{-k}\lesssim_{n,C^\star} \alpha_{k,v}^2 2^{-k}.$$

\textbf{Second Estimate (Case II-T1):} Suppose that $\alpha_{k,v}<\varepsilon$ and $v$ is terminal to the right with alternative \textbf{T1}. (The case when $v$ is terminal to the left is handled analogously.) Because alternative \textbf{T1} holds, the terminal vertex property ensures that $(k-1,w_{v,r})\in \phan(k-1)$, where $w_{v,r}$ is the rightmost vertex in $V_{k-1}\cap B(v,C^\star 2^{-(k-1)})$. Use half of the phantom length $p_{k-1,w_{v,r}}=3 C^\star 2^{-(k-1)}$ to pay for the phantom length $p_{k,v}= 3 C^\star 2^{-k}$. Use the other half of $p_{k-1,w_{v,r}}$ to pay for edges or parts of edges in $\Gamma_{k}\cap B(v,2C^\star 2^{-k})$, which by Lemma \ref{l:graph} is less than $3C^\star 2^{-k}$ as $(1+3\varepsilon^2)2<3$. That is, $$p_{k,v} + \sum_{[v',v'']\in\edges(k)} \Haus^1([v',v'']\cap B(v,2C^\star 2^{-k})) \leq p_{k-1,w_{v,r}}.$$

A degenerate case may occur if $w_{v,r}\in V_{k-1}$ is simultaneously terminal to the left and to the right (that is, if no edges in $\Gamma_{k-1}$ emanate from $w_{v,r}$), but $v$ is not terminal to the left. Let $v_l$ be the leftmost vertex in $V_k\cap B(v, 30C^\star 2^{-k})$ relative to the projection onto $\ell_{k,v}$. Then $v_l$ and $v$ are distinct, because $v$ is not terminal to the left. Assume that $\alpha_{k,v_l}<\varepsilon$ (otherwise, everything is paid for by the previous set of estimates). Then $v_l$ is terminal in the direction away from $v$ with alternative \textbf{T1}. This situation is degenerate, because the phantom length $p_{k-1,w_{v,r}}$ is not enough to pay for $p_{k,v}$, $p_{k,v_l}$, and the edges connecting $v_l$ and $v$. We handle this as follows. First, note that $|v_l-v|<2C^\star 2^{-k}$, because $w_{v,r}\in B(v_l, C^\star 2^{-k})\cap B(v, C^\star 2^{-k})$ by $(V_\two)$. In particular, since $(1+3\varepsilon^2)2<3$, the length of edges connecting $v$ and $v_l$ is less than $3C^\star 2^{-k}$ by Lemma \ref{l:graph}. Secondly, because $\#V_{k-1}\geq 2$ and $\Gamma_{k-1}$ is connected, $w_{v,r}$ must belong to an extension $E[j,u']$ of a bridge $B[j,u',u'']$ included in $\Gamma_j$ for some $k_0\leq j\leq k-1$ and some $u',u''\in V_j$. Let $\tilde v$ be the unique point in $V_k\cap E[j,u']$. By the bridge property, $(k,\tilde v)\in \phan(k-1)$. Then $$p_{k,v} + p_{k,v_l}+ \sum_{[v',v'']\in\edges(k)} \Haus^1([v',v'']\cap B(v,2C^\star 2^{-k})) \leq p_{k-1,w_{v,r}}+p_{k,\tilde v}.$$

\textbf{Third Estimate (Case II-T2):}
Suppose that $\alpha_{k,v}<\varepsilon$ and $v$ is terminal to the right with alternative \textbf{T2}. (The case when $v$ is terminal to the left is handled analogously.) Write $v_1\in V_k$ and $w_{v,r},w_{v,r+1}\in V_{k-1}$ for the vertices appearing in the definition of $\Gamma_{k,v}^R$. In this case, we will pay for $p_{k,v,v_1}$, the length of the bridge $B[k,v,v_1]$, and the length of edges in $\Gamma_k\cap B(v,2C^\star 2^{-k})$ and $\Gamma_k\cap B(v_1,2C^\star 2^{-k})$. Assume that $\alpha_{k,v_1}<\varepsilon$ (otherwise, everything is paid for by the first set of estimates). In \S\ref{ss:sum-a}, we noted that $$\Haus^1(B[k,v,v_1])\leq 4C^\star 2^{-k}+\Haus^1([v,v_1]).$$ Because $|v-w_{v,r}|<2C^\star 2^{-k}$ and $|v_1-w_{v,r+1}|< 2C^\star 2^{-k}$, it follows that $$\Haus^1(B[k,v,v_1])\leq 4C^\star 2^{-k}+\Haus^1([v,v_1])\leq 8C^\star2^{-k}+ \Haus^1([w_{v,r},w_{v,r+1}]).$$ The totality $p_{k,v,v_1}$ of phantom length associated to vertices in $B[k,v,v_1]$ is $12C^\star 2^{-k}$. Finally, since $\alpha_{k,v}<\varepsilon$ and $\alpha_{k,v_1}<\varepsilon$, the total length of (parts of) edges in $$\Gamma_k\cap (B(v,2C^\star 2^{-k})\cup B(v_1,2C^\star 2^{-k}))$$ does not exceed $5 C^\star 2^{-k}$ by Lemma \ref{l:graph} since $(1+3\varepsilon^2)2<2.5$. Altogether, \begin{equation*}\begin{split}
\Haus^1(B[k,v,v_1])+ p_{k,v,v_1} + \sum_{[v',v'']\in\edges(k)} \Haus^1\Big([v',v'']\cap (B(v,2C^\star 2^{-k})\cup B(v_1,2C^\star 2^{-k}))\Big)\\
\leq \Haus^1([w_{v,r},w_{v,r+1}]) + 25 C^\star 2^{-k} \stackrel{(9.1)}{\leq} \Haus^1([w_{v,r},w_{v,r+1}])+ \frac{25}{27}\Haus^1(C[k,v,v_1]),\end{split}\end{equation*} where $[w_{v,r},w_{v,r+1}]\in\edges(k-1)$ and $C[k,v,v_1]\in\cores(k).$

\textbf{Fourth Estimate (Case II-NT):} Let $[v',v'']$ be an edge between vertices $v',v''\in V_k$, which is not yet wholly paid for. Then $\alpha_{k,v'}<\varepsilon$ and $\alpha_{k,v''}<\varepsilon$. Let $[u',u'']$ be the largest closed subinterval of $[v',v'']$ such that $u'$ and $u''$ lie at distance at least $2C^\star 2^{-k}$ away from  \textbf{Case II-T1} and \textbf{Case II-T2} terminal vertices (see Figure \ref{fig:4}). By Lemma \ref{l:graph}, \begin{equation*}\begin{split}\Haus^1([u',u'']) &\leq (1+3\alpha_{k,v'}^2)\Haus^1([\pi_{k,v'}(u'),\pi_{k,v'}(u'')])\\ &< \Haus^1([\pi_{k,v'}(u'),\pi_{k,v'}(u'')])+90 C^\star \alpha_{k,v'}^2 2^{-k}.\end{split}\end{equation*} Without loss of generality, suppose that $u'$ lies to the left of $u''$ relative to the order of their projection onto $\ell_{k,v'}$. Let $z'$ denote the first vertex in $V_k\cap B(v',65 C^\star 2^{-k})$ to the left of $u'$ (relative to the order of their projection onto $\ell_{k,v'}$) such that $$\pi_{k,v'}(z')<\pi_{k,v'}(u') -C^\star 2^{-k}.$$ Let $z''$ denote the first vertex in $V_k\cap B(v',65 C^\star 2^{-k})$ to the right of $u''$ (relative to the order of their projection onto $\ell_{k,v'}$) such that $$\pi_{k,v'}(u'')+ C^\star 2^{-k}< \pi_{k,v'}(z'').$$ The vertices $z'$ and $z''$ exist, because $u'$ and $u''$ stay a distance of at least $2C^\star 2^{-k}$ away from \textbf{Case II-T1} and \textbf{Case II-T2} terminal vertices and $\alpha_{k,v'}<\varepsilon$. By ($V_\three$), we can find $w',w''\in V_{k-1}$ such that $|w'-z'|< C^\star 2^{-k}$ and $|w''-z''|< C^\star 2^{-k}$. It follows that \begin{equation*}\pi_{k,v'}(w')<\pi_{k,v'}(u')< \pi_{k,v'}(u'')<\pi_{k,v'}(w'').\end{equation*} Since $|z'-v'|< 30 C^\star 2^{-k}$, $|v'-v''|< 30C^\star 2^{-k}$, and $|v''-z''|< 30 C^\star 2^{-k}$, there exists a sequence of edges in $\Gamma_{k-1}\cap B(v',65 C^\star 2^{-k})$ connecting $w'$ and $w''$ in $\Gamma_{k-1}$ by $(V_\three)$. Hence we can pay for $\Haus^1([\pi_{k,v'}(u'),\pi_{k,v'}(u'')])$ using the portion of (edges in) the curve $\Gamma_{k-1}\cap B(v',65 C^\star 2^{-k})$ that lies over the line segment $[\pi_{k,v'}(u'),\pi_{k,v'}(u'')]$. Thus, $$\Haus^1([u',u'']) \leq \Haus^1\left(E_{k-1}(v')\cap \pi_{k,v'}^{-1}([\pi_{k,v'}(u'),\pi_{k,v'}(u'')])\right)+ 90C^\star \alpha_{k,v'}^2 2^{-k},$$ where $E_{k-1}(v')$ denotes the union of edges in $\Gamma_{k-1}$ between vertices in $V_{k-1}\cap B(v',65C^\star 2^{-k})$.

\begin{figure}
\includegraphics[width=.8\textwidth]{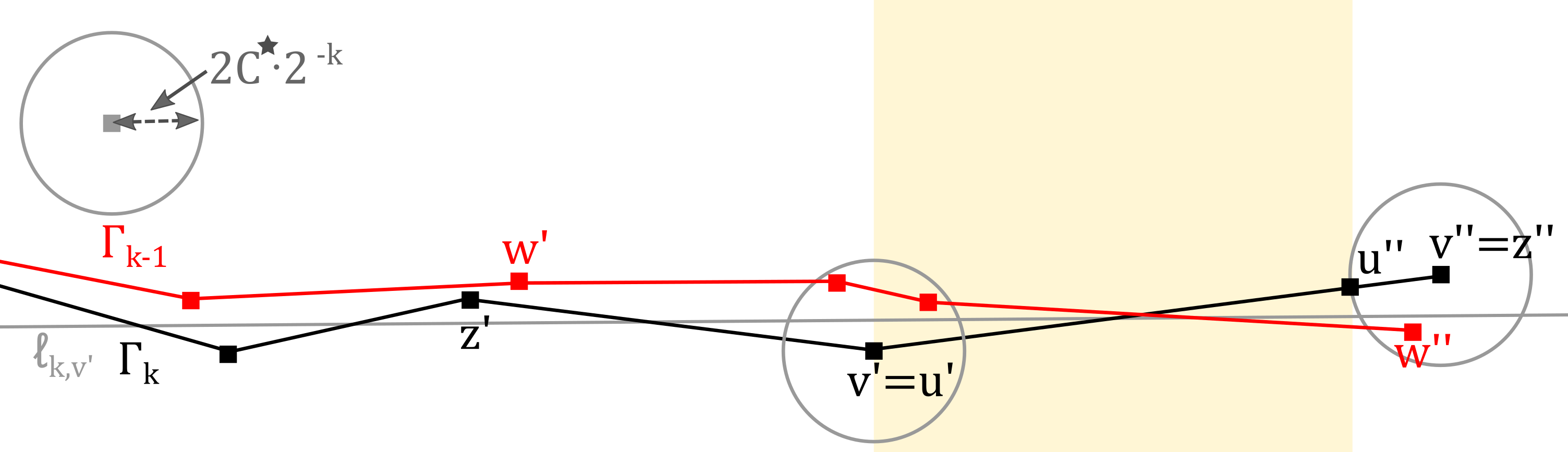}
\caption{The points $v',v'',u',u'',z',z'',w',w''$ in the Fourth Estimate, displayed with $v'$ non-terminal and $v''$ \textbf{Case II-T1} terminal to the right. The shaded yellow box represents $\pi_{k,v'}^{-1}([\pi_{k,v'}(u'),\pi_{k,v'}(u'')])$.}
\label{fig:4}
\end{figure}

All that remains is to estimate the length of overlaps of sets of the form $$S_{k,v'}[u',u'']:=E_{k-1}(v')\cap \pi_{k,v'}^{-1}([\pi_{k,v'}(u'),\pi_{k,v'}(u'')]).$$ Since $S_{k,v'}[u',u'']\subseteq S_{k,v'}[v',v'']$, it clearly suffices to estimate the length of the overlaps of sets $S_{k,v'}[v',v'']$.
Suppose that $v,v',v''$ are consecutive vertices in $V_k\cap B(v',65C^\star 2^{-k})$ such that $v$, $v'$, and $v''$ lie at distance at least $65C^\star 2^{-k}$ from \textbf{Case I} vertices and distance at least $2C^\star 2^{-k}$ from \textbf{Case II-T1} and \textbf{Case II-T2} terminal vertices. We will show that \begin{equation}\label{e:o1}\Haus^1(S_{k,v}[v,v']\cap S_{k,v'}[v',v'']) < 40\alpha^2 2^{-k},\quad\alpha=\max\{\alpha_{k,v},\alpha_{k,v'}\}.\end{equation} To start, let $\ell_1$ and $\ell_2$ be lines chosen so that $\ell_1$ is parallel to $\ell_{k,v}$, $\ell_2$ is parallel to $\ell_{k,v'}$, and $\ell_1$ and $\ell_2$ pass through $v'$. Note that, by the triangle inequality, \begin{equation}\label{e:2alpha}\dist(x,\ell_i) \leq 2\alpha 2^{-k}\quad\text{for all $x\in (V_{k-1}\cup V_k)\cap B(v',65 C^\star 2^{-k})$ and $i\in\{1,2\}$}.\end{equation} Let $\pi_i$ denote the orthogonal projection onto $\ell_i$ and let $N_i$ denote the closed tubular neighborhood of $\ell_i$ of radius $2\alpha 2^{-k}$. Also let $E_{k-1}(v,v'):=E_{k-1}(v)\cap E_{k-1}(v')$. By \eqref{e:2alpha}, \begin{equation}\begin{split} \label{e:o2} S_{k,v}[v,v'] &\cap S_{k,v'}[v',v'']\\ &\subseteq E_{k-1}(v,v')\cap \pi_1^{-1}([\pi_1(v),\pi_1(v')])\cap N_1\cap \pi_2^{-1}([\pi_2(v'),\pi_2(v'')])\cap N_2 \\ &=: E_{k-1}(v,v')\cap S.\end{split}\end{equation}  Since $2\alpha \leq 1/16$, Lemma \ref{l:graph} implies that \begin{equation}\label{e:o3}\Haus^1(E_{k-1}(v,v')\cap S)< (1+3(2\alpha)^2)\Haus^1(\pi_1(S))< 2\Haus^1(\pi_1(S)).\end{equation} Let $\theta$ be the angle subtended by $\ell_1$ and $\ell_2$ (see Figure \ref{fig:3}). By Lemma \ref{l:graph} \eqref{e:graph2}, we have $$\cos(\theta) > \frac{1}{1+12(2\alpha)^2}\geq   1-48\alpha^2>0.$$ Hence $1-\sin^2(\theta)=\cos^2(\theta) > 1-96\alpha^2.$ In particular, we obtain $\sin(\theta) < \sqrt{96}\alpha<10\alpha.$  Thus, by elementary geometry, \begin{equation}\label{e:o4}\Haus^1(\pi_1(S))\leq 2\alpha 2^{-k} \sin(\theta)\leq 20 \alpha^2 2^{-k}.\end{equation} Combining \eqref{e:o2}, \eqref{e:o3}, and \eqref{e:o4} establishes \eqref{e:o1}.

Carefully tallying the four estimates above, one obtains \eqref{e:sum-b}.

\begin{figure}
\includegraphics[width=.7\textwidth]{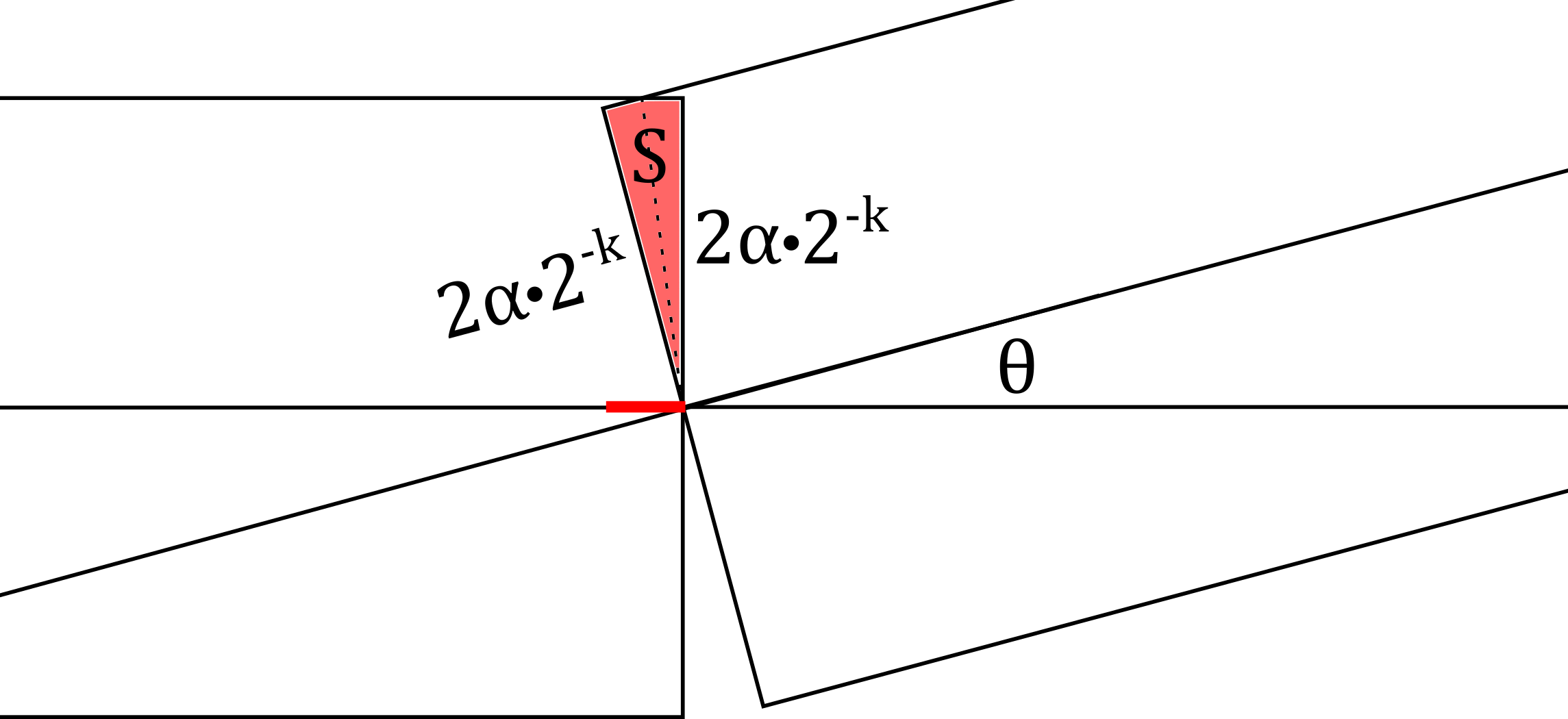}
\caption{Inside the 2-plane containing the lines $\ell_1$ and $\ell_2$, the diamond-shaped region $S$ is the union of two congruent triangles. The length of the red shadow $\pi_1(S)$ is no greater than $20 \alpha^2 2^{-k}$.}
\label{fig:3}
\end{figure}

\appendix
\section{Proof of Lemmas \ref{l:V} and \ref{l:graph}}
\label{ss:lemmas}

The proof of Lemma \ref{l:V} uses elementary properties of excess and Hausdorff distance; for a comprehensive reference, we recommend the monograph \cite{Beer} by Beer.

Recall that for nonempty sets $S,T\subseteq\RR^n$, the \emph{excess $\excess(S,T)$ of $S$ over $T$} is defined by $$\excess(S,T):= \sup_{s\in S}\inf_{t\in T} |s-t|$$ and the \emph{Hausdorff distance $\HD(S,T)$ between $S$ and $T$} is defined by $$\HD(S,T):=\max\{\excess(S,T),\excess(T,S)\}.$$ The excess satisfies the triangle inequality in the sense that $\excess(S,T)\leq \excess(S,U)+\excess(U,T)$ for all nonempty $S,T,U\subseteq\RR^n$. The set of nonempty compact subsets of $\RR^n$ equipped with the Hausdorff distance is a metric space. Thus, when restricted to the nonempty compact sets, we may refer to the Hausdorff distance as the \emph{Hausdorff metric}.

\begin{proof}[Proof of Lemma \ref{l:V}] Let $n\geq 1$, $C^\star>1$, and $r_0>0$. Assume that $V_0,V_1,V_2,\dots$ is a sequence of nonempty finite subsets of a bounded set $B$ satisfying condition ($V_\three$) of Proposition \ref{p:curve2}. Because $\overline{B}$ is compact and $(V_k)_{k=0}^\infty$ is a sequence of closed subsets of $\overline{B}$, there exist a subsequence $(V_{k_j})_{j=0}^\infty$ and a closed set $V\subseteq \overline{B}$ such that $V_{k_j}$ converges to $V$ in the Hausdorff metric as $j\rightarrow\infty$ by Blaschke's selection theorem (e.g.~see \cite[p.~90]{Rogers} or \cite[\S3.2]{Beer}). By iterating ($V_\three$), we obtain that for any $k_j< k <k_{j+1}$, \begin{align*} \excess(V_k,V_{k_j}) &\leq C^\star(2^{-(k_j+1)}+\dots+2^{-k})r_0< C^\star2^{-k_j}r_0 \quad \text{and}\\
\excess(V_{k_{j+1}},V_k) &\leq C^\star(2^{-(k+1)}+\dots+2^{-k_{j+1}})r_0<C^\star2^{-k}r_0.\end{align*} Thus, by the triangle inequality for excess, \begin{align*} \excess(V_k,V) &\leq \excess(V_k,V_{k_j})+\excess(V_{k_j},V)< C^\star 2^{-{k_j}}r_0+\HD(V_{k_j},V)\quad\text{and}\\
\excess(V,V_k) &\leq \excess(V,V_{k_{j+1}})+\excess(V_{k_{j+1}},V_k) < \HD(V,V_{k_{j+1}})+C^\star2^{-k}r_0. \end{align*} Therefore, $$\HD(V_k,V) < C^\star 2^{-k_j}r_0+\max\left\{ \HD(V_{k_j},V),\HD(V,V_{k_{j+1}})\right\}$$ whenever $k_j<k<k_{j+1}$. We conclude that the whole sequence $V_k$ converges to $V$ in the Hausdorff metric as $k\rightarrow\infty$.
\end{proof}

The proof of Lemma \ref{l:graph} that we give uses the area formula for Lipschitz graphs; for a nice presentation, see \S3.3 of the book \cite{EG} by Evans and Gariepy.

\begin{proof}[Proof of Lemma \ref{l:graph}] Let $V\subseteq\RR^n$ be a 1-separated set with at least two points. Assume that there exist straight lines $\ell_1$ and $\ell_2$ in $\RR^n$ and a number $0\leq \alpha \leq 1/16$ such that $$\dist(v,\ell_i)\leq \alpha\quad\text{for all $v\in V$ and $i=1,2$}.$$ Let $\pi_i$ denote the orthogonal projection onto $\ell_i$. Let $\pi_i^\perp$ denote the orthogonal projection onto an orthogonal complement of $\ell_i$. For any distinct pair of points $v_1,v_2\in V$, $$1\leq |v_1-v_2|^2 = |\pi_{i}(v_1)-\pi_i(v_2)|^2 +|\pi_i^\perp(v_1)-\pi_i^\perp(v_1)|^2\leq|\pi_i(v_1)-\pi_i(v_2)|^2 + 4\alpha^2,$$ because $V$ is 1-separated and the distance of points in $V$ to $\ell_i$ is bounded by $\alpha$. Hence \begin{equation}\label{e:graph3}|\pi_i(v_1)-\pi_i(v_2)|^2\geq |v_1-v_2|^2-4\alpha^2\geq 1-4\alpha^2.\end{equation} It follows that $V$ belongs to the graph of a piecewise linear function $g_i:\ell_i\rightarrow\ell_i^\perp$ such that $$\|\nabla g_i\|_{\infty} \leq \frac{2\alpha}{\sqrt{1-4\alpha^2}}.$$ By the area formula for Lipschitz maps, for any line segment $[u_1,u_2]$ in the graph of $g_i$, $$\Haus^1([u_1,u_2]) = \int_{[\pi_i(u_1),\pi_i(u_2)]} \sqrt{1+|\nabla g_i|^2}\,d\Haus^1.$$ Since $\sqrt{1+x^2}\leq 1+\frac12 x^2$ for all $x\in\RR$, we conclude that $$\Haus^1([u_1,u_2])\leq \int_{[\pi_i(u_1),\pi_i(u_2)]} \left(1+\tfrac{1}{2}|\nabla g|^2\right)\,d\Haus^1\leq \left(1+\frac{2\alpha^2}{1-4\alpha^2}\right)\Haus^1([\pi_i(u_1),\pi_i(u_2)])$$ for all $[u_1,u_2]\subseteq[v_1,v_2]$.

Because $\dist(v,\ell_i)\leq \alpha$ for all $v\in V$ and $i=1,2$, we can find points $z_1,z_2\in\ell_2$ such that $|z_1-z_2|\geq 1-2\alpha$ and $\dist(z_i,\ell_1)\leq 2\alpha$ for $i=1,2$. Thus, by analogous computation with $\beta=2\alpha/(1-2\alpha)$ in place of $\alpha$, $$\Haus^1([y_1,y_2])\leq \left(1+\frac{2\beta^2}{1-4\beta^2}\right)\Haus^1([\pi_1(y_1),\pi_1(y_2)])$$ for all $[y_1,y_2]\subseteq\ell_2$.

Now, since $0\leq \alpha\leq 1/16$, $$\frac{2\alpha^2}{1-4\alpha^2} \leq \frac{128}{63}\alpha^2<3\alpha^2,
\quad \beta=\frac{2\alpha}{1-2\alpha}\leq \frac{16}{7}\alpha\leq \frac{1}{7},\quad\text{and}$$ $$\frac{2\beta^2}{1-4\beta^2}\leq \frac{98}{45}\beta^2\leq \frac{512}{45}\alpha^2<12\alpha^2.$$
This establishes \eqref{e:graph1} and \eqref{e:graph2}.

Finally, suppose that there exist identifications of $\ell_1$ and $\ell_2$ with $\RR$ and distinct points $v,v',v''\in V$ such that $\pi_1(v)<\pi_1(v')< \pi_1(v'')$, but $\pi_2(v')<\pi_2(v)<\pi_2(v'')$. By \eqref{e:graph3}, for any distinct $w,w'\in V$ and $i\in\{1,2\}$ such that $\pi_i(w)<\pi_i(w')$, we have \begin{equation}\label{e:graph4} |w'-w| \geq \pi_i(w')-\pi_i(w) \geq \sqrt{|w'-w|^2 - 4\alpha^2} \geq |w'-w|-2\alpha >0.875,\end{equation} where the last two inequalities hold since $b \leq 2a+\sqrt{b^2-4a^2}$ for all $0\leq 2a\leq b$, $0\leq 2\alpha \leq 1/8$, and $1 \leq  |w'-w|$. Set $$x:=|v-v'|,\quad y:=|v''-v'|,\quad\text{and}\quad z:=|v''-v|.$$ On one hand, since $\pi_1(v)<\pi_1(v')<\pi_1(v'')$, we have $z\approx x+y$. On the other hand, since $\pi_2(v')<\pi_2(v)<\pi_2(v'')$, we have $y \approx x+z$. Hence $z\approx z+2x$, which yields a contradiction if $\alpha$ is sufficiently small. More precisely, by repeated application of \eqref{e:graph4}, $$z\geq x+y - 4\alpha \geq 2x+z - 8\alpha >1.75+z-8\alpha.$$ Hence $1.75< 8\alpha \leq 1/2$ (since $\alpha\leq 1/16$), which is absurd. We conclude that under any choice of identifications of $\ell_1$ and $\ell_2$ with $\RR$, either $\pi_1(v)\leq \pi_1(v')$ if and only if $\pi_2(v)\leq \pi_2(v')$ for all $v,v'\in V$, or $\pi_1(v)\leq \pi_1(v')$ if and only if $\pi_2(v)\geq \pi_2(v')$ for all $v,v'\in V$. In particular, one can always choose compatible identifications of $\ell_1$ and $\ell_2$ with $\RR$ such that $\pi_1(v')\leq \pi_1(v'')$ if and only if $\pi_2(v')\leq \pi_2(v'')$ for all $v',v''\in V$.
\end{proof}

\bibliography{spd}{}
\bibliographystyle{amsbeta}

\end{document}